\documentclass[reqno]{amsart}

\usepackage[all]{xy}
\usepackage{graphicx}

\usepackage{amssymb}
\usepackage{amsmath}
\usepackage{mathrsfs}
\usepackage{epsfig}
\usepackage{amscd}
\usepackage{graphicx, color}

\usepackage{graphicx}
\usepackage{amssymb}
\usepackage{amsmath}
\usepackage{enumerate}% http://ctan.org/pkg/enumerate
\usepackage{amsfonts}
\usepackage[mathcal]{euscript}
\usepackage{amsthm}
\usepackage{xcolor}
\usepackage{amscd}
\usepackage{mathrsfs}
\numberwithin{equation}{section}

\usepackage[colorlinks=true, pdfstartview=FitV, linkcolor=black,
			   citecolor=black, urlcolor=black]{hyperref}
\usepackage{bookmark}
\usepackage[titletoc]{appendix} %For some reason, putting the bookmark and appendix packages after hyperref gets rid of the \endcsname error.
\hypersetup{hypertexnames=false} % This prevents the "same identifier" warning when compiling.

\theoremstyle{theorem}
\newtheorem{thm}{Theorem}[section]
\newtheorem{cor}[thm]{Corollary}
\newtheorem{lem}[thm]{Lemma}
\newtheorem{prop}[thm]{Proposition}

\theoremstyle{definition}
\newtheorem{defn}[thm]{Definition}
\newtheorem{rem}[thm]{Remark}
\newtheorem{hypo}[thm]{Hypothesis}

\newtheorem{notation}[thm]{\rm\bfseries{Notation}}

\newtheorem{exm}[thm]{Example}
\newtheorem{situ}[thm]{Situation}
\newtheorem{choice}[thm]{Choice}

\newtheorem{upshot}[thm]{Upshot}

\numberwithin{equation}{section}

          \newcommand{\nc}{\newcommand}
          \nc{\DMO}{\DeclareMathOperator}	
          \nc{\commentout}[1]{}

          \nc{\newnotation}{\nomenclature}
          \nc{\wrap}{\cW}
          \nc{\Tw}{\mathsf{Tw}}
          \nc{\loc}{\mathsf{Loc}}
          \nc{\Top}{Top}
          \nc{\emb}{\mathsf{emb}}
          \nc{\ind}{\mathsf{Ind}}
          \nc{\Ind}{\mathsf{Ind}}
          \nc{\Loc}{\mathsf{Loc}}
          \nc{\Cob}{\mathsf{Cob}}
          \nc{\mul}{\mathsf{Mul}}
          \nc{\fat}{\mathsf{fat}}
          \nc{\cob}{\mathsf{Cob}}
          \nc{\coh}{\mathsf{Coh}}
          \nc{\Liouaut}{\Aut_{\mathsf{Liou}}}
          \nc{\idem}{\mathsf{Idem}}
          \nc{\sets}{\mathsf{Sets}}
          \nc{\near}{\mathsf{near}}
          \nc{\sing}{\mathsf{Sing}}
          \nc{\Sing}{\mathsf{Sing}}
          \nc{\perf}{\mathsf{Perf}}
          \nc{\block}{\mathsf{block}}
          \nc{\ssets}{\mathsf{sSets}}
          \nc{\cmpct}{\mathsf{cmpct}}
          \nc{\compact}{\mathsf{cmpct}}
          \nc{\pwrap}{\mathsf{PWrap}}
          \nc{\coder}{\mathsf{Coder}}
          \nc{\bimod}{\mathsf{Bimod}}
          \nc{\grmod}{\mathsf{GrMod}}
          \nc{\Morita}{\mathsf{Morita}}
          \nc{\morita}{\mathsf{Morita}}
          \nc{\spaces}{\mathsf{Spaces}}
          \nc{\pwrms}{\mathsf{PWrFuk}_{M,S}}
          \nc{\pwrmf}{\mathsf{PWrFuk}_{M,F}}
          \nc{\pwrapmf}{\mathsf{PWrFuk}_{M,F}}
          \nc{\fuk}{\mathsf{Fukaya}}
          \nc{\infwr}{\mathsf{InfWr}}
          \nc{\fukaya}{\mathsf{Fukaya}}
          \nc{\autml}{\mathsf{Aut}_{M,\Lambda}}
          \nc{\fukml}{\mathsf{Fukaya}_{M,\Lambda}}
          \nc{\fukmle}{\mathsf{Fukaya}_{M,\Lambda,\epsilon}}
          \nc{\fukmod}{\wrfukcompact(M)\modules}
          \nc{\lag}{\mathsf{Lag}}
          \nc{\lagm}{\lag_M}
          \nc{\lago}{\lag^o}
          \nc{\lagml}{\lag_{M,\Lambda}} % For when I get lazy.
          \nc{\lagmle}{\lag_{M,\Lambda,\epsilon}}
          \nc{\Fun}{\mathsf{Fun}}
          \nc{\fun}{\mathsf{Fun}}
          \nc{\vect}{\mathsf{Vect}}
          \nc{\chain}{\mathsf{Chain}}
          \nc{\chainn}{Chain}
          \nc{\wrfuk}{\mathsf{WrFukaya}}
          \nc{\wrfukcompact}{\mathsf{WrFukaya}_{\mathsf{cmpct}}}
          \nc{\pwrfuk}{\mathsf{PWrFukaya}}
          \nc{\inffuk}{\mathsf{InfFuk}}
          \nc{\pwrfukml}{\mathsf{PWrFukaya}_{M,\Lambda}}
          \nc{\inffukml}{\mathsf{InfFuk}_{M,\Lambda}}
          \nc{\nattrans}{\mathsf{NatTrans}}
          \nc{\corres}{\mathsf{Corres}}
          \nc{\fukep}{\fukaya_\Lambda(M,\epsilon)}
          \nc{\fukepop}{\fukaya_\Lambda(M,\epsilon)^{\op}}
          \nc{\lagep}{\lag_\Lambda(M,\epsilon)}
          \DMO{\cyl}{cyl} % Cylindrical
          \nc{\dbcoh}{D^b\mathsf{Coh}}
          \nc{\corr}{\mathsf{Corr}}
          \nc{\Liouauto}{{\Aut^o}}
          \nc{\Liouautb}{\Aut^{b}}
          \nc{\Liouautgr}{{\Aut^{gr}}}
          \nc{\Liouautgrb}{\Aut^{gr,b}}
          \nc{\Fuk}{\mathsf{Fuk}}

          \DMO{\im}{im}
          \DMO{\ev}{ev}
          \DMO{\stable}{Ex}
          \DMO{\inj}{inj}
          \DMO{\fib}{fib}
          \DMO{\conf}{Conf}
          \DMO{\chains}{Chains}
          \DMO{\cochains}{Cochains}
          \DMO{\cone}{Cone}
          \DMO{\Map}{Map}
          \DMO{\ran}{Ran}
          \DMO{\rot}{Rot}
          \DMO{\leg}{Leg}
          \DMO{\imm}{imm}
          \DMO{\adj}{adj}
          \DMO{\symp}{Symp}
          \DMO{\tree}{Tree}
          \DMO{\cube}{Cube}
          \DMO{\deep}{deep}
          \DMO{\back}{back}
          \DMO{\Hoch}{Hoch}
          \DMO{\front}{front}
          \DMO{\flow}{Flow}
          \DMO{\floer}{Floer}
          \DMO{\Maps}{Maps}
          \DMO{\exact}{exact}
          \DMO{\excess}{Excess}
          \DMO{\Decomp}{Decomp}
          \DMO{\decomp}{Decomp}
          \DMO{\collar}{collar}
          \DMO{\yoneda}{Yoneda}
          \DMO{\hamspace}{Ham}
          \DMO{\sympspace}{Symp}
          \DMO{\holomaps}{Holomaps}
          \DMO{\comp}{Comp}
          \DMO{\crit}{Crit}
          \DMO{\test}{{test}}
          \DMO{\sign}{sign}
          \DMO{\topp}{top}
          \DMO{\indx}{Index}
          \DMO{\Break}{Break} % Partitions
          \DMO{\zero}{zero} %Zero
          \DMO{\ob}{Ob}
          \DMO{\gr}{Gr} % Grassmanian
          \DMO{\Gr}{Gr} % Grassmanian
          \DMO{\cl}{Cl} % Clifford Algebra
          \DMO{\grlag}{GrLag}
          \DMO{\Pin}{Pin}
          \DMO{\Graph}{Graph}
          \DMO{\pin}{Pin}
          \DMO{\gap}{Gap}
          \DMO{\Ex}{Ex}
          \DMO{\id}{id}
          \DMO{\End}{End}
          \DMO{\sym}{Sym}
          \DMO{\aut}{aut}
          \DMO{\Aut}{Aut}
          \DMO{\haut}{hAut}
          \DMO{\hAut}{hAut}
          \DMO{\DK}{DK} %Dold-Kan
          \DMO{\poly}{poly} % Polynomial deRham forms
          \DMO{\diff}{Diff}
          \DMO{\coll}{coll}
          \DMO{\dist}{dist} %Distance function
          \DMO{\coker}{coker} %Cokernel
          \nc{\kernel}{\ker} %Kernel
          \DMO{\sspan}{span}
          \DMO{\hocolim}{hocolim}	
          \DMO{\holim}{holim}
          \DMO{\sk}{sk}

          \DMO{\ho}{ho}
          \DMO{\fin}{fin}
          \DMO{\tor}{Tor}
          \DMO{\ext}{Ext}
          \DMO{\ret}{Ret}
          \DMO{\ham}{Ham}
          \DMO{\con}{con}
          \DMO{\leaf}{leaf}
          \DMO{\supp}{supp}
          \DMO{\edge}{edge}
          \DMO{\colim}{colim}
          \DMO{\edges}{edges}
          \DMO{\Image}{image}
          \DMO{\roots}{roots}
          \DMO{\height}{height}
          \DMO{\finmod}{FinMod}
          \DMO{\leaves}{leaves}
          \DMO{\planar}{planar}
          \DMO{\vertices}{vertices}

\nc{\norm}[2]{{ \ensuremath{\|} #1 \ensuremath{\|}}_{#2}}
\nc{\Dbar}[1]{\ensuremath{{\bar{\partial}}_{#1}}}
\nc{\Ce}{\ensuremath{\mathbb{C}}}
\nc{\B}{\ensuremath{\mathbb{B}}}
\nc{\osc}{\operatorname{osc}}
\nc{\leng}{\operatorname{leng}}
          \nc{\lagg}{\lag^{\cG}}
          \nc{\iso}{\mathsf{Iso}}
          \nc{\Set}{\mathsf{Set}}
          \nc{\Ass}{\mathsf{ \bf Ass}}
          \nc{\Mod}{\mathsf{Mod}}
          \nc{\modules}{\mathsf{Mod}}
          \nc{\sset}{\mathsf{sSet}}
          \nc{\liou}{\mathsf{Liou}}
          \nc{\poset}{\mathsf{Poset}}
          \nc{\trno}{T^*\RR^n_{\geq 0}}
          \nc{\spectra}{\mathsf{Spectra}}
          \nc{\tensorfin}{\tensor^{\fin}}
          \nc{\lagptg}{\lag_{pt,pt}^{\cG}}
          \nc{\Fin}{\mathcal{F}\mathsf{in}}
          \nc{\lagnl}{\lag_{N,\Lambda}}
          \nc{\lagmlg}{\lag_{M,\Lambda}^{\cG}}
          \nc{\lagsplit}{\lag^{\mathsf{split}}}
          \nc{\lagktimes}{(\lag^{\dd k})^\times}
          \nc{\lagplanar}{\lag^{\times,\planar}}
          \nc{\Cont}{\text{\rm Cont}}
          \nc{\Ham}{\text{\rm Ham}}
          \nc{\Dev}{\text{\rm Dev}}
          \nc{\Lin}{\text{\rm Lin}}
          \nc{\Int}{\text{\rm Int}}
          \nc{\Hom}{\text{\rm Hom}}
          \nc{\Chord}{\text{\rm Chord}}
          \nc{\nbhd}{\mathcal{N}\text{\rm{bhd}}}

          \nc{\onef}{1_{\fukaya}}
          \nc{\smsh}{\wedge}
          \nc{\un}{\underline}
          \nc{\xto}{\xrightarrow}
          \nc{\xra}{\xto}
          \nc{\tensor}{\otimes}
          \nc{\del}{\partial}
          \nc{\dd}{\diamond}
          \nc{\tri}{\triangle}
          \nc{\bb}{\Box}
          \nc{\into}{\hookrightarrow}
          \nc{\onto}{\twoheadrightarrow}
          \nc{\contains}{\supset}
          \nc{\transverse}{\pitchfork}
          \nc{\uncirc}{\underline{\circ}}
          \nc{\thetacontact}{\theta} % For the contact 1-form

          \nc{\Jbar}{\overline{J}}
          \nc{\Fbar}{\overline{F}}
          \nc{\delbar}{\overline{\del}}
          \nc{\thetabar}{\overline{\theta}}
          \nc{\omegabar}{\overline{\omega}}
          \nc{\Liou}{\text{\rm Liou}}

          \nc{\Yhat}{\widehat{Y}}

          \nc{\Mliou}{M}

          \nc{\vece}{ {\vec \epsilon}}	
          \nc{\vecd}{ {\vec \delta}}
          \nc{\ov}{\overline}
          \DMO{\op}{op}
          \nc{\opp}{ ^{\op}}

          \nc{\hiro}{\textcolor{blue}}
          \nc{\YG}{\textcolor{orange}}
          \nc{\beastar}{\begin{eqnarray*}}
          \nc{\eeastar}{\end{eqnarray*}}

\numberwithin{equation}{section}

\def\R{{\mathbb R}}
\def\osc{{\hbox{\rm osc }}}
\def\Crit{{\hbox{Crit}}}

\def\E{{\mathbb E}}
\def\Z{{\mathbb Z}}
\def\C{{\mathbb C}}
\def\R{{\mathbb R}}
\def\P{{\mathbb P}}

\def\N{{\mathbb N}}

\def\11{{\mathbb I}}

\def\Jbar{{\widetilde J}}

\def\delbar{{\overline \partial}}

\def\C{\mathbb{C}}
\def\Z{\mathbb{Z}}

\def\Q{\mathbb{Q}}

\def\E{\ifmmode{\mathbb E}\else{$\mathbb E$}\fi} %natural numbers
\def\N{\ifmmode{\mathbb N}\else{$\mathbb N$}\fi} %natural numbers
\def\R{\ifmmode{\mathbb R}\else{$\mathbb R$}\fi} %real numbers
\def\Q{\ifmmode{\mathbb Q}\else{$\mathbb Q$}\fi} %rational numbers
\def\C{\ifmmode{\mathbb C}\else{$\mathbb C$}\fi} %complex numbers
\def\H{\ifmmode{\mathbb H}\else{$\mathbb H$}\fi} %complex numbers
\def\Z{\ifmmode{\mathbb Z}\else{$\mathbb Z$}\fi} %integers
\def\P{\ifmmode{\mathbb P}\else{$\mathbb P$}\fi} %real numbers
\def\SS{\ifmmode{\mathbb S}\else{$\mathbb S$}\fi} %real numbers
\def\DD{\ifmmode{\mathbb D}\else{$\mathbb D$}\fi} %real numbers

\def\R{{\mathbb R}}
\def\osc{{\hbox{\rm osc}}}
\def\Crit{{\hbox{Crit}}}
\def\E{{\mathbb E}}
\def\Z{{\mathbb Z}}
\def\C{{\mathbb C}}
\def\R{{\mathbb R}}

\def\N{{\mathbb N}}

\def\Jbar{{\widetilde J}}

\def\delbar{{\overline \partial}}

\def\p{\psi}
\def\CA{{\mathcal A}}

\def\CC{{\mathcal C}}

\def\CJ{{\mathcal J}}

\def\CL{{\mathcal L}}
\def\CM{{\mathcal M}}

\def\CP{{\mathcal P}}

\def\CR{{\mathcal R}}
\def\CP{{\mathcal P}}

\def\darr#1{\raise1.5ex\hbox{$\leftrightarrow$}
\mkern-16.5mu #1}

\def\roughly#1{\raise.3ex\hbox{$#1$\kern-.75em
\lower1ex\hbox{$\sim$}}}

\def\opname#1{\mathop{\kern0pt{\rm #1}}\nolimits}

\def\Im{\opname{Im}}
\def\End{\opname{End}}
\def\dim{\opname{dim}}

\def\dist{\opname{dist}}

\def\rank{\opname{rank}}

\def\supp{\operatorname{supp}}

\def\Dev{\operatorname{Dev}}

\def\leng{\operatorname{leng}}

\def\End{\operatorname{End}}

\def\Aut{\operatorname{Aut}}
\def\coker{\operatorname{Coker}}

\def\Cont{\operatorname{Cont}}
\def\Crit{\operatorname{Crit}}
\def\Spec{\operatorname{Spec}}
\def\Sing{\operatorname{Sing}}
\def\GFQI{\frak{G}}
\def\Index{\operatorname{Index}}

\def\Image{\operatorname{Image}}
\def\ev{\operatorname{ev}}
\def\Int{\operatorname{Int}}

\def\ben{\begin{enumerate}}
\def\een{\end{enumerate}}
\def\be{\begin{equation}}
\def\ee{\end{equation}}
\def\bea{\begin{eqnarray}}
\def\eea{\end{eqnarray}}
\def\beastar{\begin{eqnarray*}}
\def\eeastar{\end{eqnarray*}}
\def\bc{\begin{center}}
\def\ec{\end{center}}

\begin{document}

\quad \vskip1.375truein

\def\mq{\mathfrak{q}}
\def\mp{\mathfrak{p}}
\def\mH{\mathfrak{H}}
\def\mh{\mathfrak{h}}
\def\ma{\mathfrak{a}}
\def\ms{\mathfrak{s}}
\def\mm{\mathfrak{m}}
\def\mn{\mathfrak{n}}
\def\mz{\mathfrak{z}}
\def\mw{\mathfrak{w}}
\def\Hoch{{\tt Hoch}}
\def\mt{\mathfrak{t}}
\def\ml{\mathfrak{l}}
\def\mT{\mathfrak{T}}
\def\mL{\mathfrak{L}}
\def\mg{\mathfrak{g}}
\def\md{\mathfrak{d}}
\def\mr{\mathfrak{r}}
\def\Cont{\operatorname{Cont}}
\def\Crit{\operatorname{Crit}}
\def\Spec{\operatorname{Spec}}
\def\Sing{\operatorname{Sing}}
\def\GFQI{\text{\rm g.f.q.i.}}
\def\Index{\operatorname{Index}}
\def\Cross{\operatorname{Cross}}
\def\Ham{\operatorname{Ham}}
\def\Fix{\operatorname{Fix}}
\def\Graph{\operatorname{Graph}}
\def\id{\text\rm{id}}

\title[Contact instantons and entanglement]
{Geometry and analysis of contact instantons and entanglement of Legendrian links}

\author{Yong-Geun Oh}
\address{Center for Geometry and Physics, Institute for Basic Science (IBS),
77 Cheongam-ro, Nam-gu, Pohang-si, Gyeongsangbuk-do, Korea 790-784
\& POSTECH, Gyeongsangbuk-do, Korea}
\email{yongoh1@postech.ac.kr}
\thanks{This work is supported by the IBS project \# IBS-R003-D1}

%\date{October, 2021}

\begin{abstract} The purposes of the present paper are two-fold. Firstly
 we further develop the interplay between
the contact Hamiltonian geometry and the geometric analysis of the new analytical machinery of 
Hamiltonian-perturbed contact instantons with the Legendrian boundary condition
in the study of contact dynamics and topology.
We introduce the class of \emph{tame contact manifolds} $(M,\lambda)$, which includes compact
ones but not necessarily compact, and establish uniform a priori $C^0$-estimates for
the contact instantons. Then we study the problem of estimating the Reeb-untangling energy of one Legendrian submanifold from
another, and formulate a particularly designed parameterized moduli space
for the study of the problem. We establish the Gromov-Floer-Hofer type
convergence results of contact instantons of finite energy
and construct its compactification of the moduli space.
We do this  first by defining the correct energy
and then by carrying out bubbling-off analysis
and proving uniform a priori energy bounds
in terms of the specific \emph{curvature-free family} of
 contact Hamiltonians.
Secondly, as an application of this geometry and analysis
of contact instantons, we prove that the
\emph{self Reeb-untangling energy} of
a compact Legendrian submanifold $R$ in any tame contact manifold
$(M,\lambda)$ is greater than that of the period gap $T_\lambda(M,R)$
of  the Reeb chords of $R$. This is an optimal result in general.
In a sequel \cite{oh:shelukhin-conjecture}, we also prove Shelukhin's conjecture
specializing to the Legendrianization of contactomorphisms of closed
coorientable contact manifold $(Q,\xi)$ and utilizing
its $\Z_2$-symmetry as the fixed point set of anti-contact involution to overcome
the \emph{nontameness} of contact product $M = Q \times Q \times \R$.
\end{abstract}


\maketitle

\tableofcontents

\section{Introduction}

The general Lagrangian Floer theory in symplectic geometry concerns \emph{intersections} of
Lagrangian submanifolds which largely relies on the study of the moduli spaces of
solutions of Hamiltonian-perturbed pseudoholomorphic curves
under the Lagrangian boundary condition with \emph{finite energy} (and of \emph{bounded image}
in addition when the ambient space is noncompact).

In this paper and its sequels, we develop a contact analog to such a theory for the study of contact
dynamics which concerns \emph{entanglement} of Legendrian links which relies on
the study of the moduli spaces of Hamiltonian-perturbed \emph{contact instantons
under the Legendrian boundary condition}. This nonlinear elliptic boundary value problem
was introduced by the present author in \cite{oh:contacton-Legendrian-bdy}
which is based on the analytic study of contact instantons provided in
\cite{oh-wang:CR-map1,oh-wang:CR-map2,oh:contacton-Legendrian-bdy}. More importantly in ibid, we proved
a fundamental  vanishing result of the \emph{asymptotic charge} for finite $\pi$-energy
contact instantons \emph{with Legendrian boundary conditions}, which eliminates the phenomenon of the
occurrence of \emph{spiraling cusp instantons along a Reeb core}. This enables us to
carry out the Gromov-Floer-Hofer style compactification and the relevant Fredholm theory of
the associated moduli space.
The present paper and its sequels \cite{oh:shelukhin-conjecture}, \cite{oh:entanglement2}, 
\cite{oh-yso:spectral}
are based on this geometric analysis of perturbed contact instantons and
the relevant contact geometry and Hamiltonian calculus.

\begin{rem}
The equation (with $H=0$) itself, what we call the contact instanton equation,
was first introduced by Hofer \cite[p.698]{hofer:gafa}
and utilized in \cite{abbas-cieliebak-hofer}, \cite{abbas} in their attempts to attack
Weinstein's conjecture in three dimensions.
\end{rem}

Along the way, we also develop the analytic machinery of Hamiltonian-perturbed contact instantons and
illustrate its application to a quantitative study of contact dynamics.
The heart of the matter lies in the quantitative study of the moduli space of
perturbed contact instantons and its interplay
with the contact Hamiltonian geometry and calculus. Leaving a full systematic study of  perturbed contact instantons and
the aforementioned quantitative entanglement study of general Legendrian links
in future works, we utilize \emph{contact instantons with Legendrian boundary conditions}
as its probes, and prove an optimal result on the contact displacement problem of a compact Legendrian submanifold
as an application of a quantitative investigation of self-entanglement of a compact Legendrian
submanifold in the present paper. In a sequel \cite{oh:shelukhin-conjecture}, we specialize
to the \emph{Legendrianization} of contactomorphisms of general compact contact manifold and
prove Theorem \ref{thm:shelukhin-intro} by extending the current methodology 
to the non-tame case of contact product
$Q \times Q \times \R$ by exploiting an anti-contact involution
and utilizing $\Z_2$-symmetry argument by a careful choice of a contact form and other
Floer data so that they are compatible with the involutive symmetry.

The identification of correct energy and
the relevant compactification of the moduli space developed in
Part 3 of the present paper is also the analytic basis for the Floer theoretic construction of
Legendrian spectral invariants via contact instanton Floer-type homology given in \cite{oh-yso:spectral}
for the one-jet bundle, and other applications prepared in \cite{oh:entanglement2} and other sequels.
Our study provides a flexible analytic tool for general systematic quantitative study of
the contact dynamics and topology through the
geometric analysis of perturbed contact instantons.

\subsection{Legendrian isotopy and Reeb chords}

We will always consider cooriented contact manifolds $(M,\xi)$ which admits associated contact forms
$\lambda$ satisfying $\ker \lambda = \xi$. We denote by
\be\label{eq:CC-xi}
\mathcal C(\xi)= \mathcal C(M,\xi)
\ee
the set of contact forms $\lambda$ of $\xi$, i.e., of those satisfying
$\ker \lambda = \xi$ and $\lambda \wedge (d\lambda)^n$ is a volume form of $M = M^{2n+1}$.

Contact manifolds equipped with a contact form carry canonical background Reeb dynamics:
When $\lambda$ is a contact form of $\xi$, it uniquely defines a contact vector field $R_\lambda$
called the Reeb vector field by the defining condition
\be\label{eq:Rlambda}
R_\lambda \rfloor d\lambda = 0, \, R_\lambda \rfloor \lambda = 1.
\ee
Unlike the Lagrangian submanifolds in symplectic background, the generic characteristic of a
Legendrian link is its \emph{entanglement} structure relative to this background dynamics.

One of the purposes of the present paper is to make the first step towards
a systematic quantitative study of this entanglement structure. With this long-term goal
of investigation in our minds, we formulate the general problem in this introduction in the scope wider
than that of what we actually investigate in the present paper which mainly deals with
a two-component link of the type $(\psi(R),R)$ for an arbitrary contactomorphism $\psi$.

A Legendrian link is a finite disjoint union
$$
\mathsf R = \bigsqcup_{i=1}^\ell R_i
$$
of
connected Legendrian submanifolds $R_i$. We call each $R_i$ a component of the link $\mathsf R$.
\begin{defn}\label{defn:Reeb-chords-intro} Consider a co-oriented contact manifold $(M,\xi)$ and a Legendrian link $\mathsf R$.
Let $\lambda \in \mathcal C(\xi)$. A curve $\gamma:[0,T] \to M$ satisfying
$$
\begin{cases}
\dot \gamma = R_\lambda(\gamma(t)),\\
\gamma(0), \, \gamma(T) \in \mathsf R
\end{cases}
$$
is called a \emph{Reeb chord of $\mathsf R$}. We call $\gamma$ a \emph{self-chord} if
its initial and final points land at the same component, and a \emph{trans-chord}
otherwise. We denote by
$
\frak{Reeb}(M,\mathsf R;\lambda)
$
the set of $\lambda$-Reeb chords of $\mathsf R$.
\end{defn}

Here is the simplest form of nontrivial entanglement.

\begin{defn}\label{defn:dynamical-entanglement} Let $(M,\xi)$ be a contact manifold and
$\lambda$ a contact form. We say a Legendrian link $\mathsf R$ is
\emph{dynamically $\lambda$-entangled} if there exists a Reeb chord of $\mathsf R$, i.e.,
if  $\frak{Reeb}(M,\mathsf R;\lambda) \neq \emptyset$.
\end{defn}

We say a chord $\gamma$  \emph{primary} if
$\gamma((0,T))$ does not intersect $\mathsf R$.

\begin{defn} For a given Legendrian link $\mathsf R$, we define the \emph{$\lambda$-chord period gap}
(or simply a $\lambda$-chord gap) $T(M,\mathsf R;\lambda)$ of $\mathsf R$ is defined by
$$
T(M,\mathsf R;\lambda): = \inf\{T \mid \exists (T,\gamma) \in \frak{Reeb}(M,\mathsf R;\lambda), \,
\gamma \text{ is primary} \}.
$$
\end{defn}
Obviously if $\mathsf R$ is compact, $T(M,\mathsf R;\lambda) > 0$.

This definition is a direct generalization to the links of the following standard definition
$T(M,R;\lambda)$ below in contact geometry.

\begin{defn}\label{defn:spectrum} Let $\lambda$ be a contact form of contact manifold $(M,\xi)$ and $R \subset M$ a
connected Legendrian submanifold.
Denote by $\frak R eeb(M,\lambda)$ (resp. $\frak R eeb(M,R;\lambda)$) the set of closed Reeb orbits
(resp. the set of self Reeb chords of $R$).
\begin{enumerate}
\item
We define $\operatorname{Spec}(M,\lambda)$ to be the set
$$
\operatorname{Spec}(M,\lambda) = \left\{\int_\gamma \lambda \mid \lambda \in \frak Reeb(M,\lambda)\right\}
$$
and call the \emph{action spectrum} of $(M,\lambda)$.
\item We define the \emph{period gap} to be the constant given by
$$
T(M,\lambda): = \inf\left\{\int_\gamma \lambda \mid \lambda \in \frak Reeb(M,\lambda)\right\} > 0.
$$
\end{enumerate}
We define $\operatorname{Spec}(M,R;\lambda)$ and the associated $T(M,R;\lambda)$ similarly using the set
$\frak R eeb(M,R;\lambda)$ of Reeb chords of $R$.
\end{defn}
We set $T(M,\lambda) = \infty$ (resp. $T(M,R;\lambda) = \infty$) if there is no closed Reeb orbit (resp. no $(R_0,R_1)$-Reeb chord). Then we define
\be\label{eq:TMR}
T_\lambda(M;R): = \min\{T(M,\lambda), T(M,R;\lambda)\}
\ee
and call it the \emph{(chord) period gap} of $R$ in $M$.

In relation to our solution to the conjecture of Sandon and Shelukhin \cite{sandon-translated}, \cite{shelukhin:contactomorphism},
 we consider the pair $(R_0,R_1)$ of the type $R_0 = \psi(R), \, R_1 = R$
for a contactomorphism $\psi$ and prove a result in the simplest level of dynamical entanglement,
the existence of a Reeb chord of a two-component link $(\psi(R),R)$, by studying
the moduli space of contact instantons intertwining the link, and give its
application to the aforementioned conjecture.

We denote by
$$
{\mathcal Leg}(M,\xi)
$$
the set of Legendrian submanifold and by ${\mathcal Leg}(M,\xi;R)$ its connected component
containing $R \in {\mathcal Leg}(M,\xi)$, i.e, the set of Legendrian submanifolds Legendrian isotopic to
$R$. We denote by
$$
\CP({\mathcal Leg}(M,\xi))
$$
the monoid of Legendrian isotopies $[0,1] \to {\mathcal Leg}(M,\xi)$. We have
natural evaluation maps
$$
\ev_0, \, \ev_1:\CP({\mathcal Leg}(M,\xi)) \to {\mathcal Leg}(M,\xi)
$$
and denote by
$$
\CP({\mathcal Leg}(M,\xi), R) = \ev_0^{-1}(R) \subset \CP({\mathcal Leg}(M,\xi))
$$
and
$$
\CP({\mathcal Leg}(M,\xi), (R_0,R_1) = (\ev_0\times \ev_1)^{-1}(R_0,R_1)
\subset \CP({\mathcal Leg}(M,\xi)).
$$
\begin{notation}
Let $H= H(t,x)$ be a given contact Hamiltonian.
\begin{enumerate}
\item We denote by
$
\psi_H: t\mapsto \psi_H^t
$
the contact Hamiltonian path generated by $H$.
\item When $\psi \in \Cont(M,\xi)$, we denote by
$
H \mapsto \psi
$
when $\psi = \psi_H^1$.
\item
When $R_0 = R$ and $R_1 = \psi(R)$ for a contactomorphism contact isotopic to the identity,
the Hamiltonian path $\psi_H$ defines a natural Legendrian isotopy
$$
\CR_{H;R}:t \mapsto \psi_H^t(R).
$$
We denote by
$$
\CP({\mathcal Leg}(M,\xi), R;H) \subset \CP({\mathcal Leg}(M,\xi), (\psi(R),R))
$$
the set of Legendrian isotopies which is homotopic to the path $\CR_{H;R}$
relative to the ends.
\end{enumerate}
\end{notation}

\subsection{Tame contact manifolds and maximum principle}

Now we specify what kind of contact manifolds we will take as the background
geometry in the present paper. Obviously all compact ones will be included which however will
not be enough for the proof of Shelukhin's conjecture given in \cite{oh:shelukhin-conjecture}
which inevitably involve noncompact contact manifolds of the type $Q \times Q \times \R$.

Our introduction of the following class of noncompact contact manifolds is
largely motivated to make the relevant contact manifolds are amenable to the maximum
principle in the study of  contact instantons, especially to include the one-jet bundle and
\emph{contact product} $Q \times Q \times \R$ for $\{\eta > 0\}$ as a special case.
(See \cite{oh-yso:spectral} and \cite{oh:shelukhin-conjecture} respectively for their verification.)

We first introduce the following type of barrier functions.

\begin{defn}[Reeb-tame functions] Let $(M,\lambda)$ be a contact manifold.
A function $\psi: M \to \R$ is called \emph{$\lambda$-tame} on an open subset $U \subset M$ if $\CL_{R_\lambda} d\psi = 0$
on $M \setminus K$ on $U$.
\end{defn}

\begin{defn}[Contact {$J$} quasi-pseudoconvexity]\label{defn:J-convexity}
Let $J$ be a $\lambda$-adapted CR almost complex structure. We call a nonnegative function 
$\psi: M \to \R$
\emph{contact $J$ quasi-plurisubharmonic} on $U$ if it satisfies
\bea
-d(d \psi \circ J) +  k d\psi \wedge \lambda & \geq & 0 \quad \text{\rm on $\xi$}, \label{eq:quasi-pseuodconvex-intro}\\
R_\lambda \rfloor d(d\psi \circ J) & = & g\, d\psi \label{eq:Reeb-flat-intro}
\eea
for some functions $k,\, g$, $h$ on $U$. 
We call such a pair $(\psi,J)$
a \emph{contact quasi-pseudoconvex pair} on $U$
\end{defn}
Here the inequality \eqref{eq:quasi-pseuodconvex-intro} needs some explanation:
We recall a CR almost complex structure is an endomorphism $J(\xi) \subset \xi$ that satisfies
$$
J(R_\lambda) = 0\, \quad J^2 = -id|_\xi \oplus 0.
$$
In particular $J(TM) \subset \xi$.
The meaning of \eqref{eq:quasi-pseuodconvex-intro} is that
\be\label{eq:positive(1,1)-current}
-d(d \psi \circ J) + k \, d\psi \wedge \lambda  =  h\, d\lambda
\ee
on $\xi$ for some nonnegative function $h \geq 0$.
\begin{rem}
In other words, the two form $-d(d(\psi\circ J)$ we are considering 
in \eqref{eq:quasi-pseuodconvex-intro}
corresponds to the real almost complex version of 
the standard Levi-form in several complex variables,
when $M$ is a CR manifold of the hypersurface-type. 
The condition \eqref{eq:quasi-pseuodconvex-intro}
in particular implies pseudoconvexity of the hypersurface at every critical point of $\psi$.
The condition \eqref{eq:Reeb-flat-intro} is an additional requirement that involves
contact geometry. This is responsible for our naming of `contact $J$ quasi-convexity' and
`contact quasi-pseudoconvex pairs'. Similar notion of such a pair is also utilized in
\cite{oh:intrinsic} in the study of Liouville sectors introduced in \cite{gps,gps-2} in symplectic geometry.
\end{rem}

The upshot of introducing this kind of barrier functions on contact manifold is the following
amenability of the maximum principle to the pair $(\psi,J)$ in the study of contact instantons.

\begin{thm}[Theorem \ref{thm:C0estimate}]\label{thm:subharmonic-intro} Let $(M,\xi)$ be a contact manifold and consider the contact triad
$(M,\lambda,J)$ associated to it. Let $\psi$ be a $\lambda$-tame contact $J$ quasi-plurisubharmonic
 function. Then for any contact instanton $w: \dot \Sigma \to M$ for the
triad $(M,\lambda, J)$, the composition $\psi\circ w$ is
a quasi-subharmonic function, i.e., satisfies
$$
\Delta(\psi\circ w)\, dA + d(\psi\circ w) \wedge \lambda \geq 0
$$
for some one-form $\beta$ on $\dot \Sigma$. In particular, the maximum of $\psi\circ w$
cannot be achieved on the interior of $\dot \Sigma$. The same holds for the time-dependent contact
triad $(M,\lambda_t,J_t)$ and time-dependent $g_t, \, \beta_t$.
\end{thm}

Motivated by this analytical fact, we introduce the following class of contact manifolds.

\begin{defn}[Tame contact manifolds]\label{defn:tame-intro}
Let $(Q,\xi)$ be a contact manifold, and let $\lambda$ be a contact form of $\xi$.
\begin{enumerate}
\item  We say $\lambda$ is \emph{tame on $U$}
if $(M,\lambda)$ admits a pair $(\psi, J)$ of a $\lambda$-adapted CR almost complex structure $J$
and a $\lambda$-tame contact $J$ quasi-plurisubharmonic \emph{proper} exhaustion function $\psi$ on $U$.
\item We call an end of $(M,\lambda)$ \emph{tame} if $\lambda$ is tame on the end of $M$.
\end{enumerate}
We say an end of contact manifold $(M,\xi)$ (resp. $(M,\xi)$) is tame if it admits contact form $\lambda$
that is tame on the end (resp. at infinity) of $M$.
\end{defn}
The one-jet bundle is tame while the contact product $Q \times Q \times \R$ is tame for the end with $\eta > 0$.
(See \cite{oh-yso:spectral} and \cite{oh:shelukhin-conjecture} for their proofs respectively.)

\subsection{Dynamical untanglement of Legendrian submanifolds}

The first main result of the present paper is the following existence theorem of
Reeb chords between $\psi(R)$ and $R$ for any compact Legendrian submanifold $R$
on tame contact manifolds, as an
application of the analytic framework of perturbed contact instantons.

\emph{We will always assume that either $M$ is compact or $H$ is compactly supported otherwise.}

Now we recall the definition of oscillation
$$
\osc(H_t) = \max H_t - \min H_t
$$
and the $L^{1,\infty}$-norm of $H$ defined by
$$
\|H\|: =  \int_0^1 \osc(H_t) \, dt.
$$
\begin{thm}[Theorems \ref{thm:contact-mrl} \& \ref{thm:lower-bound}]\label{thm:shelukhin-intro}
Let $(M,\xi)$ be a contact manifold equipped with a tame contact form $\lambda$. Let $\psi \in \Cont_0(M,\xi)$ and
consider any compactly supported
Hamiltonian $H = H(t,x)$ with $H \mapsto \psi$. Assume $R$ is any compact
Legendrian submanifold of $(M,\xi)$. Then the following hold:
\begin{enumerate}
\item Provided $\|H\| \leq T_\lambda(M,R)$, we have
$$
\#\frak{Reeb}(\psi(R),R) \neq \emptyset.
$$
\item Provided $\|H\| < T_\lambda(M,R)$
and $\psi= \psi_H^1$ is nondegenerate to $(M,R)$, then
$$
\#\frak{Reeb}(\psi(R),R) \geq \dim H^*(R; \Z_2)
$$
\end{enumerate}
\end{thm}

\begin{rem} This result is optimal in general: We refer readers to
\cite[Example 1.4 \& Lemma 1.6]{rizell-sullivan} for an example that shows that Theorem \ref{thm:shelukhin-intro}
is optimal on $\R^5$ equipped with the standard contact structure. We thank Dylan Cant for pointing out this example from \cite{rizell-sullivan} to the author.
\end{rem}

We would like to readers' attention that when $\psi = id$, we $\psi(R) \cap R = R$ and so have
plenty of constant translated intersection points.
To uniformly relate the constant paths
to Reeb chords in our study of translated intersection points via contact instantons,
the following representation of Reeb chords is useful.

\begin{defn}[Iso-speed Reeb chords]
\label{defn:Reeb-chords-2} Let $R_0, \, R_1$ be two Legendrian submanifolds, not necessarily disjoint,
and consider a curve $\gamma:[0,1] \to M$ \emph{fixed domain}.
We say a pair $(T,\gamma)$ a \emph{iso-speed Reeb chord} from $R_0$ to $R_1$ if it satisfies
\be
\begin{cases}
\dot \gamma(t) = T R_\lambda(\gamma(t)) \\
\gamma(0) \in R_0, \quad \gamma(1) \in R_1.
\end{cases}
\ee
We say an iso-speed Reeb chord $(T,\gamma)$ is \emph{nonnegative} if $T\geq 0$ and
negative if $T < 0$. We denote by $\frak{X}(R_0,R_1)$
the set of iso-speed Reeb chords from $R_0$ to $R_1$.
\end{defn}

\begin{rem}\label{rem:generators}
\begin{enumerate}
\item We highlight that the set of iso-speed Reeb chords includes constant curves when
$R _0 \cap R_1 \neq \emptyset$, i.e., allows $T$ to be 0,
while $T$ in Definition \ref{defn:Reeb-chords-intro}
must be positive. This set will play the role of generators of the
contact instanton Floer complex we utilize in relation to the proof of Sandon-Shelukhin's
conjecture for the nondegenerate case.  (See Theorem \ref{thm:lower-bound}.)
By including constant curves enables us to extend the definition of contact instanton
Floer homology to the case of Morse-Bott situation e.g., the case with $R_0 = R_1 =R$
 as in the Lagrangian intersection Floer homology theory as in \cite{fooo:book1}.
 The existence result proved in Theorem \ref{thm:shelukhin-intro} essentially
 follows from the fact that we have a continuum of solution for this Morse-Bott case.
\item
The $\gamma$ in the definition of a iso-speed Reeb chord also naturally arises as
an asymptotic limit of a finite $\pi$-energy contact instanton
 (Theorem \ref{thm:subsequence-convergence}).
 \end{enumerate}
 \end{rem}

Theorem \ref{thm:shelukhin-intro}
 motivates us to introduce the following notion of \emph{Reeb-untangling energy} of one
subset from the \emph{Reeb trace} of the other:
We call the following union
\be\label{eq:S-trace}
Z_S: = \bigcup_{t \in \R} \phi_{R_\lambda}^t(S)
\ee
the Reeb trace of a subset $S\subset M$.

\begin{defn} Let $(M,\xi)$ be a contact manifold, and let
$S_0, \, S_1$ of compact subsets $(M,\xi)$.
\begin{enumerate}
\item  We define
\be\label{eq:lambda-untangling-energy-intro}
e_\lambda^{\text{\rm trn}}(S_0, S_1) : = \inf_H\{ \|H\| ~|~ \psi_H^1(S_0) \cap Z_{S_1} = \emptyset\}.
\ee
We put $e_\lambda^{\text{\rm trn}}(S_0, S_1) = \infty$ if $\psi_H^1(S_0) \cap Z_{S_1} \neq \emptyset$ for
all $H$. We call $e_\lambda^{\text{\rm trn}}(S_0, S_1)$ the \emph{$\lambda$-untangling energy} of $S_0$ from $S_1$
or just of the pair $(S_0,S_1)$.
\item We put
\be\label{eq:Reeb-untangling-energy-intro}
e^{\text{\rm trn}}(S_0,S_1) = \sup_{\lambda\in \CC(\xi)} e_\lambda^{\text{\rm trn}}(S_0, S_1).
\ee
We call $e^{\text{\rm trn}}(S_0,S_1)$ the \emph{Reeb-untangling energy} of $S_0$ from $S_1$ on $(M,\xi)$.
\end{enumerate}
\end{defn}
We mention that the quantity $e_\lambda^{\text{\rm trn}}(S_0, S_1)$ is \emph{not} symmetric, i.e.,
$$
e_\lambda^{\text{\rm trn}}(S_0, S_1) \neq e_\lambda^{\text{\rm trn}}(S_1, S_0)
$$
in general.

Theorem \ref{thm:shelukhin-intro} implies
$$
e^{\text{\rm trn}}_\lambda(R,R) \geq T_\lambda(M,R) > 0
$$
for all compact Legendrian submanifolds $R$.

\subsection{Hamiltonian-perturbed contact instantons}

Recall a contact form $\lambda$ admits a decomposition $TM = \xi \oplus \R\langle R_\lambda \rangle$. We denote
the associated projection to $\xi$ by $\pi: TM \to \xi$ and decompose
$$
v = v^\pi + \lambda(v) \, R_\lambda, \quad v^\pi: = \pi(v).
$$
\begin{defn}[Contact triad \cite{oh-wang:connection}]
\label{defn:adapted-J} Let $(M,\xi)$ be a contact manifold, and $\lambda$ be a contact form of $\xi$.
An endomorphism $J: TM \to TM$ is called a \emph{$\lambda$-adapted CR-almost complex structure} if it satisfies
\begin{enumerate}
\item $J(\xi) \subset \xi$, $JR_\lambda = 0$ and $J|_\xi^2 = -id|_\xi$,
\item $g_\xi: = d\lambda(\cdot, J|_{\xi} \cdot)|_{\xi}$ defines a Hermitian vector bundle $(\xi,J|_{\xi},g_\xi)$.
\end{enumerate}
We call the triple $(M,\lambda, J)$ a \emph{contact triad}.
\end{defn}
For given such a triad, we first decompose any $TM$-valued one-form $\Xi$ on a Riemann surface
$(\Sigma, j)$ into
$$
\Xi = \Xi^\pi + \lambda(\Xi)\, R_\lambda
$$
and then we further decompose
$$
\Xi^\pi = \Xi^{\pi(1,0)} + \Xi^{\pi(0,1)}
$$
into $J$ linear and $J$ anti-linear parts of $\Xi^\pi$.

We now consider $(M, \lambda, J)$ is a contact triad for the contact manifold $(M, \xi)$, and equip with it the
\emph{contact triad metric}
$$
g= d\lambda(\cdot, J\cdot) +\lambda\otimes\lambda.
$$

In terms of this splitting the contact Hamilton's equation can be decomposed
\be\label{eq:equation-decompose}
\dot x = X_H(t,x) \Longleftrightarrow
\begin{cases} (\dot x - X_H(t,x))^\pi = 0 \\
\gamma^*(\lambda + H\, dt) = 0
\end{cases}
\ee
into the $\xi$-component and the Reeb component of the equation. (See \cite{oh:contacton-Legendrian-bdy}.) 
We now introduce a Hamiltonian-perturbed contact
instanton equation that is introduced in \cite{oh:contacton-Legendrian-bdy}
as the contact counterpart of the celebrated Floer's Hamiltonian-perturbed
Cauchy-Riemann equation in symplectic geometry. For this purpose, we take the following
notation throughout the paper.

\begin{notation}\label{notation:J} In the rest of the paper, $J_0$ always stands for
the usual (time-independent) CR-almost complex structure $J_0 = J_0(x)$, $x \in M$,
$J$ stands for the \emph{domain-dependent} one, e.g., $J = \{J_t\}_{t \in [0,1]}$ or
$J = \{J_{(s,t)}\}_{(x,t) \in [0,1]^2}$ or even $J = \{J_z\}_{z \in \Sigma}$ for
some bordered Riemann surface $\Sigma$. On noncompact contact manifolds, we always assume
that outside a compact subset of $M$, there exists a compact subset $K \subset M$ such
that  pair $(H,J)$ satisfies
\be\label{eq:compact-support}
\supp H \subset K, \quad J_z \equiv J_0
\ee
for a fixed $\lambda$-adapted CR-almost complex structure $J_0$.
\end{notation}
The requirement on $(H,J)$ in \eqref{eq:compact-support} can be always assumed
for the purpose of studying Sandon-Shelukhin type problem of untangling Legendrian
submanifolds.

\begin{defn}[\cite{oh:contacton-Legendrian-bdy}]\label{defn:contacton-Legendrian-bdy}
Let $(M,\lambda)$ be a contact manifold equipped with a contact form, and
consider the (time-dependent) contact triad
$$
(M,\lambda, J), \quad J = \{J_t\}_{t \in [0,1]}.
$$
Let $H = H(t,x)$ be a time-dependent Hamiltonian.
We say $u: \R \times [0,1] \to M$ is a
\emph{$X_H$-perturbed Legendrian Floer trajectory} if it satisfies
\be\label{eq:perturbed-contacton-intro}
\begin{cases}
(du - X_H \otimes dt)^{\pi(0,1)} = 0, \quad d(e^{g_H(u)}(u^*\lambda + H\, dt)\circ j) = 0\\
u(\tau,0) \in R, \quad u(\tau,1) \in R
\end{cases}
\ee
where the function $g_H: \R\times [0,1] \to \R$ is defined by
\be\label{eq:gHu-intro}
g_H(t,x) := g_{(\psi_H^1 (\psi_H^t)^{-1})^{-1}}(u(t,x)).
\ee
\end{defn}
Here we refer to Subsection \ref{subsec:gauge-transformation} for the explanation and the perspective of
the transformation $u \mapsto \overline u$ where
$$
\overline u(\tau,t): = (\psi_H^t (\psi_H^1)^{-1})^{-1}(u(t,x)) = \psi_H^1 (\psi_H^t)^{-1}(u(t,x))
$$
which we apply above in \eqref{eq:gHu-intro}.
This kind of coordinate change from the \emph{dynamical version} to the \emph{intersection theoretic
version} of the Floer homology has been systematically utilized by the present author
in the symplectic Floer theory. (See \cite{oh:dmj,oh:cag,oh:dmj} and the book \cite[Section 12.7]{oh:book2}.).
To motivate this transformation, we will provide the general perspective associated to it
in Subsection \ref{subsec:gauge-transformation}. In Subsection \ref{subsec:nonautonomous-gauge},
 we will
also provide the parametric version of the gauge transformation of perturbed contact instanton
equations which enters in our proof of the Theorem \ref{thm:shelukhin-intro}.

\subsection{Definition of horizontal and vertical energies}

The a priori coercive estimates has been established in
\cite{oh-wang:CR-map1,oh-wang:CR-map2} for the closed string case,
and then in \cite{oh-yso:index} for the open string case of
Legendrian boundary conditions. The Fredholm theory and the first step of
compactification for the closed string case is carried out in \cite{oh:contacton}
\emph{under the asymptotic charge vanishing hypothesis}. 
It is proved in \cite{oh-yso:index} that this asymptotic charge
always vanishes for the open string case of Legendrian boundary condition.

In Part \ref{part:bubbling}, we establish all these ingredients for this moduli space up to
the level of what are needed to complete the proof of Theorem \ref{thm:shelukhin-intro}.
We leave a full account of the Fredholm theory and the compactification of
general contact instanton moduli spaces in a sequel \cite{oh-yso:index}, \cite{kim-oh:category}.

To make the study of Gromov-Floer-Hofer type compactness result, it is crucial to identify
a correct choice of energy in the framework of contact instantons. Such an identification of the Hofer-type
energy  is given in \cite{oh:contacton} for the closed string case of contact instantons. (See also \cite{oh-savelyev}
for such a study in the context of $\text{\rm lcs}$-instantons on locally conformal symplectic manifolds.)
We adapt this study to the open-string context of Legendrian boundary conditions for the perturbed equation, and
prove an a priori energy bound and develop the relevant bubbling argument and $C^1$-estimates in Part
\ref{part:bubbling}. We refer to Section \ref{sec:pienergy-bound}, \ref{sec:lambda-energy} in Part \ref{part:bubbling}
for the detailed explanation of the Hofer-type energy.

It turns out that the correct choice of the horizontal part of the energy, which we call the $\pi$-energy,
is the following.

\begin{defn}[The $\pi$-energy of perturbed contact instanton]\label{defn:pi-energy-intro}
Let $(J,H)$ be as in Definition \ref{defn:contacton-Legendrian-bdy}.
Let $u: \R \times [0,1] \to M$ be any smooth map. We define
\be\label{eq:EpiJHu-intro}
E^\pi_{J,H}(u): = \frac12 \int e^{g_H(u)} |(du - X_H(u) \otimes dt)^\pi|^2_J \,
\ee
call it the \emph{off-shell $\pi$-energy}, where $g_H(u)$ is the function given in \eqref{eq:gHu-intro}.
\end{defn}

We apply the gauge transformation $\Phi_H^{-1}$  to define a map $\overline u$ by
\be\label{eq:u-to-w}
\overline u(\tau,t) = (\psi_H^t (\psi_H^1)^{-1})^{-1}(u(\tau,t)) = \psi_H^1 (\psi_H^t)^{-1}(u(\tau,t))
\ee
and
\bea\label{eq:J-prime}
J'  =  \{J_t'\}, \quad J_t' & = & (\psi_H^t (\psi_H^1)^{-1})^*J, \\
\lambda_t' & = &  (\psi_H^t (\phi_H^1)^{-1})^*\lambda.
\eea
Then for the given contact triad $(M,\lambda, J)$, the triple
$$
(M,\{\lambda_t'\},\{J_t'\})
$$
forms a $t$-dependent family of contact triads.
We denote the associated contact triad metric by
$$
g_t' = (\psi_H^t (\phi_H^1)^{-1})^*g.
$$
The upshot of this coordinate change is that
if $u$ satisfies \eqref{eq:perturbed-contacton-intro}
with respect to $J_t$, then $\overline u$ is a \emph{nonautonomous} contact instanton
that satisfies
$$
\delbar^\pi_{J'} \overline u = 0,
$$
with the boundary condition
$$
\overline u(\tau,0) \in \psi_H^1(R_0), \, \overline u(\tau,1) \in R_1,
$$
where we have
$$
(\delbar^\pi_{J'} \overline u)(\tau,t) = (d\overline u)^{\pi(0,1)}_{J_t'}(\tau,t).
$$

The following identity justifies the presence of weight function $e^{g_H(u)}$ in the
definition of energy \eqref{eq:EpiJHu-intro} with $g_H(u)$  defined in \eqref{eq:gHu-intro}.
(See also Remark \ref{rem:weight} for its naturality and further necessity of the presence of the weight factor.)

\begin{prop}[Proposition \ref{prop:EpiJH=Epi} \& \ref{prop:pienergy-dlambda}]\label{prop:EpiJH=Epi-intro}
Let $u: \R \times [0,1] \to M$ be any smooth map and $\overline u$ be as above. Then
\be\label{eq:EpiJH=Epi-intro}
E^\pi_{J,H}(u) = E^\pi_{J'}(\overline u) = \int \overline u^*\lambda.
\ee
\end{prop}

Another crucial component of Hofer-type energy, the vertical part of energy,
is more nontrivial to describe. We refer readers to Section \ref{sec:lambda-energy},
Part \ref{part:bubbling} for the details of its construction
(See \cite{oh:contacton} also for the same definition considered in the closed string context.)
We will show in Part \ref{part:bubbling}
that these two energies will govern  the convergence behavior of perturbed contact instantons,
similarly as in the case of pseudoholomorphic curves in the symplectization \cite{hofer:invent,behwz}.

\subsection{Cut-off Hamiltonian-perturbed contact instanton equation}

For the purpose of relating the existence question of Reeb chords between
the pair $(\psi(R),R)$ to the oscillation norm of the Hamiltonian $H \mapsto \psi$, we will adapt the scheme
laid out in \cite{oh:mrl} to the context of perturbed contact instantons
with Legendrian boundary conditions in contact manifolds and prove Theorem \ref{thm:shelukhin-intro}:
The scheme of \cite{oh:mrl} was used for the study of displacement energy of Lagrangian submanifolds
on symplectic manifolds.

We consider the perturbed contact instanton with \emph{moving Legendrian boundary condition},
and set-up the deformation-cobordism framework of parameterized
moduli space of perturbed contact instantons by adapting a similar parameterized Floer moduli spaces
used in \cite{oh:mrl}.

We will be particularly interested in the case for which the domain
dependent Hamiltonian arises as
\be\label{eq:HK}
H_K(\tau,t,x) := \Dev_\lambda \left(t \mapsto \psi_H^{\chi_K(\tau)t}\right)(x)
\ee
where $\chi_K: \R \times [0,1] \to [0,1]$ is the family of cut-off functions
used in \cite{oh:mrl}. (See Section \ref{sec:cut-off} for the definition.)
Especially, we will concern the two-parameter family $H$ of the form
which is a crucial ingredient in our proof of the main theorem.

\begin{choice}[Choice of two-parameter family Hamiltonian]\label{choice:H}
We take the curvature-free family of Hamiltonians from the identity to $\psi_H^1$ given by
$$
H(s,t,x) = \Dev_\lambda(t \mapsto \psi_H^{st})(x)
$$
for the compact family $(s,t) \in [0,1]^2$.
\end{choice}
\emph{We would like to emphasize that the more standard practice of using 
the family $\{sH\}_{0 \leq s \leq 1}$ for  the study of spectral invariants in
symplectic geometry  will result in the appearance of 
conformal factor multiplied as in \cite{rizell-sullivan}.}
(See \cite[Introduction]{oh-yso:spectral} for the remark making this point
clearly emphasized.)

\begin{notation}[The family $(H^s,J^s)$]
Consider the two parameter family of
CR-almost complex structures and Hamiltonian functions:
$$
J = \{J_{(s,t)}\}, \, H = \{H_{(s,t)}\} \quad \textrm{for} \;\; (s,t) \in
[0,1]^2. \;
$$
We denote by $H^s$ the Hamiltonian given by 
$
H^s(t,x): = H(s,t,x)
$
and $J^s$ given by $J^s(t,x) = J(s,t,x)$.
\end{notation}

Note that $[0,1]^2$ is a compact set and so $J,\, H$ are compact
families. We will be particularly interested in the case:
\be\label{eq:Jstx}
J(s,t,x) = (\psi_{H^s}^t(\psi_{H^s}^1)^{-1})_*J_0, \quad J_0 \in \CJ(M,\xi)
\ee
For each $K \in \R_+ = [0,\infty)$, we take the one-parameter family of cut-off functions $\chi_K:\R \to
[0,1]$ and define a $\Theta_{K+1}$-family of contactomorphisms
$$
\psi_{\chi_K(\tau),t} = \psi_{H}^{\chi_K(\tau)t}(\psi_{H}^{\chi_K(\tau)})^{-1}, \quad (\tau,t)
 \in \Theta_{K+1} \subset \R \times [0,1],
$$
and associate the $\Theta_{K+1}$-family of contact triads
$$
(M,\{\psi_{\chi_K(\tau),t}^*\lambda\}, \{\psi_{\chi_K(\tau),t}^*J_0\})
$$
where the family $\{\Theta_{K+1}\}$ is a nested family of disc-like domains such that
$$
\bigcup_{K \in [0,\infty)} \Theta_{K+1} = \R \times [0,1].
$$
(See Section \ref{sec:cut-off} for the precise definition of the domain $\Theta_{K+1}$.)

We denote the associated parameterized moduli space by
$$
\CM_{[0,K_0]}^{\text{\rm para}}(M,R;J,H) = \bigcup_{K \in [0,K_0]}\{K\} \times \CM_K (M,R;J,H)
$$
for a sufficiently large $K_0 > 0$. (See Section \ref{sec:cobordism} for the
precise definition thereof.)
Denote by $\CM(M;\alpha_j)$ (resp. $\CM(M,R;\beta_k)$) the moduli space of contact instantons on the
plane $\C$ with $\alpha_j$ as its asymptotic Reeb orbits
(resp. that of contact instantons on the half plane $\H$ with boundary in $R$ and with $\beta_k$
as its asymptotic Reeb chord).

Then we prove the following fundamental a priori energy bounds

\begin{thm}[Propositions \ref{prop:pienergy-bound-H} \& \ref{prop:lambdaenergy-bound}]
\label{thm:energy-estimate-intro}
Assume $(M,\lambda)$ be a tame contact manifold. For given Hamiltonian $H=H(t,x)$, we
consider the two-parameter family 
$H^s(t,x): = \Dev_\lambda(t \mapsto \psi_H^{st})(x)$ and
\be\label{eq:HK-intro}
H_K(\tau,t,x) = \Dev_\lambda\left(t \mapsto \psi_H^{\chi_K(\tau)t}\right)(x)
\ee
Let $(K,u) \in \CM^{\text{\rm para}}(M,\lambda;J,H)$ be any element. Then we have
\begin{enumerate}
\item {[The horizontal energy bound]}
$$
E_{J_K,H_K}^\pi(u) \leq \|H\|.
$$
\item {[The vertical energy bound]}
$$
E_{J_K,H_K}^{\perp}(u) \leq \|H\|
$$
\end{enumerate}
\end{thm}

We would like to emphasize that the proofs of these a priori bounds much rely on
the particular form of Hamiltonian-perturbed contact instanton equation \eqref{eq:perturbed-contacton-intro}
under the Legendrian boundary condition and utilize the calculations
based on the contact Hamiltonian calculus. (See Proposition \ref{prop:pienergy-bound2} and
Proposition \ref{prop:lambdaenergy-bound2} respectively.)
These calculi are systematically developed and organized in
\cite[Section 2]{oh:contacton-Legendrian-bdy} with coherent signs,
notations and conventions, and further developed in  Subsection \ref{subsec:contact-Hamiltonian}
of the present paper.

\begin{rem}\label{rem:family-gauge-transform}
If we write $J^s = (\psi_{H^s}^1)^{-1})_*J_0$, then we can write
$$
J(s,t,x) = (\psi_{H^s}^t)_*J^s, \quad s \in [0,1].
$$
Therefore we have a family of $s$-dependent contact triads, one \emph{autonomous}
$$
(M,\lambda^s,J^s)
$$
and the other \emph{non-autonomous}
$$
(M,\{\lambda_t^s\}, \{J_t^s\})
$$
which are related by the one-parameter family of gauge transformations
$$
\Phi_{H^s}, \quad s \in [0,1].
$$
\end{rem}

\subsection{Compactification of $\CM^{\text{\rm para}}(M,R;J,H)$}

The following is a sample Gromov-Floer-Hofer type general convergence result 
 that we use in the proof of Theorem \ref{thm:shelukhin-intro}. We refer readers
 to Section \ref{sec:C1-estimates-J1B} for more detailed explanations of the 
 statement of the theorem.

\begin{thm}[Theorem \ref{thm:bubbling}]\label{thm:bubbling-intro}
Consider the moduli space $\CM^{\text{\rm para}}(M,R;J,H)$ under the assumption of 
uniform energy bound. Then one of the following alternatives holds:
\begin{enumerate}
\item
There exists some $ C  > 0$ such that
\be\label{eq:dwC0-intro}
\|d u\|_{C^0;\R \times [0,1]} \leq C
\ee
where $C$ depends only on $(M,R;J,H)$ and $\lambda$.
\item There exists a sequence $u_\alpha \in \CM_{K_\alpha}(M,R;J,H)$ with $K_\alpha \to K_\infty \leq K_0$
and a finite set $\{\gamma_j^+\}$ of closed Reeb orbits of $(M,\lambda)$ such that $u_\alpha$
weakly converges to the union
$$
u_\infty = u_{-,0 } + u_0 + u_{+,0}+ \sum_{j=1} v_j + \sum_k w_k
$$
in the Gromov-Floer-Hofer sense, where
$$
\begin{cases}
u_{-,0} \in \overline{\CM}_-^{\text{\rm para}} (M,R;J,H), \\
u_{+,0} \in \overline{\CM}_+^{\text{\rm para}}K_+(M,R;J,H),\\
u_0 \in   \overline{\CM}(M,R;J,H),
\end{cases}
$$
$$
v_j \in \overline{\CM}(M,J_{z_j}';\alpha_j); \quad \alpha_j \in \frak{Reeb}(M,\lambda_{z_j}'),
$$
and
$$
w_k \in \overline{\CM}(M,\psi_{z_j}(R), J_{z_j}';\beta_k); \quad \beta_k \in \frak{Reeb}(M,R;\lambda).
$$
\end{enumerate}
Here the domain point $z_j \in \del \dot \Theta_{K_\infty +1}$ is the point at which the corresponding bubble is attached.
\end{thm}
Since all of the components, except $u_0$, appearing here have their domains either of disk-type
or sphere-type with one-puncture whose associated asymptotic limits are (nonconstant) Reeb chords or 
Reeb orbits. In particular, under the hypothesis $\|H\| < T_\lambda(M;R)$, any of the finite-time bubbling components
$v_j$ and $w_k$ will not be generated. As a consequence, any such sequence $u_\alpha$ 
will carry a uniform $C^1$-estimate i.e., there is some $C > 0$ such that $\|u_\alpha\|_{C^1} < C < \infty$.
This, combined with our energy estimate, will then conclude the proof of Theorem \ref{thm:shelukhin-intro}.

The proof of this theorem will occupy Part \ref{part:bubbling}. Its content,
together with the gluing theory of contact instantons developed in \cite{oh:contacton-gluing}
and basic elliptic regularity theory for the equation \eqref{eq:perturbed-contacton-intro}
developed in \cite{oh:perturbed-contacton-bdy},
will be also the analytic foundation for the applications of the perturbed contact instantons
(with Legendrian boundary condition) in the sequel \cite{oh:entanglement2}.
It is also the analytic basis for the Floer theoretic construction of
Legendrian spectral invariants via contact instanton Floer-type homology given in 
\cite{oh-yso:spectral} for the one-jet bundle, and the construction of Legendrian Fukaya category
provided in \cite{kim-oh:category}.

\subsection{Discussion}

One natural avenue to pursue in the future
is to amplify our quantitative study given in the present paper to define the
\emph{contact instanton spectral invariants} of general Legendrian links and investigate
their entanglement structure. We also hope to further apply this machinery
to problems of contact topology and thermodynamics.
(See \cite{bravetti-lopez-nettel}, \cite{MNSaSc} and \cite{lim-oh:thermodynamics}, for example, in the midst of
many articles on contact geometric  formulation of thermodynamics in physics literature.
Entov and Polterovich applied ideas from contact dynamics to
a problem of non-equilibrium thermodynamics in a recent article \cite{entov-polterov:thermodynamics}.)
The first step towards this goal together with some applications is given in \cite{oh-yso:spectral}
where Legendrian spectral invariants, the analog of Floer theoretic construction of
Lagrangian spectral invariants given in \cite{oh:jdg}, are constructed.
Further study in general and other applications will be given in \cite{kim-oh:category}, \cite{oh:entanglement2}.

 In the present paper, we consider the special family $J'$ given in \eqref{eq:J-prime}
 so that the perturbed equation \eqref{eq:perturbed-contacton-intro} for the family $J$
 can be converted to the (unperturbed) contact instanton equations for $J'$
 by the gauge transformation $\Phi_H$ and utilized the analysis established
 in \cite{oh:contacton-Legendrian-bdy,oh-yso:index} for $H = 0$.
 In \cite{oh:perturbed-contacton-bdy,oh-yso:index}, we establish the coercive elliptic estimates
for the perturbed equation \eqref{eq:perturbed-contacton-intro} itself for the general family $J'$ not
necessarily of the type given in \eqref{eq:J-prime} so that one cannot convert to
the (unperturbed) contact instanton equations by the gauge transformation $\Phi_H$.
Such an analysis will be important for the general study and applications of the perturbed equation
in contact topology such as construction of contact analogs of the spectral invariants
\cite{oh:jdg,oh:alan,oh:dmj}, similarly as Floer's Hamiltonian-perturbed Cauchy-Riemann
equation does in symplectic topology.

The context of the present paper in Part 3
together with the gluing theory of contact instantons developed in \cite{oh:contacton-gluing}
and basic elliptic regularity theory for the equation \eqref{eq:perturbed-contacton-intro}
developed in \cite{oh:perturbed-contacton-bdy},
are the analytic foundation for the applications of the perturbed contact instantons
(with Legendrian boundary condition) in the sequels \cite{oh:shelukhin-conjecture}, \cite{oh:entanglement2},
\cite{kim-oh:category}. It is also the analytic basis for the Floer theoretic construction of
Legendrian spectral invariants via contact instanton Floer-type homology given in \cite{oh-yso:spectral}
for the one-jet bundle.

Throughout the paper, neither symplectization nor pseudoholomorphic
curves in symplectic geometry is involved. We refer readers to
the survey article \cite{oh-kim:survey} for the detailed relationship
between the analysis of (unperturbed) contact instantons and that of
pseudoholomorphic curves on symplectization.

We would like to emphasize
that \emph{with the analytical foundations of (perturbed) contact instantons
in our  disposal}, the remaining study of contact Hamiltonian dynamics
utilizing perturbed contact instantons  e.g., construction of relevant
contact spectral invariants on contact manifolds is largely geometro-topological
and dynamical. This enables us to carry out such a study \emph{in an optimal way}
because the perturbed contact instanton equation \eqref{eq:perturbed-contacton}
interacts with \emph{contact Hamiltonian calculus} in a straightforward
canonical fashion as illustrated by those in the present paper and in \cite{oh:shelukhin-conjecture}:
Many calculations given in the present paper and in its sequels such as \cite{oh:shelukhin-conjecture}
are of the nature of contact Hamiltonian calculus, not those arising from homogeneous Hamiltonian
dynamics. By its purely contact nature of contact instanton equation, it is supposed to interact best
with the contact Hamiltonian dynamics and calculus. In this regard
many calculations provided in the present paper are purely of contact nature, and appear
for the first time that have no their precedent  in the literature.
(See the proof of Proposition \ref{prop:pienergy-bound-H}, for example.)
We refer readers to our survey paper \cite{oh-kim:survey} (or the latest arXiv version of the
present paper) for the detailed exposition on
the relationship between the analysis of contact instantons and that of pseudoholomorphic
curves on symplectization.

\medskip
\noindent{\bf Acknowledgement:}
We would like to thank Bhupal and Sandon for patient explanation on the generating function
techniques on the Legendrian spectral invariants and their applications at the time of
ideas on the current research starting to form, which helped the author
to appreciate differences between contact dynamics and symplectic Hamiltonian dynamics.
We also thank Rui Wang for her collaboration in the very beginning of our study of
contact instantons.
We also like to thank Dylan Cant for pointing out that the contact product $Q \times Q \times \R$ is
not tame at both ends of $\{\eta > 0\}$ and $\{\eta < 0\}$  in \cite{oh:shelukhin-conjecture}. 
Finally, but not the least, I would like to express our sincere gratitude to the unknown referees
for  their hard work of  pointing out a few mathematical errors,
 imprecise presentations of the previous version of the present article
which we think much improves readability and exposition of the present version.
\bigskip

\noindent{\bf Convention and Notations:}

\medskip

\begin{itemize}
\item {(Contact Hamiltonian)} We define the contact Hamiltonian of a contact vector field $X$ to be
$- \lambda(X) =: H$.
\item
For given time-dependent function $H = H(t,x)$, we denote by $X_H$ the associated contact Hamiltonian vector field
whose associated Hamiltonian $- \lambda(X_t)$ is given by $H = H(t,x)$, and its flow by $\psi_H^t$.
\item {(Developing map)} $\Dev_\lambda(t \mapsto \psi_t)$: denotes the time-dependent contact Hamiltonian generating the contact Hamiltonian path $t \mapsto \psi_t$.
\item {(Reeb vector field)} We denote by $R_\lambda$ the Reeb vector field associated to $\lambda$
and its flow by $\phi_{R_\lambda}^t$.
\item {(Contact instanton homology)} We denote by $CI_\lambda^*(R_0,R_1)$ the $\lambda$-contact instanton complex and $HI_\lambda^*(R_0,R_1)$ its cohomology, when defined.
    (See Section \ref{sec:with-nondegeneracy} for the definition.)
\item $(\dot \Sigma,j)$: a punctured Riemann surface (with boundary) and $(\Sigma,j)$ the associated compact Riemann surface.
\item We always regard the tangent map $du$ as a $u^*TM$-valued one-form and write
$du = d^\pi u + u^*\lambda \otimes R_\lambda$ with respect to the decomposition $TM = \xi \oplus \R \langle R_\lambda \rangle$.
\item $g_H(u)$: the function defined by $g_H(u)(\tau,t): = g_{\psi_H^1 (\psi_H^t)^{-1}}(u(\tau,t))$.
\item $T(M,\lambda)$: the infimum of the action of closed $\lambda$-orbits.
\item $T(M,R;\lambda)$: the infimum of the action of self $\lambda$-chords.
\item $T_\lambda(M,R) = \min\{ T(M,\lambda), T(M,R;\lambda)\}$.
\end{itemize}

\section{Summary of contact Hamiltonian calculus}
\label{subsec:contact-Hamiltonian}

In this section, we summarize basic contact Hamiltonian calculus we are going to use. We follow the exposition
of \cite[Section 2]{oh:contacton-Legendrian-bdy} and its sign conventions, and amplify the calculus further
which will be needed for the purpose of the present paper and its sequels. We will always assume that $(M,\xi)$
is cooriented.

\begin{defn} For given coorientation preserving contact diffeomorphism $\psi$ of $(M,\xi)$ we call
the function $g$ appearing in $\psi^*\lambda = e^g \lambda$ the \emph{conformal exponent}
for $\psi$ and denote it by $g=g_\psi$.
\end{defn}

\begin{defn} Let $\lambda$ be a contact form of $(M,\xi)$.
For each contact vector field $X$, the associated function $H$ is given by
\be\label{eq:contact-Hamiltonian}
H = - \lambda(X)
\ee
is called the \emph{$\lambda$-contact Hamiltonian} of $X$. We also call $X$ the
\emph{$\lambda$-contact Hamiltonian vector field} associated to $H$.
\end{defn}

When $(M,\xi)$ is cooriented, the line bundle $L = \R_M$ is trivial and the associated
Jacobi bracket is also called the \emph{Lagrange bracket} in some literature. (See \cite{arnold-givental}.)
When a contact form $\lambda$ is given
the map $X \mapsto -\lambda(X)$ induces a Lie algebra isomorphism,
which does not necessarily satisfy the Leibnitz rule.
  Partly due to different sign conventions
literature-wise, we fix the sign convention of the bracket following
that of \cite{oh:contacton-Legendrian-bdy}.

\begin{prop}[Compare with Proposition 5.6 \cite{LOTV}, Proposition 9 \cite{dMV}]
Let $(M,\xi)$ be cooriented and $\lambda$ be an associated contact form.
Define a bilinear map
$$
\{\cdot, \cdot\}: C^\infty(M) \times C^\infty(M) \to C^\infty(M).
$$
by
\be\label{eq:bracket}
\{H,G\}: = \lambda([X_H,X_G]).
\ee
Then it satisfies Jacobi identity and the assignment $X \mapsto -\lambda(X)$ defines a
Lie algebra isomorphism $\mathfrak X(M,\xi)$ to $C^\infty(M)$.
\end{prop}

Under our sign convention, the $\lambda$-Hamiltonian $H$ of the Reeb vector field $R_\lambda$
as a contact vector field becomes the constant function $H = -1$.

Next let $\psi_t$ be a contact isotopy of $(M, \xi = \ker \lambda)$ with $\psi_t^*\lambda = e^{g_t}\lambda$
and let $X_t$ be the time-dependent vector field generating the isotopy. Let
$H: [0,1] \times M \to \R$ be the associated time-dependent contact Hamiltonian
$H_t = - \lambda(X_t)$.

We denote by $\Cont_0(M,\xi)$ the
identity component of the group $\Cont(M,\xi)$ of (orientation-preserving) contactomorphisms.
We denote by
$$
\CP(\Cont(M,\xi))
$$
the groupoid of (Moore) paths  of contactomorphisms $\ell:[0,T] \to \Cont(M,\xi)$. We call an element thereof
a \emph{contact Hamiltonian path}. We have three obvious maps,
the assignment $\ell \to T$ of the domain length of $\ell$, and the source and the target map
\be\label{eq:maps}
s, \, t: \CP(\Cont(M,\xi)) \to \Cont(M,\xi); \quad s(\ell) = \ell(0), \, t(\ell) = \ell(T)
\ee
as usual. (When we consider an element of Moore path $\ell$, one
often writes it as a pair $(T,\ell)$.)

When $(M,\xi)$ is equipped with a contact form $\lambda$ and we consider a path
on $[0,1]$, i.e., with $T = 1$,  we can associated another map the $\lambda$-developing map
$$
\Dev_\lambda: \CP(\Cont(M,\xi)) \to C^\infty([0,T] \times M,\R)
$$
by assigning its $\lambda$-contact Hamiltonian functions
\be\label{eq:lambda-contact-Hamiltonian}
\Dev_\lambda(\ell)(t,x): = -\lambda\left(\frac{\del \ell}{\del t}(t,\ell_t^{-1}(x))\right).
\ee
\emph{In the present paper, we will always assume $T = 1$ unless otherwise said.}
Unravelling this definition, we have $\Dev_\lambda(\ell)(t,x) = H(t,x)$ where $X_t$ is the
vector field generating the path $\ell$, i.e.,
$$
X_t(x) = \frac{\del \ell}{\del t}(t,\ell_t^{-1}(x)), \quad H = -\lambda(X_t).
$$
We denote by $X \mapsto \psi$ if $\psi = \psi_H^1$.

The following formulae for the contact Hamiltonians can be derived by
a straightforward calculation.
(See \cite{mueller-spaeth-I} for its first appearance.)

\begin{lem}\label{lem:product-inverse-Hamiltonian}
Let $\Psi = \{\psi_t\} \in \CP(\Cont(M,\xi))$ be a contact isotopy satisfying
$\psi_t^*\lambda = e^{g_t} \lambda$ with $g_\Psi: = g(t,x)$and generated by the vector field $X_t$ with
its contact Hamiltonian  $H(t,x) = H_t(x)$, i.e., with $\Dev_\lambda(\Psi) = H$.
\begin{enumerate}
\item Then the (timewise) inverse isotopy $\Psi^{-1}: = \{\psi_t^{-1}\}$ is generated
by the (time-dependent) contact Hamiltonian
\be\label{eq:inverse-Hamiltonian}
\Dev_\lambda(\Psi^{-1}) = - e^{-g_\Psi \circ \Psi} \Dev_\lambda(\Psi)
\ee
where the function $g_\Psi \circ \Psi$ is given by $(g_\Psi \circ \Psi)(t,x) : = g_\Psi(t,\psi_t(x))$.
\item If $\Psi' = \{\psi'_t\}$ is another contact isotopy with conformal exponent $g_{\Psi'} = \{g'_t\}$,
then the timewise product $\Psi' \Psi$ is generated by the Hamiltonian
\be\label{eq:product-Hamiltonian}
\Dev_\lambda(\Psi'\Psi) = \Dev_\lambda(\Psi') + e^{-g_\Psi \circ (\Psi')^{-1}} \Dev_\lambda(\Psi) \circ (\Psi')^{-1}.
\ee
\end{enumerate}
\end{lem}
In particular, if $\Psi$ is a strict contactomorphism like the Reeb flow $\phi_{R_\lambda}^t$,
then we have
$$
\Dev_\lambda(\Psi'\Psi) = \Dev_\lambda(\Psi') + \Dev_\lambda(\Psi) \circ (\Psi')^{-1}
$$
which is an immediate generalization of the symplectic case.

Now we prove the following contact analog of well-known Banyaga's formula from \cite{banyaga}
in symplectic geometry. For the derivation of this contact counterparts,
we need to employ somewhat different argument to derive these formulae, especially because the
bracket $\{G,H\}$, called the Lagrange bracket \cite{arnold-givental}, is defined differently from
the symplectic Poisson bracket and carries different properties, e.g., that the bracket does not satisfy
the Leibnitz rule any more. (See \cite[Proposition 5.6]{LOTV}, \cite[Proposition 9]{dMV}.)

\begin{choice}[Choice of two-parameter isotopy $\psi_{H^s}^t$]\label{choice:2parameter-isotopy}
We consider two-parameter \emph{flat} isotopy $\phi_H^{st}$  of contactomorphisms
\be \label{eq:mu}
\mu(s,t): = \phi_{H}^{st}
\ee
for any $t$-dependent Hamiltonian $H = H(t,x)$. 
\end{choice}

The following formula will play a fundamental role in our
derivation of the optimal inequality that enters in our proof of the Sandon-Shelukhin conjecture.
(See the proof of the $\pi$-energy identity stated in Proposition \ref{prop:pienergy-bound2}.)

\begin{lem}\label{lem:G} Let $G$ be the two-parameter Hamiltonian given by
$$
G(s,t,x) = \Dev_\lambda\left(s \mapsto \phi_{H}^{st}\right)(x).
$$
Then we have $G(s,t,\psi_H^{st}(x)) = t H(st,\psi_H^{st}(x))$.
In particular, we have $G(s,0,x) \equiv 0$ and
\be\label{eq:G}
G(s,1,\psi_H^s(x)) = H(s,\psi_H^s(x))
\ee
\end{lem}
\begin{proof} This follows from the direct calculation of 
$$ 
\frac{\del}{\del s} (\psi_H^{st}(x)) = t X_H(st, \psi_H^{st}(x)) 
= X_{tH}(st,  \psi_H^{st}(x))).
$$
\end{proof}

\part{Tame contact manifolds, maximum principle and contact instantons}
\label{part:co-Legendrian}

\section{Co-Legendrian submanifolds}

The notion of \emph{pre-Lagrangian} submanifolds was introduced in
\cite{eliash-hofer-salamon} which we recall now.
Let $SM = M \times \R$ be the (conical) symplectization of $M$ which is the $\R_+$-subbundle
of $T^*M$ which is formed by contact forms compatible with the given co-orientation of
$(M,\xi)$. We rephrase the definition given in \cite{eliash-hofer-salamon} as follows.

\begin{defn} Let $SM = M \times R$ be the symplectization of $M$. A submanifold
$K \subset M$ is called \emph{pre-Lagrangian} if there is a Lagrangian section
$\widehat k: K \to SM$.
\end{defn}
By setting $\widehat K = \Image \widehat k$ and $\widehat k = (\pi|_{\widehat K})^{-1}$, it can be
easily seen that this definition is equivalent to the definition
given in \cite{eliash-hofer-salamon} the condition of which reads that there is a Lagrangian lift $\widehat K \subset SM$
such that the restriction $\pi|_{\widehat K} : \widehat K \to K$ is a diffeomorphism.

We now introduce a more intrinsic notion of co-Legendrian submanifolds which does not
involve symplectization and includes that of pre-Lagrangian submanifolds as a special case.

\subsection{Definition of co-Legendrian submanifolds}

We start with recalling the definition of  coisotropic submanifolds in contact manifolds $(M,\xi)$.
(We refer to \cite{LOTV} for the definition thereof in the general, not necessarily coorientable, case.)
Majority of the discussion given in this subsection will not be directly relevant to the main
purpose of the present paper except for the purpose of providing some general perspective for a future
purpose with the \emph{Reeb trace} of Legendrian submanifold. (We also find that the notion itself is
interesting and so worthwhile to describe its geometry in some detail for a future purpose.)
The Reeb traces of Legendrian submanifolds
are the prototypes of co-Legendrian submanifolds and will enter in the intersection theoretic
translation of Sandon's translated points.

Recall by definition of $\lambda$ that $(M,d\lambda)$ is a \emph{presymplectic manifold}.

\begin{defn}\label{defn:coisotripic} Let $(M,\xi)$ be a contact manifold equipped with a contact form $\lambda$.
A ($\lambda$-)coisotropic submanifold $C \subset M$ (with respect to $\lambda$) is one that satisfies
$$
(TC)^{d\lambda} \subset TC.
$$
\end{defn}

\begin{defn} A coisotropic submanifold $C \subset (M,\lambda)$ is called
\emph{co-Legendrian} if $C$ is minimally coisotropic, i.e., if $TC = (TC)^{d\lambda}$.
\end{defn}

\begin{prop}\label{prop:co-Legendrian-property} Let $Z$ be a co-Legendrian submanifold of $M$. Then
we have $\dim Z = n+1$, and
$$
TZ = (TZ)^{d\lambda}=(\xi \cap TZ) \oplus  \R\langle R_\lambda \rangle.
$$
In particular, $\xi \cap TZ \subset TM|_Z$ is a Legendrian subbundle.
\end{prop}
\begin{proof}
We know
$$
\R\langle R_\lambda\rangle \subset (TZ)^{d\lambda} = TZ.
$$
The latter also implies $TZ/\R\langle R_\lambda\rangle = (TZ/\R\langle R_\lambda\rangle)^{[d\lambda]}$
with respect to the induced fiberwise symplectic bilinear form $[d\lambda]$ on the quotient bundle
$TM/\R\langle R_\lambda \rangle \cong \xi$ of rank $2n$. In other words,
$TZ/\R\langle R_\lambda \rangle$ is a Lagrangian subbundle of the quotient and hence
$\rank TZ/\R\langle R_\lambda\rangle = n$.

Combining the above, we have finished the proof.
\end{proof}

\begin{rem}\label{rem:deformation}
It is shown in \cite{oh-park}, \cite{zambon} (for the symplectic case) and in \cite{LOTV} (for the
contact case) that the local deformation problem of general coisotropic submanifolds
is obstructed. In particular the set of general coisotropic submanifolds is not a
smooth (Frechet) manifold \cite{zambon}. However for  the corresponding deformation problem of the
co-Legendrian case, this obstruction vanishes and so the set of co-Legendrian submanifolds
forms a smooth manifold: The quadratic term in the defining equation of coisotropic
subspace given in \cite[Proposition 2.2]{oh-park} vanishes for the co-Legendrian case!
\end{rem}

\subsection{Examples of co-Legendrian submanifolds}

\begin{exm}[Conormal jets] The most natural examples of co-Legendrian submanifolds are the
\emph{conormal one-jets} $\widetilde \nu^*N \subset J^1B$
for any submanifolds $N \subset B$ where we define $\widetilde \nu^*N$ by
\be\label{eq:conormal-one-jet}
\widetilde \nu^*N: = \pi^{-1}(\nu^*N) = \{(q,p,z) \mid (q,p) \in \nu^*N\}.
\ee
In fact, any co-Legendrian submanifold of general contact manifold is locally of
this form.
\end{exm}

The next example is the one of our main interest in the
present paper.

\begin{exm}[Reeb trace of Legendrian submanifolds]
Let $R$ be any Legendrian submanifold of $(M,\lambda)$ and consider its Reeb trace
\be\label{eq:R-trace}
Z_R: = \bigcup_{t \in \R} \phi_{R_\lambda}^t(R).
\ee
This is an immersed co-Legendrian submanifold in general.
\emph{Under the hypothetical situation}
that there is no self Reeb chord of $R$ i.e., that \emph{Arnold's Reeb chord conjecture fails for $R$},
this immersion becomes a one-one immersion.
\end{exm}

\begin{rem}[Control manifolds]
The class \eqref{eq:R-trace} of submanifolds has been considered in the geometric formulation of
thermodynamics and the information theory  in the physics literature with the name
of \emph{control manifold}. (See \cite{MNSaSc}, \cite{bravetti-lopez}, \cite{bravetti-lopez-nettel} for example.)
(Strictly speaking, the control manifolds in these references are considered only in the one-jet bundle and also
equipped with a (pseudo-)Riemannian metric in addition.) Co-Legendrian submanifolds are defined on
general contact manifold, not just in the one-jet bundle case, without being equipped with a metric.
\end{rem}

We now introduce the intersection theoretic version of translated fixed points.

\begin{defn}[Translated intersection points]\label{defn:translated-intersection}
Let $(M,\xi)$ be a contact manifold. Let $(R_0,R_1)$ a pair of Legendrian submanifolds.
\begin{enumerate}
\item Let $\lambda \in \CC(\xi)$.
We call a pair $(x, \eta) \in R_0 \times \R_+$ a \emph{$\lambda$-translated intersection point}
of $R_0$ and $R_1$ if there is a Reeb chord $\gamma$ satisfying
$$
\gamma(0) \in R_0, \quad \gamma(\eta) = x \in R_1.
$$
i.e., if $x \in R_0 \cap (\phi_{R_\lambda}^{\eta})^{-1}(R_1) \subset R_0 \cap Z_{R_1}$.
\item We say a pair $(R_0,R_1)$ is \emph{dynamically $\lambda$-entangled} if there is a Reeb
chord of $\lambda$ from $R_0$ to $R_1$, and just \emph{dynamically Reeb-entangled} if
there is $\lambda \in \CC(\xi)$ for which $\frak{Reeb}(M,R;\lambda) \neq \emptyset$.
\end{enumerate}
\end{defn}
By definition, for each given translated intersection point $(x,\eta)$ with $x \in \psi(R)$, there is a
Reeb chord $\gamma_x: [0,\eta] \to M$ defined by
$$
\gamma_x(t) := \phi_{R_\lambda}^{t\eta}(x)
$$
which satisfies
\be\label{eq:gamma-eta-x}
\gamma_x(0) = x\in \psi(R),  \,  \gamma_x(\eta) \in R
\ee
and vice versa.

We summarize the relationship between the set of translated intersection points
between $R_0$ and $R_1$ and the intersection set $\psi(R_0) \cap Z_{R_1}$ in terms of the following
general correspondence lemma.

\begin{lem}\label{lem:key-conversion} Let $(M,\lambda)$ be a contact manifold, and
let $\psi: M \to M$ be a contactomorphism. Consider any compact subset $S_0, \, S_1$ of $M$.
\be\label{eq:no-ZS-intersect-psi(R)}
\psi(R) \cap Z_{S_1} = \emptyset \Longleftrightarrow  \psi(S_0) \cap S_1 = \emptyset \quad \& \quad
\frak{Reeb}(\psi(S_0),S_1) = \emptyset.
\ee
\end{lem}

We remark that for a Legendrian pair $(R_0,R_1)$, the condition $\psi(R_0) \cap R_i = \emptyset$
holds for a generic contactomorphism $\psi$ by dimensional reason.
This translation applied to a pair of compact Legendrian submanifolds $(R_0,R_1)$
will be what enables us to study the existence question of translated
intersection points through the study of the moduli space of contact instantons
with Legendrian boundary conditions which is the main analytical framework that
we employ in the present paper.

\section{Hamiltonian-perturbed contact instantons}

Consider a time-dependent function $H = H(t,x): \R \times M \to \R$. Denote by
$X_H$ the associated contact vector field. Then we have $\lambda(X_H) = -H$ by definition.

\begin{defn} Let $(M,\lambda)$ be a contact manifold equipped with a contact form,
let $(R_0,  R_1)$ be a pair of Legendrian submanifolds in $(M,\xi)$.
 We say  a map $u: \R \times [0,1] \to M$ is a $X_H$-perturbed contact instanton trajectory
with Legendrian boundary condition if it satisfies
\be\label{eq:perturbed-contacton}
\begin{cases}
(du - X_H \otimes dt)^{\pi(0,1)} = 0, \\
d(e^{g_H(u)}(u^*\lambda + H\, \gamma)\circ j) = 0
\end{cases}
\ee
together with the boundary condition
\be\label{eq:Legendrian-bdy-condition}
u(\tau,0) \in R_0, \quad u(\tau,1) \in R_1.
\ee
\end{defn}

\begin{rem}
Note that it is easy to see that any asymptotic limit $\gamma_\infty$ of a solution $u$ with finite $\pi$-energy
$E^\pi_{J,H}(u)$ satisfies
$$
(\dot \gamma(t) - X_H(t,\gamma(t)))^\pi = 0.
$$
We refer to \cite[Proposition 3.4]{oh:contacton-Legendrian-bdy} for the
general description of such a curve which is consistent with the asymptotic convergence result
of the equation \eqref{eq:perturbed-contacton}.
\end{rem}

We will also need to consider the full domain-dependent parameterized Hamiltonians of the type
$$
H = H_K(\tau,t, x)
$$
and  consider the 2-parameter family of contactomorphisms
$
\Psi_{s,t}: = \psi_{H^s}^t.
$
Then we have non-trivial $t$-developing Hamiltonian $\Dev_\lambda(t \mapsto \Psi_{(s,t)}) = H^s$.
We also need to consider the $\tau$-developing Hamiltonian
$$
G_K(\tau,t,x) = \Dev_\lambda(\tau \mapsto \Psi_{\tau,t}^K)
$$
where $\Psi_{\tau,t}^K = \Psi_{\chi_K(\tau),t}$.
Then we consider the 2-parameter perturbed contact instanton equation given by
\be\label{eq:K}
\begin{cases}
(du - X_{H_K}(u)\, dt - X_{G_K}(u)\, ds)^{\pi(0,1)} = 0, \\
d\left(e^{g_K(u)}(u^*\lambda + u^*H_K dt + u^*G_K\, d\tau) \circ j\right) = 0,\\
u(\tau,0) \in R,\, u(\tau,1) \in R.
\end{cases}
\ee
where $g_K(u)$ is the function  defined by
\be\label{eq:gKu}
g_K(u)(\tau,t): =  g_{\psi_{H_K}^1(\psi_{H_K}^t)^{-1}}(u(\tau,t))
\ee
for $0 \leq K \leq K_0$.

\subsection{Gauge transformation: autonomous version}
\label{subsec:gauge-transformation}

We first consider two different representations of the trajectory of the ODE $\dot x = X_H(x)$,
one in terms of the \emph{initial point} and the other in terms of the \emph{final point}.
The first one is given by
\be\label{eq:z^qH}
z^q_H(t): = \psi_H^t(q), \quad q \in M
\ee
and the other given by
\be\label{eq:z_pH}
z_p^H(t) = \psi_H^t((\psi_H^1)^{-1}(p)).
\ee
For a given pair $(R_0,R_1)$ of Legendrian submanifolds in $M$
we denote by  $\Omega(R_0,R_1)$ the set of smooth paths
from $R_0$ to $R_1$.

For a given Hamiltonian $H \mapsto \psi$, the above representations
\eqref{eq:z^qH} and \eqref{eq:z_pH} provide two different
one-to-one correspondences
$$
\Phi_H: \bar \ell(t) \mapsto \psi_H^t (\psi_H^1)^{-1}(\bar \ell(t))
$$
and
$$
\Phi_H': \widehat \ell(t) \mapsto \psi_H^t(\widehat \ell(t)):
$$
$\Phi_H$ defines a bijective map
\be\label{eq:PhiH}
\Phi_H: \Omega(\psi_H^1(R_0), R_1) \to \Omega(R_0,R_1)
\ee
and $\Phi_H'$ defines
\be\label{eq:PhiH'}
\Phi_H': \Omega(R_0, (\psi_H^1)^{-1}(R_1)) \to \Omega(R_0,R_1).
\ee
Obviously the composition
$$
(\Phi_H)^{-1}\circ \Phi_H':\Omega(R_0, (\psi_H^1)^{-1}(R_1)) \to \Omega(\psi_H^1(R_0), R_1)
$$
is induced by the diffeomorphism $\phi_H^1:M \to M$ in that
$$
(\Phi_H)^{-1}\circ \Phi_H'(\widehat\ell)(t) = (\psi_H^1\circ \widehat \ell)(t).
$$
\begin{rem}
In symplectic geometry, thanks to the equality  $(\phi_H^1)^*\omega = \omega$,
a (symplectic) Hamiltonian diffeomorphism $\phi_H^1$ induces a filtration
preserving map between two Floer complexes $CF(\phi_H^1(R_0), R_1)$
and $CF(R_0,(\phi_H^1)^{-1}(R_1))$. Therefore one can freely go back and forth between
the two without changing the relevant quantitative invariants. (See \cite{oh:cag}.)
However in the current contact context, the transformation changes the overall
filtration structure between the two. We would like to emphasize that
we will consistently use $\Phi_H$ and its inverse and never use the
transformation $\Phi_H'$ in the present paper. This is because the conformal exponent $e^{g_\psi}$
would appear if $\Phi_H'$ were used. This appearance of conformal exponent is also responsible
for the phenomenon that the  Reeb-untangling energy $e^{\text{\rm trn}}(S_0,S_1)$
is not symmetric.
To keep the importance of this coordinate change in readers' minds, we name and call
$\Phi_H$ and its inverse \emph{gauge transformation}.
\end{rem}

Now let $(M, \lambda, J)$ be a contact triad for the contact manifold $(M, \xi)$, and equip with it the
\emph{contact triad metric}
$$
g = d\lambda(\cdot, J\cdot) +\lambda\otimes\lambda.
$$
We also consider the time-dependent contact triads and  $H = H(t,x)$ be a time-dependent Hamiltonian,
and consider
\be\label{eq:perturbed-contacton-RR}
\begin{cases}
(du - X_H \otimes dt)^{\pi(0,1)} = 0, \quad d(e^{g_H(u)}(u^*\lambda + H\, dt)\circ j) = 0\\
u(\tau,0) \in R, \quad u(\tau,1) \in R
\end{cases}
\ee
where the function $g_H: \R\times [0,1] \to \R$ is defined by
\be\label{eq:gHu}
g_H(t,x) := g_{\psi_H^1 (\psi_H^1)^{-1}}(u(t,x))
\ee
already introduced in \eqref{eq:gHu-intro}.
Now we take the following coordinate change \eqref{eq:u-to-w}
\beastar
\overline u(\tau,t)  :=  \Phi_H^{-1}(u)(t,x) & = & \left (\psi_H^t(\psi_H^1)^{-1}\right)^{-1}(u(\tau,t))\\
& = & \psi_H^1(\psi_H^t)^{-1}(u(\tau,t))
\eeastar
and consider the following particular time-dependent family of $J$'s.

\begin{choice}[CR almost complex structures]
Let $J_0 \in \CJ(\lambda)$. For given contact Hamiltonian $H = H(t,x)$,
we fix a time-dependent CR almost complex structures given by
\be\label{eq:Jt}
J = \{J_t\}_{0 \leq t \leq 1}, \quad J_t:= (\psi_H^t(\psi_H^1)^{-1})_*J_0.
\ee
of $\lambda$-admissible almost complex structures.
\end{choice}
The (translated) intersection theoretic version of the contact instanton complex for the pair
$(R_0, R_1)$ is generated by the set of Reeb chords
$$
\frak{Reeb}(R_0, R_1)
$$
between them and its boundary map is constructed by the moduli space of unperturbed contact instanton
equation \eqref{eq:contacton-Legendrian-bdy}.

On the other hand, the dynamical version of the complex is generated by
some set of solutions of Hamilton's equation $\dot x = X_H(t,x)$
and its boundary map is constructed by the moduli space of \eqref{eq:perturbed-contacton-RR}.
(See Lemma \ref{lem:Ham-Reeb} and Proposition \ref{prop:Reeb-Ham} below for the
details of this correspondence.)
These two frameworks are related by the bijective map $\Phi_H$
via the correspondence
\be\label{eq:gauge1}
\ell(t) = \psi_H^t ((\psi_H^1)^{-1}(\ell'(t))), \quad
u(s,t)=\psi_H^t((\psi_H^1)^{-1}(\overline u(s,t))).
\ee
In particular we have
$$
\overline u = \Phi_H^{-1}(u)
$$
when we regard $u$ as a path on $\Omega(R_0, R_1)$ and $\overline u$ as one on $\Omega(\psi_H^1(R_0),R_1)$.

For the given pair $(R_0,R_1)$ of compact Legendrian submanifolds and a Hamiltonian $H$,
we also consider the family \eqref{eq:Jt}
of $\lambda$-admissible almost complex structures.
A straightforward calculation gives rise to the following.

\begin{lem}\label{lem:Ham-Reeb} Let $J_0 \in \CJ(\lambda)$ and $J_t$ defined as
in \eqref{eq:Jt}. We equip $(\Sigma,j)$ a K\"ahler metric $h$. Let $g_H(u)$ be the
function defined in \eqref{eq:gHu-intro}.
Suppose $u$ satisfies
\be\label{eq:perturbed-contacton-bdy}
\begin{cases}
(du - X_H \otimes dt)^{\pi(0,1)}_J = 0, \quad d(e^{g_H(u)}(u^*\lambda + H\, dt)\circ j) = 0\\
u(\tau,0) \in R_0, \quad u(\tau,1) \in R_1
\end{cases}
\ee
with respect to $J_t$. Then $\overline u$ satisfies
\be\label{eq:contacton-bdy-J0}
\begin{cases}
\delbar^\pi_{J_0} \overline u = 0, \quad d(\overline u^*\lambda \circ j) = 0 \\
\overline u(\tau,0) \in \psi_H^1(R_0), \, \overline u(\tau,1) \in R_1
\end{cases}
\ee
for $J_0$.
\end{lem}

\subsection{Gauge transformation: non-autonomous version}
\label{subsec:nonautonomous-gauge}

In relation to the parametric study of moduli spaces considered in the present paper, we
also need to consider the parametric gauge transformation too.
We make the following specific choice of two-parameter Hamiltonians associated to each time-dependent
Hamiltonian $H = H_K(\tau, t,x)$  with slight abuse of notations.
We then consider the 2-parameter family of contactomorphisms
$
\Psi_{s,t}: = \psi_{H^s}^t.
$
We need to consider both  $t$-developing Hamiltonian $\Dev_\lambda(t \mapsto \Psi_{(s,t)}) = H^s$.
We consider the elongated two parameter family
$$
H_K(\tau,t,x) = \Dev_\lambda (t \mapsto \phi_H^{\chi_K(\tau)t})
$$
and the $\tau$-developing Hamiltonian
$$
G_K(\tau,t,x) = \Dev_\lambda(\tau \mapsto \Psi_{\tau,t}^K)
$$
where $\Psi_{\tau,t}^K = \Psi_{\chi_K(\tau),t}$.
Now we take the following coordinate change \eqref{eq:u-to-w}
\beastar
\overline u(\tau,t)  :=  \Phi_{H_K}^{-1}(u)(t,x) & = & \left (\psi_{H_K}^t(\psi_{H_K}^1)^{-1}\right)^{-1}(u(\tau,t))\\
& = & \psi_{H_K}^1(\psi_{H_K}^t)^{-1}(u(\tau,t))
\eeastar
and consider the following particular time-dependent family of $J$'s.

\begin{choice}[CR almost complex structures of {$J$}] \label{choiceLJst}
Let $J' = J'_{(s,t)}$ be a given smooth family in $\CJ(\lambda)$.
For given contact Hamiltonian $H_K = H_K(s,t,x)$,
we fix a time-dependent CR almost complex structures given by
\be\label{eq:Jt-K}
J = \{J_{(s,t)}\}
_{(s,t) \in [0,1]^2}, \quad J_{(s,t)}:= (\psi_{H_K^s}^t(\psi_{H_K^s}^1)^{-1})_*J'_{(s,t)}.
\ee
of $\lambda$-admissible almost complex structures.
\end{choice}
By construction, $J'$ satisfies
\be\label{eq:bdy-flat}
J'_{(s,t)} \equiv J_0 \in \CJ(\xi)
\ee
near $s = 0, \, 1$.
For the purpose of the present paper, especially for the purpose of establishing
the $C^0$-estimates by the maximum principle, we will need to consider only the case
$$
J' \equiv J_0
$$
but we discuss this general case since it makes no difference for the study of
gauge transformation and the energy estimates.

Through the parametric gauge transformation, we have the following
parametric analog to Lemma \ref{lem:Ham-Reeb}

\begin{lem}\label{lem:Ham-Reeb-K} Let $J' \in \CJ(\lambda)$ and $J_t$ defined as
in \eqref{eq:Jt-K}. We equip $(\Sigma,j)$ a K\"ahler metric $h$. Let $g_H(u)$ be the
function defined in \eqref{eq:gKu}.
Suppose $u$ satisfies
\be\label{eq:perturbed-contacton-bdy-K}
\begin{cases}
(du  -X_{H_K} \otimes dt - X_{G_K} \otimes d\tau)^{\pi(0,1)}_J = 0, \\
d(e^{g_H(u)}(u^*\lambda + H_K \, dt + G_K \, d\tau)\circ j) = 0\\
u(\tau,0) \in R_0, \quad u(\tau,1) \in R_1
\end{cases}
\ee
with respect to $J_t$. Then $\overline u$ satisfies
\be\label{eq:contacton-bdy-J'}
\begin{cases}
\delbar^\pi_{J'} \overline u = 0, \quad d(\overline u^*\lambda \circ j) = 0 \\
\overline u(\tau,0) \in \psi_{H_K}^1(R_0), \, \overline u(\tau,1) \in R_1
\end{cases}
\ee
for $J'$.
\end{lem}

\begin{rem} We will  consider a variation
 the domain-dependent family $J' = J'(\tau,t)$ such that
$$
J'(\tau,t) \equiv J_0
$$
for $|\tau|$ sufficiently large.  See Choice \ref{choice:lambda-J}.
Then we will also consider the 2-parameter perturbed contact instanton
equation given by \eqref{eq:K}.
\end{rem}

\subsection{Coercive elliptic estimates and subsequence convergence}
\label{subsec:coercive}

In this section, we fix a contact triad
$$
(M,\lambda, J)
$$
where $J \in \CJ(\lambda)$ is a $\lambda$-adapted CR-almost complex structure.
All relevant estimate is in terms of the associated triad metric $g = g_{\lambda,J}$
given by
$$
g = d\lambda(\cdot, J\cdot) + \lambda\otimes \lambda.
$$
Then we consider the equation of (unperturbed) contact instantons
\be\label{eq:contacton-Legendrian-bdy}
\begin{cases}
\delbar^\pi w = 0, \quad d(w^*\lambda \circ j) = 0, \\
w(\tau,0) \in R_0, \quad w(\tau,1) \in R_1,
\end{cases}
\ee
for a general Legendrian pair $(R_0,R_1)$ on general contact manifold $(M,\lambda)$.
We will apply the result to the case of $(\psi_H^1(R),R)$ with $w = \overline u$.

We collect some basic results on the analysis of the equation
established in \cite{oh-wang:CR-map1} or in \cite{oh:contacton-Legendrian-bdy}.
We start with the following local $W^{2,2}$-estimates.

\begin{thm}[Theorem 1.3 \cite{oh:contacton-Legendrian-bdy}]\label{thm:local-W22} Let $w: \R \times [0,1] \to M$ satisfy
\eqref{eq:contacton-Legendrian-bdy}.
Then for any relatively compact domains $D_1$ and $D_2$ in
$\dot\Sigma$ such that $\overline{D_1}\subset D_2$ with
$w(\del D_2) \subset R_0$ or $w(\del D_2) \subset R_1$,
 we have
$$
\|dw\|^2_{W^{1,2}(D_1)}\leq C_1 \|dw\|^2_{L^2(D_2)} + C_2 \|dw\|^4_{L^4(D_2)}
$$
where $C_1, \ C_2$ are some constants which
depend only on $D_1$, $D_2$ and $(M,\lambda, J)$ and $C_3$ is a
constant which also depends on $R_0$ or $R_1$.
\end{thm}

Once this $W^{2,2}$-estimate is established, we then proceed the following higher
regularity estimates.

\begin{thm}[Theorem 1.4 \cite{oh-yso:index}]\label{thm:higher-regularity}
 Let $w$ be a contact instanton satisfying
\eqref{eq:contacton-Legendrian-bdy}.
Then for any pair of domains $D_1 \subset D_2 \subset \dot \Sigma$ such that $\overline{D_1}\subset D_2$
 with $w(\del D_2) \subset R_0$ or $w(\del D_2) \subset R_1$, we have
$$
\|dw\|_{C^{k,\alpha}(D_1)} \leq C_\delta(\|dw\|_{W^{1,2}(D_2)})
$$
for some function $C_\delta = C_\delta(r)> 0$ for $r> 0$ and continuous at $r = 0$ with $C(0) = 0$ which
depends on $J$, $\lambda$, $D_1, \, D_2$, $R_0$ or $R_1$,
and $(k,\alpha)$ but independent of $w$.
\end{thm}

Let $w: \R \times [0,1] \to M$ be any smooth map.
As in \cite{oh-wang:CR-map1}, we define the total $\pi$-harmonic energy $E^\pi(w)$
by
\be\label{eq:endenergy}
E^\pi(w) = E^\pi_{(\lambda,J)}(w) = \frac{1}{2} \int |d^\pi w|^2
\ee
where the norm is taken in terms of the given metric $h$ on $\dot \Sigma$ and the triad metric on $M$.

Next we study the asymptotic behavior of contact instantons $w$ satisfying the following
hypotheses.
\begin{hypo}\label{hypo:basic}
Assume $w:\R \times [0,1] \to M$ satisfies the contact instanton equation \eqref{eq:contacton-Legendrian-bdy}
and
\begin{enumerate}
\item $E^\pi_{(\lambda,J;\dot\Sigma,h)}(w)<\infty$ (finite $\pi$-energy);
\item $\|d w\|_{C^0(\dot\Sigma)} <\infty$.
\end{enumerate}
\end{hypo}

For any $w$ satisfying  Hypothesis \ref{hypo:basic}, we associate two
natural asymptotic invariants at $\tau \pm \infty$.
\begin{defn}\label{defn:T-Q} The \emph{asymptotic action} is defined to be
\be
T : =  \lim_{\tau \to \infty} \int_{\{\tau\}\times [0,1]}(w|_{\{0\}\times [0,1]})^*\lambda\label{eq:TQ-T}
\ee
and the \emph{asymptotic charge} is by
\be
- Q : = \lim_{\tau \to \infty} \int_{\{\tau\}\times [0,1]}(w|_{\{0\}\times [0,1] })^*\lambda\circ j.\label{eq:TQ-Q}
\ee
provided they exist.
(Here we only look at the positive end. The case of negative end is similar.)
\end{defn}
The above finite $\pi$-energy and $C^0$ bound hypotheses imply
\be\label{eq:hypo-basic-pt}
\int_{[0, \infty)\times [0,1]}|d^\pi w|^2 \, d\tau \, dt <\infty, \quad \|d w\|_{C^0([0, \infty)\times [0,1])}<\infty.
\ee
\begin{rem}
In general there is no reason why these limits exist and even if the limits exist, they may also depend on the choice of
subsequences under Hypothesis \ref{hypo:basic}.
In the closed string case, \cite{oh-wang:CR-map1} shows that the asymptotic charge $Q$ may not vanish which
is the key obstacle to the compactification and the Fredholm theory of contact instantons for $Q \neq 0$.
\end{rem}

As in \cite{oh-wang:CR-map1}, \cite{oh:contacton-Legendrian-bdy}, we call $T$ the \emph{asymptotic contact action}
and $Q$ the \emph{asymptotic contact charge} of the contact instanton $w$ at the given puncture.

\begin{thm}[Vanishing asymptotic charge; Theorem 6.7
\cite{oh:contacton-Legendrian-bdy}]\label{thm:subsequence-convergence}
Let $w:[0, \infty)\times [0,1]\to M$ satisfy the contact instanton equations \eqref{eq:contacton-Legendrian-bdy}
and Hypothesis \ref{hypo:basic}.
Then for any sequence $s_k\to \infty$, there exists a subsequence, still denoted by $s_k$, and a
massless instanton $w_\infty(\tau,t)$ (i.e., $E^\pi(w_\infty) = 0$)
on the cylinder $\R \times [0,1]$ that satisfies the following:
\begin{enumerate}
\item On any given compact subset $K \subset [0,\infty)$, we have
$$
\lim_{k\to \infty}w(s_k + \tau, t) = w_\infty(\tau,t)
$$
in the $C^l(K \times [0,1], M)$ sense for any $l$.
\item $w_\infty$ has $Q = 0$ and the formula $w_\infty(\tau,t)= \gamma(T\, t)$, where $\gamma$ is some Reeb chord
joining $R_0$ and $R_1$ with action $T$.
\end{enumerate}
\end{thm}
We mention that the asymptotic action $T$ could be either positive, negative or zero, i.e.,
the pair $(T,\gamma)$ is a iso-speed Reeb chord in the sense of Definition \ref{defn:Reeb-chords-2}.

We also need these estimates for the case of domain-dependent contact triads $(M,\lambda_z,J_z)$ for $z \in \dot \Sigma$.
The relevant a priori estimates for the general perturbed contact instantons including such a case 
are established in \cite{oh:perturbed-contacton-bdy} which extends the results from \cite{oh-yso:index}
to the perturbed context.

\subsection{Obstruction to the existence of translated intersections}

We apply the above discussion to the case $R_0 = R_1 = R$
and then the subsequence convergence result, Theorem \ref{thm:subsequence-convergence},
to the contact instantons $w = \overline u$.
In the rest of the present paper, we will investigate the simplest dynamical entanglement question
for the two-component link $(\psi(R),R)$, which concerns existence of Reeb chords from $\psi(R)$ to $R$.
For the purpose of the present paper, from now on, we will make the following hypothesis
which can be achieved by making an arbitrarily $C^\infty$ small perturbation of the given
Hamiltonian.
\begin{hypo} We consider the Hamiltonian $H = H(t,x)$ such that
\be\label{eq:no-intersection}
\psi_H^1(R) \cap R = \emptyset.
\ee
\end{hypo}

The following proposition illustrates how the analytical study of the above
considered perturbed Cauchy-Riemann equation gives rise to an existence result
of Reeb chords between $R$ and $R$. It is the converse of Lemma \ref{lem:Ham-Reeb}
which holds for the case $R_0 =R_1$.

\begin{prop}
\label{prop:Reeb-Ham} Let $u:[0, \infty)\times [0,1]\to M$ be a smooth map satisfying
the boundary condition
$$
 u(\tau,0), \quad u(\tau,1) \in  R.
$$
Denote by $\overline u$ the map defined in \eqref{eq:u-to-w}.
Suppose that $\overline u$ is a contact instanton of finite $\pi$-energy
with uniform $C^1$-bound satisfying the boundary condition
$$
\overline u(\tau,0) \in \psi_H^1(R), \quad \overline u(\tau,1) \in R.
$$
Then the followings hold:
\begin{enumerate}
\item There exist a pair of translated points $(\eta_\pm, x_\pm)$
with
$$
x_\pm \in \psi_H^1(R), \quad \phi_{R_\lambda}^\eta(x_\pm) \in R
$$
such that the asymptotic Reeb chords of $\overline u$
have the form
$$
t \mapsto \phi_{R_\lambda}^{t\eta}(x_\pm).
$$
\item We have a contact Hamiltonian trajectory $\gamma_\pm$ from $R$ to $\psi_H^1(R)$ given by
\be\label{eq:gamma-pm}
\gamma_\pm(t) =  z^H_{x_{\pm}}(t) = \psi_H^t\left((\psi_H^1)^{-1}(x_\pm)\right).
\ee
(See \eqref{eq:z_pH}.)
\end{enumerate}
\end{prop}
\begin{proof} Statement (1) is the consequence of \cite[Theorem 1.6]{oh:contacton-Legendrian-bdy}.
(See Theorem \ref{thm:subsequence-convergence} below for the precise statement.)

Statement (2) is just a translation of Statement (1) by the definition of translated points $(\eta_\pm,x_\pm)$
associated the asymptotic Reeb chords at $\pm \infty$ respectively.
\end{proof}

The following diagram describes the relationship between translated intersection points and
Reeb chords between $R$ and $\psi(R)$ for a Hamiltonian $H \mapsto \psi$:
\bea\label{eq:translated-diagram}
\xymatrix{
R \ar@/_1.5pc/[rr]|{\psi_H^1}
&& \psi_H^1(R) \ar@/_1.5pc/[ll]|{\phi_{R_\lambda}^{\eta}}
}
\\
{} \nonumber \\
{\text{\rm Figure 1. \, Translated intersection}} \nonumber
\eea

Before delving into the analysis of contact instantons, we would like to describe 
an obstruction to the existence of finite energy solution to \eqref{eq:perturbed-contacton-bdy}, which
s the contact analog to a similar obstruction appearing in \cite[Lemma 2.2]{oh:mrl}.

Assuming the energy estimates given in Theorem \ref{thm:energy-estimate-intro}, we
combine Lemma \ref{lem:Ham-Reeb} and Proposition \ref{prop:Reeb-Ham},
we have the following nonexistence of solutions to \eqref{eq:perturbed-contacton-bdy}.
It will play an
important role as an obstruction to compactness of the moduli space of
finite energy solutions of a suitably cut-off version of Floer trajectory equation.

\begin{upshot}[Obstruction to existence]\label{upshot:obstruction}
Suppose the following:
\begin{enumerate}
\item $\psi_H^1(R) \cap R = \emptyset$,
\item  $\|H\| < T_\lambda(M,R)$. 
\end{enumerate}
Then the equation \eqref{eq:perturbed-contacton-bdy}
has no solution of finite $\pi$-energy.
\end{upshot}
Here we would like to emphasize that the first condition is a generic phenomenon: it can be always
achieved by a $C^\infty$-small perturbation of any given Hamiltonian and so without destroying the second condition.
Note that the first condition removes the possibility of appearance of constant (iso-speed) Reeb chord.

We will establish the two energy estimates for \eqref{eq:K} given
 in Theorem \ref{thm:energy-estimate-intro}
under the standing hypothesis $\|H\| < T_\lambda(M,R)$ later in Sections \ref{sec:pienergy-bound},
\ref{sec:C1-estimate} respectively.

\section{Maximum principle and $C^0$ estimates on tame contact manifolds}

In this paper, our analysis of (perturbed) contact instantons will be
performed on a general class of contact manifolds which may not be necessarily compact.
As in other geometric analysis problems such as that of pseudoholomorphic curves,
compactness of the ambient manifolds is not needed as long as relevant $C^0$ confinement results
can be achieved.

We first recall readers the definitions of contact $J$ quasi-convexity 
(Definition \ref{defn:J-convexity}) and 
of tame contact manifolds from Definition \ref{defn:tame-intro} given
in the introduction of the present paper.

Using this, we introduce the following class of contact manifolds for which
the $C^0$ estimates for the (perturbed) contact instantons will be available.

\begin{defn}[Tame contact manifolds]\label{defn:tame}
Let $(Q,\xi)$ be a contact manifold, and let $\lambda$ be a contact form of $\xi$.
\begin{enumerate}
\item  We say $\lambda$ is \emph{tame on $U$}
if $(M,\lambda)$ admits a pair $(\psi, J)$ of a \emph{proper} 
$\lambda$-adapted CR almost complex structure $J$
and a $\lambda$-tame contact $J$  plurisubharmonic exhaustion function $\psi$ on $U$.
\item We call an end of $(M,\lambda)$ \emph{tame} if $\lambda$ is tame on the end of $M$.
\end{enumerate}
We say an end of contact manifold $(M,\xi)$ (resp. $(M,\xi)$) is tame if it admits contact form $\lambda$
that is tame on the end (resp. at infinity) of $M$.
\end{defn}

\begin{exm}
\begin{enumerate}
\item  Obviously any contact form $\lambda$ on any compact contact manifold is tame.
\item
We show in \cite{oh:shelukhin-conjecture} that the contact form
$$
\mathscr A = - e^{\eta/2} \pi_1^*\lambda + e^{-\eta/2} \pi_2^*\lambda
$$
on $Q \times Q \times \R$ is tame on $\{\eta> 0 \}$ for any compact contact manifold $(Q,\lambda)$
by proving that the function $e^\eta$ associated to the coordinate function $\eta$
$\widetilde J$  plurisubharmonic on $\{\eta> 0 \}$ with respect to
some natural class of $\mathscr A$-adapted CR almost complex structures $\widetilde J$.
\item Any one-jet bundle $J^1B$ is tame with the standard contact form $dz - pdq$ for a compact manifold $B$:
In fact, the coordinate function $z$  and the norm square $|p|^2_g$
(with respect to the Sasakian metric) is a contact $J$-convex Reeb tame exhaustion function.
Then just take $\psi = |z| + |p|^2_g$.
(See \cite{oh-yso:spectral} for its proof and usage in the study of Legendrian spectral invariants.)
\end{enumerate}
\end{exm}

The following theorem is the reason why we introduce the class of tame contact manifolds.

\begin{thm}\label{thm:C0estimate} Let $(M,\lambda)$ be a contact manifold and consider a
$\lambda$  pseudoconvex pair $(\psi,J)$.
Let $\vec R = \{R_i\}$ be any Legendrian link consisting of compact Legendrian submanifolds $R_i$.
Suppose that there exists a compact subset $K \subset M$ containing
$\{R_i\}$ and $\{\gamma_i\}$. Then for any contact instanton $w:\dot \Sigma \to M$  satisfying
\be\label{eq:contacton-bdy}
\begin{cases}
\delbar^\pi w = 0, \quad d(w^*\lambda \circ j) = 0, &\\
w(\overline z_iz_{i+1}) \subset R_i, \quad & i = 0, \cdots, k\\
w(\infty_i) = \gamma_i, \quad & i = 0, \cdots, k
\end{cases}
\ee
we have $\Image w \subset K$.
\end{thm}
\begin{proof}
By the definition of tame contact manifolds, there exists a compact subset $K' \subset M$ such that
on $M \setminus K'$ $\psi$ satisfies
\be\label{eq:lambda-tame}
\CL_{R_\lambda}d \psi = 0,
\ee
and
\bea
-d(d \psi \circ J) + k d\eta \wedge \lambda & = & h\, d\lambda
\quad \text{\rm on $\xi$}, \label{eq:quasi-pseudoconvex2}\\
R_\lambda \rfloor d(d\psi \circ J) & = & g\, d\psi \label{eq:Reeb-flat2}
\eea
for some smooth one-form $\beta$ and smooth functions $k, \, g$ and $h \geq 0$ on $M \setminus K$.
By enlarging $K$, we may assume that $K$ itself satisfies the above property.

Let $w$ be a solution to \eqref{eq:contacton-bdy}. By the standing hypothesis, it will be enough to
prove that the maximum $\psi \circ w$ cannot be achieved at a point in $M \setminus K$.

Suppose to the contrary that a maximum is achieved at $z_0$ with $w(z_0) \in M \setminus K$.
We decompose
$$
dw = d^\pi w + w^*\lambda \otimes R_\lambda
$$
which we often just write $dw = d^\pi w + w^*\lambda \, R_\lambda$
as a $T^*M$-valued one-form on $\dot \Sigma$ in the calculation below and henceforth.

We compute
\bea\label{eq:DeltadA}
\Delta (\psi \circ w)\, dA & = & -d(d(\psi\circ w)\circ j) = -d(d\psi \circ dw)\circ j) \nonumber\\
& = & -d(d\psi (d^\pi w + w^*\lambda \otimes R_\lambda(w))\circ j) \nonumber\\
& = & -d(d \psi (d^\pi w \circ j) -d(d\psi (w^*\lambda \otimes R_\lambda(w))\circ j).
\eea
Using the $\lambda$-tameness of $\psi$ and the closedness of $w^*\lambda \circ j$, we compute
the second term
\beastar
d(d\psi (w^*\lambda \otimes R_\lambda(w))\circ j)& = & d(w^*(d\psi(R_\lambda))\, (w^*\lambda\circ j)) \\
& = &  w^*(d(d\psi(R_\lambda))) \wedge w^*\lambda\circ j.
\eeastar
This vanishes by \eqref{eq:lambda-tame} since
$$
d(d\psi(R_\lambda)) = \CL_{R_\lambda} d\psi.
$$

For the first term of \eqref{eq:DeltadA}, using the equation $J d^\pi w = d^\pi w j$ and the
properties
$$
J(R_\lambda) = 0, \quad \Image J = \xi,
$$
of $CR$ almost complex structure $J$, we compute
\bea\label{eq:ddpipsi}
-d(d \psi (d^\pi w \circ j)) & = & -d(d \psi(J d^\pi w)) = -d(d \psi\circ J d w) \nonumber\\
& = &  -d(w^*(d \psi \circ J)) = -w^*(d(d \psi \circ J)).
\eea
Therefore, for any tangent vector $v \in T_{z_0}\dot \Sigma$, we have
\be\label{eq:ddpipsi(v,jv)}
-d(d \psi (d^\pi w \circ j))(v,jv) = -w^*(d(d \psi \circ J))(v,jv)
= - d( d\psi \circ J)(dw(v), dw(jv)).
\ee
Again we decompose $dw = d^\pi w + w^*\lambda \, R_\lambda(w)$ and evaluate
\bea\label{eq:ddpsiJ}
&{}& d(d \psi \circ J)(dw(v), dw(jv)) \nonumber\\
& = & d(d\psi \circ J)(d^\pi w(v) + w^*\lambda(v) \, R_\lambda(w(z_0)),
d^\pi w(jv) + w^*\lambda(j v) \, R_\lambda(w(z_0))) \nonumber\\
& = & d(d \psi \circ J)(d^\pi w(v), d^\pi w(j v))
+  d(d \psi \circ J)(d^\pi w(v),w^*\lambda(jv) R_\lambda(w(z_0))) \nonumber\\
&{}& \quad
+  d(d\psi \circ J)(w^*\lambda(v) R_\lambda(w(z_0)), d^\pi w(jv)) \nonumber\\
&{}& \quad
+ d(d\psi \circ J)(w^*\lambda(v) R_\lambda(w(z_0)), w^*\lambda(jv) R_\lambda(w(z_0))).
\eea
Now we evaluate the last three terms. Obviously the fourth term vanishes.

\begin{lem}\label{lem:sum-vanish} The sum of the second and the third vanishes.
\end{lem}
\begin{proof} We evaluate the first term
\beastar
&{}& d(d \psi \circ J)(d^\pi w(v),w^*\lambda(jv) R_\lambda(w(z_0)))\\
& = & - w^*\lambda(jv)\,  \left(R_\lambda(w(z_0)) \rfloor d(d \psi \circ J)\right)(d^\pi w(v)) \\
& = & - w^*\lambda(jv)\, g\circ w(z_0) d\psi(dw(v) - w^*\lambda(v) R_\lambda)\\
& = & - w^*\lambda(jv)\, g \circ w(z_0) d(\psi \circ w)(v) +  w^*\lambda(jv)\, g\circ w(z_0)\,  w^*\lambda(v) \, d\psi(R_\lambda)
\eeastar
by the condition \eqref{eq:Reeb-flat2} for the second equality.
Similar computation leads to
\beastar
&{}&
 d(d\psi \circ J)(w^*\lambda(v) R_\lambda(w(z_0)), d^\pi w(jv))\\
 & = &
w^*\lambda(v)\, g \circ w(z_0) d(\psi \circ w)(jv) - w^*\lambda(v)\, g\circ w(z_0) \, w^*\lambda(jv) \, d\psi(R_\lambda).
\eeastar
By summing them over, we get
$$
- w^*\lambda(jv)\, g \circ w(z_0)\,  d(\psi \circ w)(v) + w^*\lambda(v)\, g \circ w(z_0)\,  d(\psi \circ w)(jv)
$$
which vanish at $z_0$ since $z_0$ is a critical point of $\psi \circ w$.
\end{proof}

Finally, using the definition \eqref{eq:quasi-pseudoconvex2} of quasi-plurisubharmonicity of $\psi$, we rewrite the first term above into
\beastar
- d(d\psi \circ J)(d^\pi w(v), d^\pi w(jv)) & = & (-k d\psi \wedge \lambda) 
+ h d\lambda )(d^\pi w(v), d^\pi w(jv)) \\
& = & h d\lambda (d^\pi w(v), d^\pi w(jv)) = h d\lambda (d^\pi w(v), J d^\pi w(v)) \geq 0
\eeastar
for all $v \in T_{z_0}\dot \Sigma$, where the last positivity follows
by \eqref{eq:quasi-pseudoconvex2}. This proves
$$
\Delta (\psi \circ w) \geq 0.
$$

Once we have this differential inequality, the classical strong maximum principle
applies to get the inequality
$$
\sup_{\dot \Sigma} |\psi \circ w| = \max\left \{\sup |\psi \circ w|,
\sup \{|\psi\circ \gamma_i|\}_{i=0}^k \right\}
$$
(See \cite[Theorem 3.1]{GT} for example). Since $w(\del \dot\Sigma) \subset R$ and $\gamma_i$ are given,
this proves
$$
\sup_{\dot \Sigma} |\psi \circ w| \leq \max \left\{ \sup |\psi|_R|,
\sup \{|\psi\circ \gamma_i|\}_{i=0}^k\}\right\}.
$$
By setting
$$
C = \max \left\{ \sup |\psi|_R|,
\sup \{|\psi\circ \gamma_i|\}_{i=0}^k\}\right\}, \quad K = \psi^{-1}([-C,C])
$$
we have now finished the proof.
\end{proof}

\section{The $\pi$- energy identity}

Let $H= H(t,x)$ be a given Hamiltonian.
It turns out that the correct definition of the $\pi$-energy for the perturbed contact instanton is the following.

\begin{defn}[The $\pi$-energy of perturbed contact instanton]
Let $u: \R \times [0,1] \to M$ be any smooth map. We define
$$
E^\pi_{J,H}(u): = \frac12 \int e^{g_H(u)} |(du - X_H(u) \otimes dt)^\pi|^2_J \,
$$
call it the \emph{off-shell $\pi$-energy}.
\end{defn}

Now we apply the gauge transformation $\Phi_H^{-1}$ to $u$ and define $\overline u : = \Phi_H^{-1}(u)$
which has the expression
\be\label{eq:ubar}
\overline u(\tau,t) = (\psi_H^t(\psi_H^1)^{-1})^{-1}(u(\tau,t))
= \psi_H^1(\psi_H^t)^{-1}(u(\tau,t)).
\ee
The following identity justifies the presence of weight $e^{g_H(u)}$ in this definition.
(Also see Remark \ref{rem:weight} for further justification.)

\begin{prop}\label{prop:EpiJH=Epi} Let $u: \R \times [0,1] \to M$ be any smooth map and $\overline u$ be
as above. Then
\be\label{eq:EpiJH=Epi}
E^\pi_{J,H}(u) = E_{J'}^\pi(\overline u).
\ee
\end{prop}
\begin{proof}  We first note that $u$ satisfies
$$
0 = \delbar^\pi \overline u(\del_\tau) = \left(\frac{\del \overline u}{\del \tau}\right)^\pi +
J'\left(\frac{\del \overline u}{\del t}\right)^\pi.
$$
We compute $|d^\pi \overline u(\del_\tau)|_J$ and $|d^\pi \overline u(\del_t)|_J$
separately. By definition, we have
$$
u(\tau,t) = \psi_H^t(\psi_H^1)^{-1}(\overline u(\tau,t))).
$$
Then
$$
d u(\del_t)= \frac{\del u}{\del t}
= d\left(\psi_H^t(\psi_H^1)^{-1}\right)\left(\frac{\del \overline u}{\del t}\right) + X_H(u(\tau,t))
$$
and hence we have
$$
\frac{\del \overline u}{\del t} = \left(d(\psi_H^t(\psi_H^1)^{-1})\right)^{-1}
\left(\frac{\del u}{\del t} - X_H(u(\tau,t))\right).
$$
And more easily, we compute
$$
\frac{\del \overline u}{\del \tau} 
= \left(d(\psi_H^t(\psi_H^1)^{-1})\right)^{-1}\left(\frac{\del u}{\del \tau}\right).
$$

We recall that  $J(\xi) \subset \xi$  for any $\lambda$-adapted CR-almost complex structure $J$,
and
$$
d(\psi_H^t(\psi_H^1)^{-1}) (\xi) = \xi
$$
since $\psi_H^t(\psi_H^1)^{-1} $ is a contactomorphism.
For the simplicity of notations, we write
\be\label{eq:psit-gt}
\psi_t: = (\psi_H^t(\psi_H^1)^{-1})^{-1} = \psi_H^1(\psi_H^t)^{-1}, \quad g_t: = g_{\psi_t}
\ee
for the conformal exponent of the contactomorphism $\psi_t$ in the following calculation.

Then we have
\beastar
&{}&
\left|\left(\frac{\del \overline u}{\del t}\right)^\pi\right|_{J'}^2
= \left|d\psi_t\left(\frac{\del u}{\del t} - X_H(u(\tau,t))\right)^\pi\right|_{J'}^2\\
& = & d\lambda \left( (d\psi_t\left(\frac{\del u}{\del t} - X_H(u(\tau,t))^\pi\right),
J' (d\psi_t\left(\frac{\del u}{\del t} - X_H(u(\tau,t)\right)^\pi\right)\\
& = & \psi_t^*d \lambda \left(\left(\frac{\del u}{\del t} - X_H(u(\tau,t))\right)^\pi,
J_t \left(\frac{\del u}{\del t} - X_H(u(\tau,t))\right)^\pi\right) \\
& = & d(e^{g_t} \lambda)  \left(\left(\frac{\del u}{\del \tau} - X_H(u(\tau,t))\right)^\pi,
J_t \left(\frac{\del u}{\del t} - X_H(u(\tau,t))\right)^\pi\right) \\
& = & e^{g_t}(d\lambda + dg \wedge \lambda) \left(\left(\frac{\del u}{\del t} - X_H(u(\tau,t))\right)^\pi,
J_t\left( \frac{\del u}{\del t} - X_H(u(\tau,t))\right)^\pi \right) \\
& = & e^{g_t} \left|\left(\frac{\del u}{\del t} - X_H(u(\tau,t))\right)^\pi\right|_{J_t}^2.
\eeastar
Here we used the vanishing
$\lambda((\cdot)^\pi) = 0$
for the penultimate equality.

More easily, we also derive
$$
\left|\left(\frac{\del \overline u}{\del \tau}\right)^\pi\right|_{J'}^2
=  e^{g_t} \left|\left(\frac{\del u}{\del \tau}\right)^\pi\right|_{J_t}.
$$
By adding the two, we have finished the proof.
\end{proof}

\begin{rem}\label{rem:weight}
\begin{enumerate}
\item
Unless there were aforementioned exponential weight-factor, it would be
possible to achieve the kind of a priori estimates neither for the $\pi$-energy
nor for the $\lambda$-energy we are going to define in Part \ref{part:bubbling}.
In this regard, the presence of the exponential weight factor in the
definition of $\pi$-energy is essential.
\item In fact, the most natural explanation of the appearance of this weighting factor
can be given in terms of the contact mapping tori construction. See
\cite[Section 2.1]{oh-savelyev}. We will elaborate this point of view  when
we consider contact fibration elsewhere.
\end{enumerate}
\end{rem}

The following proposition is one of the key
energy estimates for the solutions $u$ of \eqref{eq:perturbed-contacton}
in terms of the geometry of Legendrian boundary conditions and its asymptotic chords.

\begin{prop}\label{prop:pienergy-dlambda} Let $\psi_t$ be as in \eqref{eq:psit-gt}.
Let $u$ be any finite energy solution of \eqref{eq:perturbed-contacton} with the
limits
$$
\gamma_\pm(t) := \lim_{\tau \to \pm \infty} u(\tau,t)
$$
and
$\overline u$ be as above. Consider the paths given by
$$
\overline \gamma_\pm(t) = \psi_t(\gamma_\pm(t)).
$$
Then $\overline \gamma_\pm$ are Reeb chords from $\psi_H^1(R_0)$ to $R_1$ and satisfy
\be\label{eq:pi-energy}
E^\pi_{J,H}(u)
= \int_0^1 (\overline \gamma_+)^*\lambda - \int_0^1 (\overline \gamma_-)^*\lambda.
\ee
\end{prop}
\begin{proof} The first statement immediately follows from definition of the gauge transformation $\Phi_H$
and by the subsequence convergence theorem, Theorem \ref{thm:subsequence-convergence}.

For the energy identity,  it is enough to compute $E^\pi_J(\overline u)$
by Proposition \ref{prop:EpiJH=Epi}.
 Similarly as in the proof of Proposition \ref{prop:EpiJH=Epi}, we compute, this time using the equation
 $Jd \overline{u} (\del_\tau) = d\overline u(\del t)$,
\beastar
&{}&
\int_{-\infty}^\infty \int_0^1 |d^\pi \overline u(\del_\tau)|_{J'}^2\, dt\, d\tau \\
&=& \int_{-\infty}^\infty \int_0^1 d\lambda(d\overline u(\del_\tau), J' d\overline u(\del_\tau))\, dt\, d\tau \\
& = & \int_{-\infty}^\infty \int_0^1 d\lambda(d\overline u(\del_\tau), d\overline u(\del_t))\, dt\, d\tau =  \int (\overline u)^*d\lambda \\
& = & \left(\int_0^1 ({\overline \gamma_+})^*\lambda - \int_0^1 ({\overline \gamma_-})^*\lambda \right) \\
&{}& \quad + \int_{-\infty}^\infty \lambda\left(\frac{\del \overline u}{\del \tau}(\tau,0) \right)\, d\tau -
\int_{-\infty}^\infty \lambda\left(\frac{\del \overline u}{\del \tau}(\tau,1) \right)\, d\tau\\
& = & \int_0^1 ({\overline \gamma_+})^*\lambda - \int_0^1 ({\overline \gamma_-})^*\lambda.
\eeastar
Here the last equality follows from the Legendrian boundary condition
$$
\overline u(\tau,0) \in \psi_H^1(R), \quad \overline u(\tau,1) \in R
$$
for $i = 0, \, 1$. This finishes the proof.
\end{proof}

%The following lemma is an important criterion for non-appearance of 
%isospeed Reeb chords with non-zero period.
%
%\begin{lem}\label{lem:gamma+=const}   
%Suppose a solution $u$ to \eqref{eq:HGK}  satisfies the following:
%\begin{enumerate}
%\item $u$ has finite energy $E_{J,H}(u) < \infty$ and
%$E^\pi_{J,H}(u) < T(M,\lambda;R)$,
%\item $\gamma_- := u(-\infty)$ is  constant, i.e., $T_- = \int (\gamma_-)^*\lambda = 0$.
%\end{enumerate}
%Then $u$ has asymptotic orbit at $\tau = + \infty$
%that must have $ \int_{\gamma_+} \lambda = 0$.
%\end{lem} 
%\begin{proof} Since $X_{H_K} = 0 X_{G_K}$  outside $[-K-1,=K] \times [0,1]$,
%$u$ satisfies 
%$$
%\delbar^\pi w= 0, \quad d(w^*\lambda \circ j) = 0
%$$
%thereon. Therefore \eqref{eq:pi-energy} is reduced to
%$$
%E^\pi_{J,H}(u) = \int_0^1 (\gamma_+)^*\lambda.
%$$
%By the hypothesis $E^\pi_{J,H}(u) < T(M,J;R)$, this is a contradiction unless $\gamma_+$ is 
%a constant curve. This finishes the proof.
%\end{proof}

\part{Study of Reeb-untangling energy: Proof of Theorem \ref{thm:shelukhin-intro}}
\label{part:shelukhin}

The discussion given in the previous section applies
more generally when the pair $(J,H)$ depends on
the domain parameter.

In this part, we set-up the deformation-cobordism framework of parameterized
moduli space of perturbed contact instantons by
adapting a similar parameterized Floer moduli spaces
appearing in the study of the displacement energy of
compact Lagrangian submanifolds in \cite{oh:mrl}.

\section{Cut-off Hamiltonian-perturbed contact instanton equation}
\label{sec:cut-off}

We will consider the two-parameter family of CR-almost complex structures and Hamiltonian functions:
$$
J = \{J_{(s,t)}\}, \, H= \{H_t^s\} \; \textrm{for} \;\; (s,t) \in
[0,1]^2.
$$
We write
$$
H_t^s(x): = H(s,t,x)
$$
in general for a given two-parameter family of functions $H = H(s,t,x)$.
Note that $[0,1]^2$ is a compact set and so $J,\, H$ are compact
families. We always assume the $(s,t)$-family $J$ or $H$ are
constant near $s = 0, \, 1$. 

Then we take the $K$-family of cut-off functions
$\chi_K$ introduced
\bea\label{eq:chiK}
\chi_K(\tau) & = & 
\begin{cases} 1 \quad & |\tau| \leq  K \\
0 \quad & |\tau| \geq K+1
\end{cases} 
\nonumber\\
\chi_K'(\tau) & \geq & 0 \quad \text{\rm for }\, \tau \in [-K -1, -K] \cup [K,K+1].
\eea

\subsection{Setting up the parameterized moduli space}

Knowing that $H_K \equiv 0$ when $|\tau|$ is sufficiently large and
having Proposition \ref{prop:obstruction} in our disposal which we will prove later,
we consider a one-parameter family of domains defined as follows.

Consider the following capped semi-infinite cylinders
\beastar
\Theta_- & = &\{ z\in \C  \mid \vert z\vert \le 1 \} \cup
\{ z\in \C \mid \text{Re} z \ge 0,\vert
\text{Im} z\vert \le 1\}
\\
\Theta_+ & = & \{ z\in \C \mid \text{Re} z \le 0,\vert
\text{Im} z\vert \le 1\}
\cup  \{ z\in \C  \mid \vert z\vert \le 1 \}.
\eeastar

We will fix the $K_0$ once and for all and consider $K$ with $0 \leq K \leq K_0$.
(See Proposition \ref{prop:obstruction} below for its required condition to satisfy.)
For such given $K_0$, we define the spaces
\bea\label{eq:Theta+-}
\Theta_{-,K_0+1} & := & \{ z \in \Theta_- \mid \operatorname{Re} z \le K_0+1\}, \nonumber\\
\Theta_{+,K_0+1} & := & \{ z \in \Theta_- \mid \operatorname{Re} z \ge -K_0-1\}.
\eea
We glue the three spaces
$$
\Theta_{-,K_0+1}, \quad [-K_0+1,K_0+1] \times [0,1], \quad \Theta_{+,K_0+1}
$$
subdomains of $\R \times [0,1]$, by making the identification
\beastar
(0,t) \in \Theta_{-,K_0+1} & \longleftrightarrow &  (-K_0-1,t) \in [-K_0-1,K_0+1] \times [0,1],\\
 (0,t) \in \Theta_{+,K_0+1} & \longleftrightarrow & (K_0+1,t) \in [-K_0-1,K_0+1] \times [0,1],
\eeastar
respectively. We denote the resulting domain as
\be\label{eq:Z-K}
\Theta_{K_0+1}: = \Theta_- \#_{K_0+1} (\R \times [0,1]) \#_{K_0+1} \Theta_+ \subset \C
\ee
and equip it with the natural complex structure induced from $\C$.
(See \cite[Figure 8.1.2]{fooo:book1} for the visualization of this domain.)
We can also decompose $\Theta_{K_0+1}$ into the union
\be\label{eq:Theta-K0+1}
\Theta_{K_0+1} := D^- \cup  [-2K_0 -1, 2K_0+1] \cup  D^+
\ee
where we denote
\be\label{eq:D+-}
D^\pm  = D^\pm_{K_0}:= \{ z\in \C  \mid \vert z\vert \le 1, \, \pm \text{\rm Im}(z) \leq 0 \} \pm (2K_0+1)
\ee
respectively.

\begin{rem}[Domain {$\Theta_{K_0+1}$}]\label{rem:ThetaK0+1}
We highlight the obvious fact that the domain $\Theta_{K_0+1}$ is a
compact disk-like domain, where any closed one-form is exact. In particular any
contact instanton $w$ on $\Theta_{K_0+1}$ can be lifted to a
pseudoholomorphic curve $\widetilde u = (w,f)$ for a uniquely defined function
$f: \Theta_{K_0+1} \to \R$ satisfying
$$
df = w^*\lambda \circ j
$$
where $w^*\lambda \circ j$ is closed by the defining equation of contact instantons.
\end{rem}

We make the following specific choice of two-parameter Hamiltonians associated to each time-dependent
Hamiltonian $H = H(t,x)$  with slight abuse of notations. \emph{This choice of 2-parameter
Hamiltonian is a special property in that the $\pi$-energy estimates of our interest
does not involve the conformal factor $e^{g_H}$.}
(See \cite{oh-yso:spectral} for a similar practice, especially Remark 1.13 in ibid. for
the explanation on its importance.)

\begin{choice} Take the family $H = H(s,t,x)$ given by
\be\label{eq:sH}
H^s(t,x) = \Dev_\lambda(t\mapsto \psi_H^{st})
\ee
\end{choice}
We consider the 2-parameter family of contactomorphisms
$
\Psi_{s,t}: = \psi_{H^s}^t.
$
Obviously we have the $t$-developing Hamiltonian $\Dev_\lambda(t \mapsto \Psi_{(s,t)}) = H^s$.
We then consider the elongated two parameter family
\be\label{eq:HK}
H_K(\tau,t,x) := \Dev_\lambda\left (t \mapsto \psi_H^{\chi_{K(\tau)}t}\right) 
\ee
and write the $\tau$-developing Hamiltonian
\be\label{eq:GK}
G_K(\tau,t,x) := \Dev_\lambda\left(\tau \mapsto \psi_H^{\chi_{K(\tau)}t}\right).
\ee

Then we consider the 2-parameter perturbed contact instanton equation for a map $u: \Theta_{K_0+1} \to M$ given by
\be\label{eq:HGK-Theta}
\begin{cases}
(du - X_{H_K}(u)\, dt - X_{G_K}(u)\, ds)^{\pi(0,1)} = 0, \\
d\left(e^{g_K(u)}(u^*\lambda + u^*H_K dt + u^*G_K\, d\tau) \circ j\right) = 0,\\
u(\del \Theta_{K_0+1}) \subset R
\end{cases}
\ee
where $g_K(u)$ is the function on $\Theta_{K_0+1}$ defined by
\eqref{eq:gKu} for $0 \leq K \leq K_0$. We note that
if $|\tau| \geq K +1$, the equation becomes
\be\label{eq:contacton}
\delbar^\pi u = 0, \, \quad d(u^*\lambda \circ j) = 0.
\ee
It will be useful to introduce the full version of vector-valued
one-form on $\Theta_{K_0+1}$ that is familiar in the study of
Hamiltonian fibrations in symplectic topology.
(See \cite{seidel:pi1}, \cite{entov:K-area},
\cite[Section 20.2]{oh:book2}.)
In this spirit, we regard
\be\label{eq:PK}
P_{H_K}(u): = X_{H_K}(u)\, dt + X_{G_K}(u)\, ds
\ee
as a $u^*TM$ valued one-form on $\Theta_{K_0+1}$ for which
we express using the basis of one-forms $\{ds,dt\}$ coming from
the inclusion map $\Theta_{K_0+1} \subset \R \times [0,1]$.
This being said, we define the relevant $\pi$-energy as follows.

\begin{defn} Let $(J,H)$ be a Floer data with $J = J(t,x)$, $H = H(t,x)$.
Fix $K_0 > 0$. For $0 \leq K \leq K_0$, we define
the (perturbed) off-shell $\pi$-energy $E^\pi_{J_K,H_K}$ by
\be \label{eq:EJKH}
E^\pi_{J_K,H_K} (u) =
\int_{\Theta_{K_0+1}} e^{g_K(u)} |(du -P_K(u))^\pi|^2_{J_K} \, dA
\ee
for any smooth map $u: \Theta_{K_0+1} \to M$.
\end{defn}
We also define the corresponding vertical energy
$E^\lambda_{J_K,H_K}$ in Section \ref{sec:lambda-energy},
which is also called the $\lambda$-energy. Then we define the
total energy by
$$
E_{J_K,H_K}(u) = E^\pi_{J_K,H_K} (u) + E^\lambda_{J_K,H_K}(u).
$$
Then we introduce the following moduli space of
finite energy solutions.

\begin{defn}\label{defn:CM-K} Let $K_0 \in \R \cup \{\infty\}$ be given.
For $0 \leq K \leq K_0$, we define
\be\label{eq:MM-K}
\CM_K(M,R;J,H)  = \{ u:\R \times [0,1] \to M \, |\,  u
\; \mbox{satisfies \eqref{eq:HGK-Theta} and} \; E_{J_K,H_K}(u) < \infty\}
\ee
and
\be\label{eq:MM-para}
\CM_{[0,K_0]}^{\text{\rm para}}(M,R;J,H) = \bigcup_{K \in [0,K_0]} \{K\} \times \CM_K(M,R;J,H).
\ee
\end{defn}

The Fredholm theory developed in \cite{oh:contacton} provides a natural smooth structure with
$$
\CM_K(M,R;J,H), \quad \CM_{[0,K_0]}^{\text{\rm para}}(M,R;J,H).
$$
\cite{oh:contacton} deals with the closed string case which can be
easily adapted to the current case with boundary, whose details will be
given elsewhere.

\begin{rem} \begin{enumerate}
\item
We would like to attract readers' attention to the difference in the
setting-up of the \emph{domain-varying} parameterized moduli space: In \cite{oh:mrl}, we
fix the domain to be the total strip $\R \times [0,1]$ while here we compactify the domain
by using the $K$-dependent family of capped-strips and vary $K \to \infty$. The reason for this
is because for the pseudoholomorphic curves for the Lagrangian boundary condition the removal singularity
theorem automatically compactifies the domain of any finite energy solution defined on $\R \times [0,1]$,
while such a removable singularity theorem does not apply to the contact instanton equation.
Using the fact that the contact instanton equation is still coordinate free, we consider the family of
compact domains $\Theta_{K_0+1}$ and repeat the same kind of proof given in the proof of \cite[Lemma 2.2]{oh:mrl}.
Indeed we can also apply the same domain-changing family to the latter case and rewrite the proof of
\cite[Lemma 2.2]{oh:mrl} as in the way how the current proof goes.
\item  Because of the aforementioned difference working on the compact domains $\Theta_{K_0+1}$
instead of the full strip $\R \times [0,1]$, strictly speaking,
we need to smooth out the seam $|\tau| = K+1$ on a small neighborhood thereof, say on
$$
\{(\tau,t) \in \Theta_{K_0+1} \mid \tau \in [K+1 - \delta,K+1 + \delta]\}
$$
for some given $\delta > 0$. (Note that $\del \Theta_{K_0+1}$ is $C^1$ but not smooth.)
We replace $\chi_K$ by $\chi_{K+\delta}$ so that
the equation \eqref{eq:HGK-Theta} still becomes \eqref{eq:contacton} for $0 \leq K \leq K_0$.
This being said, we will ignore non-smoothness of the boundary of $\Theta_{K_0+1}$
at $|\tau| = K+1$ and directly work with it instead of the smoothed one for the
simplicity of exposition.
\end{enumerate}
\end{rem}

\subsection{Obstruction to existence of finite energy solution}

We recall the following correspondence
\be\label{eq:no-ZR-intersect-psi(R)}
\psi(R) \cap Z_R= \emptyset \Longleftrightarrow \psi(R) \cap R = \emptyset \quad \& \quad
\frak{Reeb}(\psi(R),R) = \emptyset
\ee
from Lemma \ref{lem:key-conversion} applied to $(S_1,S_2)=(R,R)$.
We will show that the condition $\psi(R) \cap Z_R = \emptyset$
will play the role of an obstruction to the existence of finite energy solutions to
the contact instanton equation \eqref{eq:perturbed-contacton-RR}.
In this subsection, we assume the priori energy bound which we will establish in
Section \ref{sec:pienergy-bound} for the $\lambda$-energy $E_{J_{K_\alpha},H}^\lambda(u)$
and Section \ref{sec:C1-estimate}.

\begin{rem}
This obstruction is the analog to the obstruction employed by  in \cite{chekanov:dmj}, \cite{oh:mrl}
for the study of displacement energy of compact Lagrangian submanifolds.
Similarly as in the present paper, this kind of non-intersection played the role of an obstruction to
compactness of certain parameterized moduli space of Hamiltonian-perturbed Floer trajectories
in symplectic geometry. (See \cite{oh:mrl} for the details.)
\end{rem}

To state the aforementioned a priori energy bounds,
we also consider the family $J'_K = J'_{K;(s,t)}$ defined by
\be\label{eq:Jst}
J'_{K;(s,t)} = (\psi_{H^s}^t(\psi_{H^s}^1)^{-1})^*J_{(s,t)} =  (\psi_{H^s}^1(\psi_{H^s}^t)^{-1})_*J_{(s,t)}
\ee
for a given $J = \{J_{(s,t)}\}$.
We also consider the gauge transformation of $u$
\be\label{eq:ubarK}
\overline u_K(\tau,t): = \psi_{H}^{\chi_K(\tau)}(\psi_{H}^{\chi_K(\tau)t})^{-1}(u(\tau,t)).
\ee

 We will prove the following propositions
in Section \ref{sec:pienergy-bound} and Sections \ref{sec:lambda-energy}, \ref{sec:C1-estimate}
of Part \ref{part:bubbling} respectively.  In particular,
the a priori $\pi$-energy bound is a key ingredient in relation to
the lower bound of the Reeb-untangling energy.
The proof will be carried out by the calculation similar to
the one which we extract from the calculation made in \cite{oh:mrl}
 in the context of Lagrangian Floer theory.
However the current calculation is  significantly much more complex than and varies from that of
\cite{oh:mrl} because of the presence of the new conformal exponent factor $e^{g_K(u)}$ which
makes the calculation much more involved that of \cite{oh:mrl}.

\begin{rem} We would like to emphasize that the calculation leading to the proof of
the energy identity is the heart of the matter that relates the $\pi$-energy and the oscillation norm
similarly as in \cite{oh:mrl}. This is one of the key ingredients in the quantitative
symplectic and contact topology.
\end{rem}

We recall the definition of oscillation
$$
\osc(H_t) = \max H_t - \min H_t.
$$
\begin{prop}\label{prop:pienergy-bound-H}
Let $u$ be any finite energy solution of \eqref{eq:HGK}. Then we have
\be\label{eq:pienergy-bound}
E_{J'_K}^\pi(\overline u_K) \leq \int_0^1 \osc(H_t) \, dt = :\|H\|
\ee
\end{prop}

\begin{prop}\label{prop:lambdaenergy-bound}
Let $u$ be any finite energy solution of \eqref{eq:HGK}. Then we have
\be\label{eq:lambdaenergy-bound}
E^\perp_{J'_K}(\overline u_K) \leq \|H\|.
\ee
\end{prop}
Let
$$
E_{J'_K}(\overline u_K) = E_{J'_K}^\pi(u) + E_{J'_K}^\perp(\overline u_K)
$$
be the total energy. Then using these energy bounds,
we will prove the following Gromov-type weak convergence theorem, the details of
which we postpone till Section \ref{sec:C1-estimate} of Part \ref{part:bubbling}. We just mention here that
the above explicit upper bound of $E_{(J_K,H_K)}^\pi(u)$ plays a fundamental role in our quantitative study while
that of $E^\perp_{(J_K,H_K)}(u)$, other than the existence of uniform finite upper bound, does not
play any role.

\begin{thm}\label{thm:compactness}
Suppose $K_\alpha \to K_\infty \leq K_0 \in \R$
and let $\overline u_\alpha$ be solutions of \eqref{eq:K} for $K = K_\alpha$ with uniform
energy bound
$$
E_{J_{K_\alpha}'} (\overline u_\alpha) < C < \infty
$$
for $C$ independent of $\alpha$. Let $\overline u$ be as above.

Then, there exist a subsequence again enumerated by $v_\alpha$ and a
cusp-trajectory $(u_0,; {\bf v}, {\bf w}, )$ such that
\begin{enumerate}
\item  $u_0$ is a solution of \eqref{eq:perturbed-contacton-bdy-K}
for $H_K$ with $K = K_\infty$.
\item   ${\bf v} = \{ v_i \}^k_{i=1}$ where
\begin{itemize}
\item each $v_i$ is a
contact instanton on the plane $\C$ with end converging to a closed Reeb orbit, and
\item each $w_j$ is a contact instanton on the half place $\H$ with its boundary lying on $R$
with end converging to a Reeb chord of the pair $(\psi(R),R)$.
\end{itemize}
\item  We have
 $$
\lim_{\alpha \to \infty} E_{J'_{K_\alpha}}^\pi
(\overline u_{\alpha,K_\alpha}) = E_{J_{K_\infty}'}^\pi (\overline u_{K_\infty})
+ \sum_i \int v_i^*d\lambda_{z_i} +
\sum_j w_j^*d\lambda_{z_j}.
$$
\item  And $\overline u_{\alpha,K_\alpha}$ weakly converges to $(\overline u, {\bf v}, {\bf
w})$ in the sense of Gromov-Floer-Hofer and converges in compact $C^\infty$ topology away
from the nodes.
\end{enumerate}
Furthermore, if ${\bf v} = {\bf w} = \emptyset$, then
$u_\alpha \to u$ smoothly on $\Theta_{K_0 + 1}$.
\end{thm}

With the above energy bound and convergence result in our disposal, we
prove the following obstruction result.
The following proposition is the analog to \cite[Lemma 2.2]{oh:mrl}: It shows
that non-intersection $\psi_H^1(R) \cap Z_R = \emptyset$
is an obstruction to the asymptotic existence of finite energy solutions for the \emph{autonomous} equation
\eqref{eq:perturbed-contacton-bdy}.

\begin{prop}\label{prop:obstruction}
Let $H_t$ be the Hamiltonian such that $\psi_H^1(R) \cap Z_R$ is empty
and $H = sH_t$ and $J$ as before. Suppose $\|H\| < T_\lambda(M,R)$.
Then there exists $K_0> 0$ sufficiently large such that
$\CM_K(M,R;J,H)$ is empty for all $K \geq K_0$.
\end{prop}
\begin{proof}
We give the proof by contradiction by closely following the scheme used  in \cite{oh:mrl}.

Suppose to the contrary that there exists a sequence $K_\alpha \to
\infty$ and a solution $u_\alpha \in M_{K_\alpha} (M,R,J,H)$. By the
a priori energy bound from Propositions \ref{prop:pienergy-bound-H}, \ref{prop:lambdaenergy-bound}
for $M_{K_\alpha} (M,R;J,H)$ and the bubbling convergence theorem,
Theorem \ref{thm:bubbling-intro}, there exists $ (u_{-0},u_0, u_{+0};
{\bf v}, {\bf w})$ such that $u_\alpha \to ((u_{-0},u_0, u_{+0}; {\bf v}, {\bf w})$
in the sense of Theorem \ref{thm:compactness} with energy convergence
\beastar
\lim_{K_\alpha \to \infty} E^\pi_{J_{K_\alpha}'}
(\overline u_{\alpha,K_\alpha}) & = & E_{J_{-\infty}^{'\rho_+}}^\pi(\overline u_{-0}) 
+ E_{J_{K_\infty}'}^\pi (\overline u_{K_\infty}) + E_{J_{+\infty}^{'\rho_-}}^\pi
(\overline u_{-0}) \\
&{}& \quad + \sum_i \int v_i^*d\lambda_{z_i} + \sum_j w_j^*d\lambda_{z_j}
\eeastar
because $v_i$ all are nonconstant $J_0$ contact disc-type (or more precisely
cigar-type) instanton with  its asymptotic limit at $\tau = \infty$ becoming

On the other hand,
by the hypothesis $\|H\| < T_\lambda(M,R)$, we derive ${\bf v} = \emptyset = {\bf w}$.

In particular, we have produced the $C^1$-limit $u$ of $u_\alpha$
which satisfies the equation \eqref{eq:perturbed-contacton-RR}, i.e.,
$$
\begin{cases}
(du - X_H(u))^{\pi(0,1)} = 0, \quad d(e^{g_H(u)}(u^*\lambda + u^*H dt) \circ j) = 0,\\
u(\tau,0), \quad u(\tau,1) \in R
\end{cases}
$$
with $\pi$-energy bound. In particular we also have $\|du\|_{C^0} < \infty$.
Furthermore we also have
$$
 E_{J,H}^\pi (u) \leq \limsup_{\alpha \to \infty} E_{J_{K_\alpha},H}^\pi (u_\alpha) \leq \|  H \| < \infty.
$$
Then by Theorem \ref{thm:subsequence-convergence} applied to $w = \overline u$, we derive
that $w(\pm\infty)$ is an element of  $\frak{Reeb}(\psi(R),R) \neq \emptyset$ for $\psi =\psi_H^1$. Furthermore by the standing
hypothesis \eqref{eq:standing-hypothesis}, the relevant iso-speed Reeb chord $\gamma$ 
cannot be  a constant chord.  This implies that 
there exists at least one non-constant iso-speed Reeb chord between $\psi(R)$ and $R$
such that it bounds an element 
$$
u_{-,0} \in \overline{\CM}_-^{\text{\rm para}} (M,R;J,H), \quad{\text{\rm or}} \quad
u_{+,0} \in \overline{\CM}_+^{\text{\rm para}}K_+(M,R;J,H)
$$
with the inequality
$$
0 < \int \gamma^*\lambda \leq E^\pi_{J,H}(u_{\pm0}) \leq \|H\|.
$$
This contradicts to the standing hypothesis $\|H\| < T_\lambda(M,R)$ and so finishes the proof of
the proposition.
\end{proof}

\section{Deformation-cobordism analysis of parameterized moduli space}
\label{sec:cobordism}

With the $\pi$-energy bound \eqref{eq:pienergy-bound} and the $\lambda$-energy bound \eqref{eq:lambdaenergy-bound},
we are now ready to make a deformation-cobordism analysis of
$\CM^{\text{\rm para}}_{[0,K_0+1]}(M,\lambda;R,H)$. The logical scheme of this analysis is similar to that of \cite{oh:mrl}.

\subsection{The case $K = 0$: $J \equiv J_0$ and $H \equiv0$.}

In this case, the equation \eqref{eq:perturbed-contacton-RR} becomes
\be\label{eq:K=0}
\begin{cases}
\delbar^\pi u = 0, \quad d(u^*\lambda \circ j) = 0\\
u(\tau,0), \, u(\tau,1) \in R
\end{cases}
\ee
with $E_{J^0}(u) = E^\pi_{J_0}(u) + E^\lambda_{J_0}(u) < \infty$.

The following is the open string version of \cite[Proposition 1.4]{abbas}, \cite[Propostition 3.4]{oh-wang:CR-map1}.

\begin{prop}\label{prop:abbas-open}
Assume $w:(\Sigma, \del \Sigma) \to (M, R)$ is a smooth contact instanton from a compact
connected Riemann surface $(\Sigma,j)$ genus zero with one boundary component.
Then $w$ is a constant map.
\end{prop}
\begin{proof} For contact Cauchy--Riemann maps, we have
$$
|d^\pi w|^2\, dA =d(2w^*\lambda).
$$
By Stokes' formula and the Legendrian boundary condition, we derive
$$
\frac12 \int_\Sigma |d^\pi w|^2 = \int_{\Sigma} dw^*\lambda = \int_{\del \Sigma} w^*\lambda = 0
$$
when $\Sigma$ is compact with $w(\del \Sigma) \subset R$. This implies
$|d^\pi w|^2=0$ which in turn implies $dw^*\lambda = 0$ by the above equality.
Combining the defining equation $d(w^*\lambda \circ j) = $ for contact instantons, this vanishing
implies that $w^*\lambda$ (so is $*w^*\lambda$) is a harmonic one-form on the
compact Riemann surface $\Sigma$ satisfying
$$
w^*\lambda|_{\del \Sigma} = 0.
$$
If the genus of $\Sigma$ is zero, i.e., when $\Sigma = D^2$. By the reflection argument, we
prove $w^*\lambda \equiv 0$ on $\Sigma$. This together with $d^\pi w = 0$ implies $w$ must be
a constant map valued at a point of $R$. This finishes the proof.
\end{proof}

\begin{rem} Similarly as in \cite[Proposition 1.4]{abbas}, \cite[Propostition 3.4]{oh-wang:CR-map1},
we can prove that for the case with $g(\Sigma)\geq 1$, $w$ is either a constant or the locus of its image
is either a closed Reeb chord or a finite union of Reeb chords of $R$ contained in a single leaf.
Since this is not used in the present paper, we do not elaborate its proof referring the similar
result for the closed string case to ibid.
\end{rem}

Postponing derivation of the full index formula for the linearization operator $D\Upsilon(w)$ of
the equation \eqref{eq:K=0} till elsewhere, we just prove the following
proposition. (See \cite[Section10-11]{oh:contacton} for the linearization and the full index formula
for the closed string case, and \cite{oh-yso:index} for the open string case.)

\begin{prop} Let $w_p: (D^2,\del D^2) \to (M,R)$ be the constant map valued at $p \in R \subset M$
regarded as a constant solution to \eqref{eq:K=0}. Consider the map
$$
\Upsilon: w \mapsto (\delbar^\pi w,d(w^*\lambda \circ j))
$$
and its  linearization operator
$$
D\Upsilon(w): \Omega^0(w^*TM, (\del w)^*TR) \to \Omega^{(0,1)}(w^*\xi) \oplus \Omega^2(\Sigma,\R).
$$
Then we have
$$
\ker D\Upsilon(w_p) = n (=\dim R), \quad \coker D\Upsilon(w_p) = 0.
$$
In particular, we have
$$
\Index D\Upsilon(w) = n
$$
for any element $w \in \CM_0(M,R;J_0, H_K)$ homotopic to a constant map relative to $R$ for all $K$.
\end{prop}
\begin{proof} We denote by
$$
D\Upsilon(w): \Omega^0_{k,p}(w^*TM, (\del w)^*TR) \to \Omega_{k-1,p}^{(0,1)}(w^*\xi) \oplus \Omega^2_{k-2,p}(\Sigma,\R)
$$
the $W^{(k,p)}$-completion of $\Omega(w^*TM, (\del w)^*TR)$, the set of vector fields over the map $w$ and similarly for
other completions.
When $w_p$ is the constant map valued at $p \in R$, it is easy to see that we have
$$
D\Upsilon(w_p)(\eta) = (\delbar \eta^\pi , -*\Delta f), \quad \eta = \eta^\pi + f R_\lambda
$$
where
\begin{itemize}
\item $\eta^\pi$ is a map $D^2 \to (\xi_p,J_p)$ satisfying the totally
real boundary condition $\eta^\pi_p(\del D^2) \subset T_pR$,
\item $f$ is a real-valued function on $D^2$ that satisfies the Dirichlet boundary condition,
which follows from the Legendrian boundary condition $R$ along $\del D^2$.
(We refer to \cite[Theorem 10.1]{oh:contacton} for the precise formula for the linearization operator $\Upsilon(w)$.)
\end{itemize}
It is well-known that the Laplacian with Dirichlet boundary condition on the disc has zero kernel
and the corresponding index is zero. 

When we identity $(\xi_p, J_p)$ with $(\C^n,i)$, $\delbar$
is nothing but the standard Cauchy-Riemann operator. Then the first statement of the proposition is a well-known result
whose proof can be found in \cite{oh:Riemann-Hilbert}. The second statement follows from the homotopy
invariance of the Fredholm index.
\end{proof}

An immediate corollary of the above two propositions (with $g=0$) is the following description of the moduli space
$\CM_0(M,R;J_0, H) \cong \CM(M,R;J_0,0)$.

\begin{cor} The evaluation map $\ev_{(0,0)}: \CM_0(M,R;J_0,H) \to R$ is a diffeomorphism.
In particular its $\Z_2$-degree is nonzero.
\end{cor}

\subsection{The case $K \to K_0$}

Let $K_0 > 0$ be the constant that satisfies Proposition \ref{prop:obstruction} so that
\be\label{eq:K_0-empty}
\CM_{K_0}(M,R;J,H) = \emptyset.
\ee
The following is the basic structure theorem of
$\CM_K(M,R;J,H)$ given in \eqref{eq:MM-K} whose proof
is a variation of the generic transversality theorem and so is
omitted. (See \cite{oh:contacton} for a proof of similar transversality result
proven for the closed string case.)

\begin{thm}\label{thm:MMKA}
\begin{enumerate}
\item For each fixed $K>0$, there exists a generic choice of
$(J,H)$ such that $\CM_K(M,R;J,H)$ becomes a smooth manifold of dim
$n$ if non-empty. In particular,  $\dim \CM_K(M,R;J,H) = n$ if non-empty.
\item For the case $K=0$, all solutions are constant and Fredholm regular
and hence $\CM_K(M,R;J,H) \cong R$. Furthermore the evaluation map
$$
\ev : \CM_0(M,R;J,H) \to R :  u \mapsto u(0,0)
$$
is a diffeomorphism.
\item The parameterized moduli space
$
\CM_{[0,K_0]}^{\text{\rm para}}(M,R;J,H) \to [0,K_0]
$
is a smooth manifold of dimension $n+1$ with boundary
$$
\{0\} \times \CM_0(M,R;J,H) \coprod \{K_0\} \times \CM_{K_0}(M,R;J,H)
$$
and the evaluation map
$$
Ev :\CM_{[0,K_0]}^{\text{\rm para}}(M,R;J,H) \times \R \to L \times \R_+ \times
\R : ((K,u), \tau) \mapsto (K,u(\tau),\tau)
$$
is smooth.
\item $\CM_K(M,R;J,H) = \emptyset$ for all $K \geq K_0$.
\end{enumerate}
\end{thm}

\subsection{Cobordism-deformation of triads}

The following upper bound for the bubble energy is
a key ingredient in the proof of Theorem \ref{thm:shelukhin-intro} in which consideration of
domain dependent family of contact triads
$$
\{(M, \lambda_z,J_z)\}_{z \in \R \times [0,1]}, \quad z = (\chi_K(\tau),t)
$$
is a crucial ingredient.
\begin{choice}\label{choice:lambda-J} We consider the following two parameter families of $J$ and $\lambda$:
\bea
J_{(s,t)}' & =  & ((\psi_{H^s}^t (\psi_{H^s}^1)^{-1})^{-1})_*J= (\psi_{H^s}^t (\psi_{H^s}^1)^{-1})^*J, \\
\lambda_{(s,t)}' & = & ((\psi_{H^s}^t (\psi_{H^s}^1)^{-1})^{-1})_*\lambda= (\psi_{H^s}^t (\phi_{H^s}^1)^{-1})^*\lambda.
\eea
\end{choice}
See Remark \ref{rem:contact-bubbling}.
Once this set-up is carefully introduced, the proof of the
upper bound is an easy consequence of Propositions \ref{prop:obstruction}
and \ref{prop:pienergy-bound-H}.

\begin{prop}\label{cor:<||H||}
Let $(K_\alpha,u_\alpha)$ be a bubbling-off sequence with $0 \leq K_\alpha \leq K_0 < \infty$ such that
$$
u_\alpha \in \CM_{K_\alpha } (J, H)
$$
with
$$
u_\alpha \to (u, {\bf v}, {\bf w} )
$$
in the sense of Theorem \ref{thm:compactness}.
Then any bubble must have positive asymptotic action less than $\| H \|$.
\end{prop}
\begin{proof} By the way how the bubble is constructed,
 there exists a subsequence, still denoted by $u_\alpha$, we have
\beastar
&{}& \limsup_{\alpha} E_{(\lambda_K,J_K)}^\pi (\overline u_{\alpha,K_\alpha} )\\
&= &E_{(\lambda_\infty,J_\infty)}^\pi (\overline u_{\infty,K_\infty})
+ \sum_i E_{(\lambda_{z_i},J_{z_i})}^\pi(v_i) + \sum_j E_{(\lambda_j,J_{z_j})}^\pi(w_j)
\eeastar
where each bubble $v_i$ (resp. $w_j$) is a contact instanton for the triad
$$
(M,\lambda_{z_i}, J_{z_i}), \quad z_i = (\tau_i,t_i),
$$
and the norm is taken with respect to the associated triad metrics:
$$
g_{z} = d\lambda_z(\cdot, J_{z} \cdot) + \lambda_z \otimes \lambda_z.
$$
(\emph{Under the given condition $\|H\| < T_\lambda(M,R)$, Proposition \ref{prop:obstruction}
rules out the `bubbling at infinity of $\tau$' that should be included without it
as shown in Theorem \ref{thm:bubbling-intro} for the full compactification 
with a uniform energy bound.})

We first consider the disc bubbles $w_j$.
We also assume that the bubble point is contained in $\{t=0\}$.
The other cases with $\{t = 1\}$ is easier and can be treated in the same way.

But by the definition of the triad metric, we have
$$
\frac12 |d^\pi w_j|_{(\lambda_{z_j},J_{z_j})}^2\, dA = w_j^*d\lambda_{z_j} \geq 0
$$
since $d^\pi w_j$ is $(J_{z_j},j)$-complex linear and $J_{z_j}$ is $d\lambda_{z_i}$-adapted.
Therefore it follows from Propositions \ref{prop:pienergy-bound-H} that
$$
E_{J_{z_j}}^\pi (w_j) \leq \limsup_{\alpha} E_{(\lambda_K,J_K)}^\pi (\overline u_{\alpha,K_\alpha} ) \leq \| H\|,
$$
we have derived
$$
\int_\H dw_j^*\lambda_{z_j} \leq \|H\|.
$$
Now the proof will be complete once we prove the following key lemma.
We recall the definition of the action spectrum $\operatorname{Spec}(M,\lambda)$
(resp. $\operatorname{Spec}(M,R;\lambda)$)
and the period gap $T(M,\lambda)$ (resp. $T(M,R;\lambda)$)
from Definition \ref{defn:spectrum}.

\begin{lem}\label{lem:bubble-action}
Let $\gamma_j$ be the asymptotic $\lambda_{z_j}$-Reeb chord of $w_j$.
Then the value
$$
\int_\H dw_j^*\lambda_{z_j}
$$
is contained in $\Spec(M,R;\lambda)$.
\end{lem}
\begin{proof} By finiteness of the energy $E_{J_{z_j}}^\pi (w_j) < \infty$ and the Legendrian boundary condition,
we derive
$$
\int_\H dw_j^*\lambda_{z_j} = \int_{\del \H} \gamma^*\lambda_{z_j}
$$
with
$$
\gamma(t) = \lim_{\tau \to \infty} w (e^{\pi(\tau+it)})
$$
where $(\tau,t) \in [0,\infty) \times [0,1] \subset \dot \Sigma \setminus \{z_j\}$ is the strip-like coordinate
around the bubble point $z_j = (\tau_j,0) \in \Sigma: = \R \times [0,1]$.

Recall that $\gamma$ satisfies the boundary condition
\be\label{eq:gamma-bdy}
\gamma(0), \, \gamma(1) \in \psi_{H^{s_j}}^1(R)
\ee
for $s_j = \chi_{K_j}(\tau_j)$,
since it is an asymptotic limit of a disc bubble $w_j$ at $t=0$ where $w_j$ satisfies
$$
w_j(\{t=0\}) \subset \psi_{H^{s_j}}^1(R), \quad s_j = \chi_{K_\infty}(\tau_j).
$$
On the other hand, we compute
$$
\gamma^*\lambda_{z_j} = \gamma^*(\psi_{H^{s_j}}^0 (\psi_{H^{s_j}}^1)^{-1})^*\lambda
=  \gamma^*((\psi_{H^{s_j}}^1)^{-1})^*\lambda = \left((\psi_{H^{s_j}}^1)^{-1} \circ \gamma\right)^*\lambda.
$$
If we set
$$
\widetilde \gamma(y) := (\psi_{H^{s_j}}^1)^{-1}(\gamma(y)),
$$
\eqref{eq:gamma-bdy} implies
$$
\widetilde \gamma(0), \, \widetilde \gamma(1) \in R.
$$
Furthermore $\widetilde \gamma$ is a $\lambda$-Reeb chord, since $\gamma$ is a
$\lambda_{(s_j,0)}$-Reeb chord and
$$
\lambda_{(s_j,0)} = (\psi_{H^{s_j}}^1)^*\lambda.
$$
Therefore we have proved
$$
\int_{\del \H} \gamma^*\lambda_{z_j} = \int_{\del \H} \widetilde \gamma^*\lambda \in T(M,R;\lambda).
$$
This finishes the proof.
\end{proof}
Easier proof also applies to the sphere bubble $v_i$, which now finishes the proof of
the proposition.
\end{proof}

\begin{rem}\label{rem:contact-bubbling} For the above action bound to hold for the bubbles, it is
\emph{essential to vary} the contact triads depending on the domain parameters by simultaneously varying the pairs
$$
(\lambda, J).
$$
For example, if one fixed $\lambda$ while $J$ varies, one would not be able to prove
this bound of the asymptotic period for the given contact form $\lambda$ in terms of the oscillation norm $\|H\|$:
\emph{The only thing one could get is the kind of statement that there is a $s_1 \in [0,1]$ and a bubble
whose $(\psi_H^{s_1})^*\lambda$-period is less than or equal to $\|H\|$.} Since \emph{general
contactomorphism does not preserve $\lambda$ and hence $(\psi_H^{s_1})^*\lambda \ne \lambda$,
the two periods are generally different.}
This is a fundamental difference from the symplectic case where the symplectic form
can be fixed while the almost complex structure varies since \emph{symplectomorphism
preserves the symplectic form.} This is a reflection of the fact that the contact distribution $\xi$ is
the fundamental geometric structure and the contact form itself should be regarded as one of
the free parameters in the study of contact topology and dynamics like $J$ or $H$.
\end{rem}

\section{Without nondegeneracy: existence of Reeb chords}
\label{sec:translated-intersections}

We start with the following definition.

\begin{defn}\label{defn:Reeb-transversal-pair}
We say the pair
$$
(\lambda,(R_0,R_1))
$$
is nondegenerate if all closed $\lambda$-Reeb orbits of $M$ and $\lambda$-Reeb chords from
$R_0$ to $R_1$ are nondegenerate.
\end{defn}
Note that the latter nondegeneracy is equivalent to the
transversality of the intersection
$$
R_0 \pitchfork Z_{R_1}.
$$

The following is a standard lemma in contact geometry, which is an easy consequence of
nondegeneracy.

\begin{lem} Let $(M,\xi)$ be a closed contact manifold. Then we have the following:
\begin{enumerate}
\item  $\operatorname{Spec}(M,\lambda)$ (resp. $\operatorname{Spec}(M,\lambda;R_0,R_1)$)
is either empty or countable nowhere dense in $\R_+$.
\item $T(M,\lambda)> 0$ (resp. $T_\lambda(R_0,R_1)$).
\item When the pair $(\lambda,(R_0,R_1))$ is nondegenerate, then each of the sets
$$
\operatorname{Spec}^{K}(M,\lambda) = \operatorname{Spec}(M,\lambda) \cap (0,K]
$$
and
$$
\operatorname{Spec}^{K}(M,\lambda;R_0,R_1) = \operatorname{Spec}(M,\lambda;R_0,R_1) \cap (0,K]
$$
are finite for each $K> 0$.
\end{enumerate}
\end{lem}

\begin{defn}[Relative period gap] For given Legendrian submanifold $R \subset M$, we define
the constant $T_\lambda(M;R) > 0$ by
$$
T_{\lambda}(M,R) = \min \{T(M,\lambda),\, T(M,R;\lambda)\}
$$
and call it the \emph{ period gap} of the pair $(M,R)$.
\end{defn}

We now introduce the notion of \emph{Reeb-untangling energy} of one
Legendrian submanifold from the Reeb trace of another Legendrian submanifold.

\begin{defn}[Reeb-untangling energy] Let $(M,\xi)$ be a contact manifold, and let
$R_0, \, R_1$ of compact Legendrian submanifolds $(M,\xi)$.
\begin{enumerate}
\item  We define
\be\label{eq:lambda-untangling-energy}
e_\lambda^{\text{\rm trn}}(R_0, R_1) : = \inf_H\{ \|H\| ~|~ \psi_H^1(R_0) \cap Z_{R_1} = \emptyset\}.
\ee
We put $e_\lambda^{\text{\rm trn}}(R_0, R_1) = \infty$ if $\psi_H^1(R_0) \cap Z_{R_1} \neq \emptyset$ for
all $H$. We call $e_\lambda^{\text{\rm trn}}(R_0, R_1)$ the \emph{$\lambda$-untangling energy} between them.
\item We put
\be\label{eq:Reeb-untangling-energy}
e^{\text{\rm trn}}(R_0,R_1) = \inf_{\lambda \in \CC(\xi)} e_\lambda^{\text{\rm trn}}(R_0, R_1).
\ee
We call $e^{\text{\rm trn}}(R_0,R_1)$ the \emph{Reeb-untangling energy} between $(R_0,R_1)$ on $(M,\xi)$.
\end{enumerate}
\end{defn}
We postpone  proving  the following energy estimates till
Part \ref{part:bubbling}.

\begin{prop}\label{prop:pienergy-bound-H} 
Let $H_K^\chi$ be as above.
Let $u$ be any finite energy solution thereof. Then we have
\be\label{eq:pienergy-bound-infty}
 E_{J'_K}^\pi(\overline u_K) \leq \ \|H\| 
\ee
where $\|H\| =  E_+(H) + E_-(H) = \int_0^1 \osc(H_t) \, dt$.
\end{prop}

We recall the definition
$$
\osc(H_t) = \max H_t - \min H_t.
$$
\begin{prop}\label{prop:pienergy-bound-H-infty} For a given $H = H(t,x)$, consider the
two parameter Hamiltonians $H = \{H^s\}_{0 \leq s \leq 1}$ with $H^s= \Dev (t \mapsto \psi_H^{st})$ and then
$$
H_K = \Dev \left(t \mapsto \psi_H^{\chi_K(\tau)t} \right)
$$
and its associated equation \eqref{eq:HGK}.
Let $u$ be any finite energy solution thereof. Then we have
\be\label{eq:pienergy-bound-infty}
 E_{J'_K}^\pi(\overline u_K) \leq  \|H\| 
\ee
where $\|H\| = \int_0^1 \osc(H_t) \, dt$.
\end{prop}

We also have the following $\lambda$-energy estimate.

\begin{prop}\label{prop:lambdaenergy-bound-infty}
Let $u$ be any finite energy solution of \eqref{eq:HGK}. Then we have
\be\label{eq:lambdaenergy-bound-infty}
E^\perp_{J'_K}(\overline u_K) \leq \|H\|.
\ee
\end{prop}

The following proof is the contact analog to \cite[Theorem]{oh:mrl}.
(See also \cite{chekanov:dmj}.)

\begin{thm}\label{thm:contact-mrl} Let $R \subset (M,\lambda)$
be a compact Legendrian submanifold. Then we have
$$
e_\lambda^{\text{\rm trn}}(R,R) \geq T_\lambda(M,R)
$$
\end{thm}
\begin{proof} Let $H$ be any  Hamiltonian satisfying
 \be\label{eq:||H||<TlambdaR}
 \|H\| <  T_\lambda(M,R)
 \ee
We may assume after making a $C^\infty$-small perturbation of $H$,
we may also assume  the standing hypothesis \eqref{eq:no-intersection}
\be\label{eq:no-intersection2}
\psi_H^1(R) \cap R = \emptyset
\ee
in addition. Therefore any element in $\psi_H^1(R) \cap Z_R$, if any, corresponds to
a non-constant iso-speed Reeb chord of $R$.

We will prove that under the conditions \eqref{eq:||H||<TlambdaR} and \eqref{eq:no-intersection2},
we will produce a non-constant Reeb chord $\gamma$ between $\psi_S^1(R)$ and $R$
satisfying $0 < \int \gamma^*\lambda < \|H\|$, a contradiction to the standing
hypothesis $\|H\| < T_\lambda(M,R) \leq \int \gamma^*\lambda$, which will finish the
proof of the theorem by contradiction.

For this purpose, we consider the parameterized moduli space
\be\label{eq:M-para-K0}
\CM^{\rm para}_{[0,K_0]}(M,R;J,H)  =  \bigcup_{K \in [0,K_0]}\{K\} \times \CM_K(M,R;J,H)
\ee
which is fibered over $[0,K_0]$, and consider the evaluation map
$$
Ev : \CM^{\text{\rm para}}_{[0,K_0]}(M,R;J,H) \to L \times [0, K_0];\quad
            u   \mapsto  (u(0,0), K).
$$
The transversality theorem (see \cite[Section 12]{oh:contacton})
implies that  for generic choice of $J$, $\CM^{\rm
para}_{[0,K_0]}(M,R; J,H)$ is a smooth manifold of dimension $(\dim R + 1)$
with boundary
$$
\CM_0 (M,R;J,H) \coprod \CM_{K_0} (M,R;J,H).
$$
We also know that
$$
\ev_0 : \CM_0(M,R;J,H) \to R \quad \text{\rm is a diffeomorphism.}
$$
In particular, its $\Z_2$-degree of $ev_0$ is $1$.  On the other hand
we have
$$
\CM_{K_0}(M,R;J,H) = \emptyset
$$
by Proposition \ref{prop:obstruction},
hence the $\Z_2$-degree of the map $\ev_{K_0}$ is zero. But the
$\Z_2$-degree is invariant under a \emph{compact} cobordism. Therefore
$\CM^{\text{\rm para}}_{[0,K_0]}(M,R;J,H)$ cannot be compact. By Theorem \ref{thm:compactness},
a bubble must develop, i.e., there exists subsequences of
$K_\alpha\, \, u_\alpha$ again denoted by the same such that
$$
K_\alpha \to K_\infty \in [0, K_0],
$$
and
$
   u_\alpha \in \CM_{K_\alpha}(M,R;J,H)
$
converging to some cusp-curve $(u, {\bf v}$ or ${\bf w})$ in the sense of Theorem \ref{thm:compactness}
with either ${\bf v} \neq \emptyset, {\bf w} \neq \emptyset$. Recall Corollary \ref{cor:<||H||} implies
that the $\pi$-energy of the bubble is always less than $\|H\|$, and hence
\be\label{eq:2||H||}
T_{\lambda}(M,R) \leq \|H\|.
\ee
Now taking the infimum of $\|H\|$ over
all $H$ with $Z_R \cap \psi^1_H (R) = \emptyset$, we obtain
$$
0 < T_{\lambda}(M,R) \leq  \inf_H \{ \| H\| \mid Z_R \cap \psi^1_H (R) = \emptyset \}.
$$
But the quantity in the right hand side of this inequality is nothing but the Reeb-untangling
energy $e_\lambda^{\text{\rm trn}}(R,R)$ of $R$.

On the other hand, by Lemma \ref{lem:bubble-action}, there cannot be bubbles $v_i$ or $w_j$
at a finite time of $\tau$ because otherwise it will violate the standing hypothesis
\eqref{eq:standing-hypothesis}.  This implies the uniform $C^1$-bound for $u_\alpha$
\be\label{eq:uniform-C1bound}
\|d \overline u_\alpha\| < C \infty.
\ee
It implies that there exists a sequence 
$$
u_\alpha \in \CM^{\rm para}_{[0,K_\alpha]}(M,R; J,H)
$$
that converges to a limit of the form
$$
u_\infty = u_{-,0} + u_0 + u_{+,0}
$$
with $u_{\pm,0} \in \overline{\CM}^{\text{\rm para}}_{\pm}(M,R;J,H)$ respectively.
Furthermore the asymptotic limits of the gauge transformations $\overline u_{\pm,0}$
are iso-speed Reeb chords $(\gamma_\pm,T_\pm)$ between $\psi_H^1(R)$ and $R$.
By \eqref{eq:no-intersection}, $T_\pm \neq 0$. Furthermore we also have
$$
T_\pm = \int (\gamma_pm)^*\lambda = \int (\overline u_{\pm,0})^*d\lambda 
 \leq E^\pi_{J'}(\overline u_{\pm,0}) \leq \|H\|
 $$
which gives rise to a contradiction.
\end{proof}

\begin{rem}\label{rem:short-cut} 
Here is a more direct way of getting a contradiction without using the
above full compactness theorem, when we have established the uniform $C^1$-bound
\eqref{eq:uniform-C1bound}.
We consider the translated curve $\widetilde u_\alpha$ defined by
$$
\widetilde u_\alpha(\tau,t): = u_\alpha(\tau + K_\alpha, t)
$$
on the semi-strip $[0,\infty) \times [0,1]$. We observe
$$
\CA(\widetilde u_\alpha(0,\cdot)) = 0
$$
and hence we derive  
\bea\label{eq:energyonsemistrip}
\CA(\widetilde u_\alpha(K_\alpha,\cdot) & = & \CA(\widetilde u_\alpha(K_\alpha,\cdot)  -\CA(\widetilde u_\alpha(0,\cdot))
\nonumber \\
& = & \int_0^{K_\alpha}\int_0^1 |d^\pi \widetilde u_\alpha|^2 \, dt\, d\tau
\leq E^\pi(u_\alpha) \leq  \|H\|
\eea
by the action identity \eqref{eq:pi-energy}.
Therefore $\widetilde u_\alpha$ on $[0,\infty) \times [0,1]$ carries a subsequence still denoted by $\widetilde u_\alpha$
such that $\widetilde u_\alpha \to w_\infty$ on any compact subset of $(-\infty,0] \times [0,1]$ 
 for a map $u_\infty: (-\infty, 0] \times [0,1] \to M$  that satisfies
$$
\delbar^\pi_{J,H} u= 0, \quad u_\infty(\tau,0) \in \Gamma_\psi, \, _\infty(\tau,1) \in \Gamma_{id}
$$
with finite energy. This shows that it carries a Reeb chord $\gamma = u(\cdot, \infty)$ with $\CA(\gamma) < \|H\|$.
In particular, we have derived $T(M,\lambda) <  \|H\|$, a contradiction to the standing 
hypothesis \eqref{eq:||H||<TlambdaR}.
This finishes the proof.
\end{rem}

\section{With nondegeneracy assumption: the lower bound of 
$\# \Fix_\lambda^{\text{\rm trn}}(\psi)$}
\label{sec:with-nondegeneracy}

In this subsection, we assume that $(\lambda, (\psi(R),R))$ is nondegenerate
in the sense of Definition \ref{defn:Reeb-transversal-pair}. 
We will prove the following lower bound adapting the argument used by
Chekanov \cite{chekanov:dmj} in the context of Lagrangian Floer theory.
(See \cite{fooo:polydisks} for the relevant energy estimates and the proof of lower bound.)
The algebraic arguments leading to the lower bound is based on a purely algebraic
homological machinery  equally applies to the current context, 
\emph{as long as bubbling does not occur} which is ensured by the inequality
\be\label{eq:||H||}
\CA(\gamma_-) = 0, \quad \CA(\gamma_+)  < \|H\| < T_\lambda(M,R).
\ee
This is exactly the same argument as the one use in  \cite{chekanov:dmj}. 
We refer readers  to \cite{chekanov:dmj} or \cite{fooo:polydisks} 
for the relevant algebraic argument also.
Therefore we will be brief in the details of the proof leaving full details 
of the construction of general contact instanton Legendrian Floer cohomology
to \cite{oh:contacton-gluing,oh:perturbed-contacton-bdy}, \cite{oh-yso:spectral}.

Under the bound of $H$ given \eqref{eq:||H||} and $\phi_H^1(R) \cap R = \emptyset$, 
and together with a removal singularity theorem from \cite[Theorem 8.7]{oh:contacton},
no bubbling whatsoever will occur in our study of all relevant moduli spaces entering
in the present paper.

\begin{thm}\label{thm:lower-bound} Suppose $\psi$ is nondegenerate and let 
$H \mapsto \psi$ with $\|H\| < T_{\lambda}(M,R)$.
Then
$$
\#(\Fix_\lambda^{\text{\rm trn}}(\psi)) \geq \dim H_*(R;\Z_2).
$$
\end{thm}
The remaining section will be occupied by the proof of this theorem.

For this purpose, we need the analog of Proposition \ref{prop:pienergy-dlambda}
for the continuation map version. We consider the same family 
$\{H^s\}_{0 \leq s \leq 1}$  for the family of Hamiltonian $H^s(t,x)= \Dev_\lambda(t \mapsto \psi_H^{st})$. 
We elongate the family by the function $\rho: \R \to [0,1]$ satisfying
\bea\label{eq:rhoK}
\rho(\tau) & = & \begin{cases} 1 \quad & \tau \geq 1 \\
0 \quad & \tau \leq 0
\end{cases} , \quad \rho_K(\tau): = \rho(\tau-K)
\nonumber\\
\rho_K'(\tau) &\geq&  0 \quad \text{\rm on }\, [0,1].
\eea
We then study the associated continuation map equation
\be\label{eq:HG1}
\begin{cases}
(du - X_{H^\rho}(u)\, dt - X_{G}(u)\, ds)^{\pi(0,1)} = 0, \\
d\left(e^{g_{H^\rho}(u)}(u^*\lambda + u^*H^\rho dt + u^*G\, d\tau) \circ j\right) = 0,\\
u(\tau,0) \in R,\, u(\tau,1) \in R.
\end{cases}
\ee
where $G =  G(\tau,t,x) = \Dev_\lambda\left(\tau \mapsto \psi_H^{\rho(\tau)t}\right)(x)$.

We will prove the following propositions
in Section \ref{sec:pienergy-bound} and Sections \ref{sec:lambda-energy}, \ref{sec:C1-estimate}
of Part \ref{part:bubbling} respectively.
We recall the following definitions
$$
E_+(H) = \int_0^1 \max{H_t}\, dt, \quad E_-(H) = \int_0^1 - \min{H_t}\, dt.
$$
\begin{prop}[Compare with Theorem 12.1 \cite{oh-yso:spectral}]
\label{prop:pienergy-bound-H-infty} For a given $H = H(t,x)$, take the
two parameter Hamiltonians $H = \{H^s\}_{0 \leq s \leq 1}$ with $H^s= \Dev (t \mapsto \psi_H^{st})$
as in \eqref{eq:HK-intro}, and then consider
$$
H_K^\rho = \Dev_\lambda \left(t \mapsto \psi_H^{\rho_K(\tau)t} \right)
$$
and its associated equation \eqref{eq:HG1}.
Let $u$ be any finite energy solution thereof. Then we have
\be\label{eq:pienergy-bound-infty}
 E_{J'_K}^\pi(\overline u_K) \leq \int(\gamma_-)^*\lambda - \int (\gamma_+)^*\lambda  + E_+(H)
\ee
where $\|H\| = \int_0^1 \osc(H_t) \, dt$.
\end{prop}
An immediate corollary is the following action estimate

\begin{cor}\label{cor:action-estimate}
$$
\int(\gamma_+)^*\lambda - \int (\gamma_-)^*\lambda  
\leq  E_+(H) -  E_{J'_K}^\pi(\overline u_K) \leq E_+(H).
$$
In particular, if $u$ is not constant, the second inequality is strict.
\end{cor}

We also have the following $\lambda$-energy estimate.

\begin{prop}\label{prop:lambdaenergy-bound-infty}
Let $u$ be any finite energy solution of \eqref{eq:HGK}. Then we have
\be\label{eq:lambdaenergy-bound-infty}
E^\perp_{J'_K}(\overline u_K) \leq \|H\|.
\ee
\end{prop}

We denote by $\widetilde H$ the time-reversal Hamiltonian of $H$ given by
the formula \eqref{eq:inverse-Hamiltonian}. We take the elongation function 
$\rho: \R \to [0,1]$ and consider the maps
\be\label{eq:PsiHrho}
\Psi_{H^\rho} = (\Phi_{H^\rho})_*:  CI_\lambda^*(R,R) \to  CI_\lambda^*(\psi(R),R)
\ee
and
\be\label{eq:PsitildeHrho}
\Psi_{\widetilde H^{\rho}} = (\Psi_{\widetilde H^{\rho}}))_*: CI_\lambda^*(\psi(R),R) \to CI_\lambda^*(R,R)
\ee
where $\Phi_H$ is the gauge transformation defined in Subsection \ref{subsec:gauge-transformation}.
Recall $\psi_{\widetilde H}^1 = \psi^{-1}$.

Recall $\psi_{\widetilde H}^1 = \psi^{-1}$. Applying Corollary \ref{cor:action-estimate} to $H^{\rho}$ and
$\widetilde H^{\rho}$ exactly in the same way as \cite{chekanov:dmj} did, it follows that
$\Psi_{H^{\rho}}$ changes the filtration level not more that $E_-(H)$ and
$\Psi_{\widetilde H^{\rho}}$ not more than $E_+(H)$.

We also have similar estimate for the elongation function $\chi_K = \rho \#_K \widetilde \rho$.
Then for any given $H = H(t,x)$, we  consider
 the 2-parameter family Hamiltonians $H = \{H^s\}_{0 \leq s \leq 1}$ 
 with $H^s= \Dev (t \mapsto \psi_H^{st})$ mentioned  in the introduction. Then we consider 
$$
H_K^\chi = \Dev_\lambda \left(t \mapsto \psi_H^{\chi_K(\tau)t} \right)
$$
and its associated equation 
\be\label{eq:HGK}
\begin{cases}
(du - X_{H_K^\chi}(u)\, dt - X_{G_K}(u)\, ds)^{\pi(0,1)} = 0, \\
d\left(e^{g_K(u)}(u^*\lambda + u^*H_K^\chi dt + u^*G_K\, d\tau) \circ j\right) = 0,\\
u(\tau,0) \in R,\, u(\tau,1) \in R.
\end{cases}
\ee
where $g_K(u)$ is the function defined by
\eqref{eq:gKu} for $0 \leq K \leq K_0$. 

Now with all these preparations,
we are ready to give the proof of Theorem \ref{thm:lower-bound}
adapting Chekanov's argument used in \cite{chekanov:dmj} in the Lagrangian Floer
theory in symplectic geometry.

\begin{proof}[Proof of Theorem \ref{thm:lower-bound}]
Under the nondegeneracy assumption, we consider the $\Z_2$-vector space
$$
CI_\lambda^*(\psi(R), R) : = \Z_2\langle \frak{X}(\psi(R),R)\rangle
$$
where $\frak{X}(\psi(R),R)$ is the set of iso-speed Reeb chords introduced in
Definition \ref{defn:Reeb-chords-2}.
(Here $CI_\lambda$ stands for \emph{contact instanton for $\lambda$} as well as the letter $C$ also
stands for \emph{complex} at the same time. It follows by definition
 that when $T \neq 0$ and $\psi(R) \cap R = \emptyset$ (Upshot \ref{upshot:obstruction}), each element
$(T,\gamma) \in \frak{X}(R_0,R_1)$ gives rise to a Reeb chord from
$R_0$ to $R_1$ of period $T$ given by $\gamma_T := \gamma((\cdot)/T)$.)

We define a $\Z_2$-linear map
$$
\delta_{(\psi(R),R;H)} : CI_\lambda^*(\psi(R),R) \to  CI_\lambda^*(\psi(R),R)
$$
by its matrix element
$$
\langle \delta_{(\psi(R),R;H)}(\gamma^-), \gamma_+ \rangle : = \#_{\Z_2}(\CM(\gamma_-,\gamma_+)).
$$
Here $\CM(\gamma_-,\gamma_+)$ is the moduli space
$$
\CM(\gamma_-,\gamma_+) = \CM(M,\lambda;R;\gamma_-,\gamma_+)
$$
of contact instanton Floer trajectories $u$ satisfying
$u(\pm\infty) = \gamma_\pm$.

Similarly we define the maps
$$
\delta_{(R,R;0)} : CI_\lambda^*(R,R) \to  CI_\lambda^*(R,R).
$$
This is the Morse-Bott case in that we have decomposition
$$
\frak{X}(R,R) = \frak{X}_{< 0}(R,R) \sqcup \frak{X}_{0}(R,R) \sqcup \frak{X}_{> 0}(R,R)
$$
by definition of $\frak{X}(R,R)$ in Definition \ref{defn:Reeb-chords-2}
where $\frak{X}_0(R,R) \cong R$ is the set of constant paths and $\frak{X}_{> 0}(R,R)$ (resp.
$\frak{X}_{< 0}(R,R)$) is ones with $T > 0$ (resp. $T < 0$). In practice, we take
$$
CI_\lambda^*(R,R): =\Z \langle \frak{X}_{<\mathbb{} 0}(R,R) \rangle \oplus  C^*(R) \oplus \Z \langle \frak{X}_{> 0}(R,R) \rangle
$$
with any model $C^*(R)$ of cochain complex of $R$. We will take the
Morse homology complex $CM^*(R,f)$ for a $C^2$-small function $f$ on $R$.
We refer readers to \cite{oh:contacton-gluing}  for full explanation on this process.
(See \cite{bko:wrapped} for the same discussion in the symplectic context.)
This being said, we will just work with the Morse-Bott case $(R,R)$ in the following
exposition.

Applying Corollary \ref{cor:action-estimate} to $H^{\rho}$ and
$\widetilde H^{\rho}$ exactly in the same way as \cite{chekanov:dmj} did, it follows that
$\Psi_{H^{\rho}}$ changes the filtration level not more that $E_-(H)$ and
$\Psi_{\widetilde H^{\rho}}$ not more than $E_+(H)$.

Then we consider the composition
\be\label{eq:PsitildeHrhoHrho}
\Psi_{\widetilde H^\rho}\circ \Psi_{H^\rho}:  CI_\lambda^*(R,R) \to  CI_\lambda^*(R,R).
\ee
Note that the function $\chi_K$ is the obvious gluing of $\rho_K = \rho(\tau +K) $ and 
$\rho(K - \tau))$.
By considering the one-parameter family of Hamiltonians $H_K$
defined in \eqref{eq:HK}, we define a family of homomorphisms
\be\label{eq:PsiHK}
\Psi_{H_K}: CI_\lambda^*(R,R) \to  CI_\lambda^*(R,R)
\ee
with $0 \leq K \leq K_0$ which defines a chain homotopy map
$$
\mathfrak H: CI_\lambda^*(R,R) \to  CI_\lambda^{*-1}(R,R)
$$
between $\Psi_{\widetilde H}\circ \Psi_H$ and
$id$ on $ CI_\lambda^*(R,R)$, i.e., it satisfies
$$
\Psi_{\widetilde H}\circ \Psi_H -id = \delta \mathfrak H + \mathfrak H \delta
$$
on $CI_\lambda^*(R,R)$
\emph{provided the parameterized moduli space
$\CM^{\text{\rm para}}_{[0,K_0]}(M,R;J,H)$ that we considered in the previous section does not
bubble-off}. A standard algebraic argument \cite{floer-fixedpoints},
\cite{chekanov:dmj}, \cite{fooo:polydisks} then shows that all the above 4 maps
\eqref{eq:PsiHrho}, \eqref{eq:PsitildeHrho}, \eqref{eq:PsitildeHrhoHrho} and \eqref{eq:PsiHK} 
for sufficiently large $K > 0$ are defined under the subcomplex
$$
CI_\lambda^{(-\infty, \|H\|+ \epsilon]}
$$
for a sufficiently small $\epsilon > 0$, say $\epsilon < T_\lambda(M,R) - \|H\|$,
as long as no bubbling occurs (see \cite{chekanov:dmj} for the argument used for
Lagrangian Floer homology):  
The inequality \eqref{eq:||H||}
ensures no bubbling by Proposition \ref{prop:pienergy-bound-H} and Theorem \ref{thm:compactness}.
(The relevant gluing result leading to such an algebraic consequence
is obtained in \cite{oh:contacton-gluing} for the current case of contact instantons.)

Therefore we have shown 
$$
(\Psi_{\widetilde H})_*\circ (\Psi_H)_* =  \id \quad \text{\rm on }\, CI_\lambda^{(-\infty, \|H\| + \epsilon]}.
$$
In particular,
$$
 (\Psi_H)_* |_{CI_\lambda^{(-\infty, \|H\| + \epsilon]}(R,R)}
 $$
 is injective. Therefore we have
$$
\rank HI^{(-\infty,\|H\| +\epsilon]}_\lambda(R,R) \geq \rank HI^{(-\infty, \|H\| + \epsilon]}_\lambda(R,R).
$$
On the other hand, the inequality $\|H \| + \epsilon < T_\lambda(M,R)$ implies
$$
HI^{(-\infty, \|H\| + \epsilon]}_\lambda(R,R) \cong HI^{(-\infty, \epsilon]}_\lambda(R,R)
$$
for all sufficiently $\epsilon > 0$.

It remains to show that $HI_\lambda^{(-\infty,\epsilon]}(R,R) \cong H^*(R)$.
This can be shown by a couple of ways either by the Morse-Bott argument 
(or the PSS-type argument) or by
the direct comparison argument with the Morse complex of the generating function as done in
\cite{fukaya-oh,milinkovic-pacific} in a Darboux-Weinstein neighborhood
after localizing the cohomology as in \cite{oh:imrn}.
(We refer readers to \cite{oh:contacton-gluing} for the reduction to the case of
one-jet bundle and to \cite{sandon-translated}, \cite{oh-yso:spectral} for the details of computation
for the case of one-jet bundle in the current contact context.)
This finishes the proof.
\end{proof}

\part{Energy bounds, bubbling analysis and weak convergence}
\label{part:bubbling}

In this part, we assume that $(M,\lambda)$ is a tame contact manifold. Then we consider
$J_K$ is the two-parameter family of $\lambda$-adapted CR almost complex structures defined by
$$
J_K(\tau,t) = (\psi_{H}^{\chi_K(\tau)t}(\psi_{H}^{\chi_K(\tau)})^{-1})^*J_K'(\tau,t)
$$
which is associated to the family given in \eqref{eq:Jst}, and the two-parameter family of
Hamiltonians $H_K = H^{\chi_K}$ given in \eqref{eq:HK}. By the boundary-flatness assumed in
\eqref{eq:bdy-flat}, $J_K' \equiv J_0$ for all $\tau$ with $|\tau|$ sufficiently large.

Then we study the equation \eqref{eq:K}
which we recall here:
\be\label{eq:K-2}
\begin{cases}
(du - X_{H^{\chi_K}}(u))^{\pi(0,1)} = 0, \quad d\left(e^{g_K(u)}(u^*\lambda + u^*H_K dt) \circ j\right) = 0,\\
u(\tau,0) \in R,\, u(\tau,1) \in R
\end{cases}
\ee
where $g_K$ is the function
on $\Theta_{K_0+1}$ given in \eqref{eq:gKu} for $0 \leq K \leq K_0$.

We will develop the necessary analytic package which provides
the definition of relevant off-shell energy and the  bubbling argument
and construct the compactification of the moduli space
$$
\CM(M,R;J,H), \quad \CM^{\text{\rm para}}(M,R;J,H)
$$
of solutions of \eqref{eq:K-2},
and other relevant moduli space of (perturbed) contact instantons.

For the purpose of compactification of the moduli space of contact instantons, identifying
a suitable notion of the $\lambda$-energy that controls the $C^1$-estimate is the key element in the
compactness study of moduli space of contact instantons. The framework we will
use is the one from \cite{oh:contacton} where the closed string case is studied.
In the closed string case, the asymptotic charge $Q$ may not be zero which is the
key obstacle to the compactification and the Fredholm theory.  Since in our current
case, the asymptotic charge vanishes by Theorem \ref{thm:subsequence-convergence}, which is
proved in \cite{oh:contacton-Legendrian-bdy},
this bubbling-off analysis applies to general Legendrian boundary condition.

\section{A priori uniform $\pi$-energy bound}
\label{sec:pienergy-bound}

In this section and henceforth, we simplify the notation $E_{J_{K_0},H}$ to $E$
since the pair $(J_{K_0},H)$ will not be changed. And we will also highlight the domain
almost complex structure dependence thereof by considering the pair $(j,w)$ instead of $w$.

\begin{defn}\label{defn:pi-energy}
For a smooth map $\dot \Sigma \to M$, we define the $\pi$-energy of $w$ by
\be\label{eq:Epi}
E^\pi(j,w) = \frac{1}{2} \int_{\dot \Sigma} |d^\pi w|^2.
\ee
\end{defn}

We start with the following crucial lemma. The symplectic counterpart of this lemma is
well-known to the experts and has been used much in quantitative symplectic topology via filtered
symplectic Floer theory.

\begin{lem}\label{lem:muK-family}
Let $H = H(t,x)$ be given. Consider the 2-parameter Hamiltonian $H^s(t,x)$ defined by
\be\label{eq:Hs}
H^s(t,x) = \Dev_\lambda(t \mapsto \psi_H^{st}).
\ee
Then we have
\bea\label{eq:Devts}
\Dev_\lambda(t \mapsto \psi_{H}^{\chi_K(\tau)t})(\tau,t,x) & = & \chi_K(\tau) H(\chi_K(\tau)t,x) 
\label{eq:Devt}\\
\Dev_\lambda \left(\tau \mapsto \psi_{H}^{\chi_K(\tau)t}\right)(\tau,t,x) & = & \chi_K'(\tau)  t H(\chi_K(\tau)t,x).
\label{eq:Devtau}
\eea
\end{lem}
\begin{proof} We recall the generating Hamiltonian of any contact vector field $X$ is given by
$-\lambda(X)$ by definition. Therefore it suffices to compute the associated generating vector fields.
Write
$$
\mu^K(\tau,t,x) = \psi_{H}^{\chi_K(\tau)t}(x).
$$
For the $t$-Hamiltonian, we compute
$$
\frac{\del \mu_K}{\del t} \circ \mu_K^{-1} (\tau,t,x)= \chi_K(\tau) X_{H_{\chi_K(\tau)t}}(x).
$$
Therefore we obtain
$$
\Dev_\lambda(t \mapsto \mu^K_{(\tau,t)})(x) = - \lambda\left( \chi_K(\tau) X_{H_{\chi_K(\tau)t}}(x)\right) = 
\chi_K(\tau) H(\chi_K(\tau)t, x).
$$
Similarly we compute
$$
\frac{\del \mu_K}{\del \tau} \circ \mu_K^{-1} (\tau,t,x) = \chi_K'(\tau) t X_{H_{\chi_K(\tau)t}}(x)
$$
and hence
$$
\Dev_\lambda\left(\tau \mapsto \mu^K_{(\tau,t)}\right) (x)= - \lambda\left( \chi_K'(\tau) X_{H_{\chi_K(\tau)t}}(x)\right) = 
\chi_K'(\tau)t\,  H(\chi_K(\tau)t,x).
$$
This finishes the proof.
\end{proof}

We now split our discussion on the $\pi$-energy bound into the two cases of 
domain $\Theta_{K_0+1}$ and $\R \times [0,1]$.

\subsection{The case of $\Theta_{K_0+1}$}

Let $\dot \Sigma = \Theta_{K_0+1}$ which is a domain of disc-type without puncture.
Proposition \ref{prop:pienergy-bound-H} immediately follows from the following 
identity.

\begin{prop}\label{prop:pienergy-bound2} 
Let $u$ be any finite energy solution of \eqref{eq:K}. Then we have
\bea\label{eq:EJKu=}
E_{(J_K,H_K)}^\pi(u) &=&
\int_{-2K-1}^{-2K} - \chi_K'(\tau) \, H \left(\chi_K(\tau), \psi_{H}^{\chi_K(\tau)}(u(\tau,0))\right)\, d\tau \nonumber\\
&{}& \quad + \int_{2K}^{2K+1} - \chi_K'(\tau) \,
H\left(\chi_K(\tau),\psi_{H}^{\chi_K(\tau)}(u(\tau,0))\right)\, d\tau.
\nonumber\\
{}
\eea
\end{prop}
\begin{proof}
By the energy identity given in Proposition \ref{prop:EpiJH=Epi}, we have
We decompose the integral into
\bea\label{eq:energy-K}
\int_{\del \Theta_{K+1}} \overline u^*\lambda & = & \int_{[-2K-1, 2K+1]} \overline u^*\lambda|_{t=0}
- \int_{[-2K-1,2K+1]} \overline u^*\lambda|_{t=0} \nonumber\\
&{}& \quad  +  \int_{\del D_{K_0}^+} \overline u^*\lambda - \int_{\del D_{K_0}^-} \overline u^*\lambda.
\eea
We now examine the four summands of the right hand side separately.

The last two terms vanish by the Legendrian boundary condition on $u$,
since $H_K \equiv 0$ and so $\overline u = u$ on the relevant integration domains.
It remains to estimate the first two terms.

We recall the definition
$$
\overline u_K(\tau,t): = \psi_{H}^{\chi_K(\tau)}(\psi_{H}^{\chi_K(\tau)t})^{-1}(u(\tau,t))
$$
from \eqref{eq:ubarK} for the $(\tau,t)$-family of Hamiltonians
$$
G_K(\tau,t,x) = \Dev_\lambda \left(\tau \mapsto \psi_H^{\chi_K(\tau) t}t\right)(\tau,x) = \chi_K'(\tau) H^{\chi_K(\tau)t}(x)
$$
where $H^s(t,x) = \Dev \left(s \mapsto \psi_H^{st}\right)(x)$.
We denote by $\nu_\tau^K$ the $\tau$-path associated to the two-parameter family
\be\label{eq:psiK}
\tau \mapsto \nu_{(\tau,t)}^K := \psi_{H}^{\chi_K(\tau)}\circ (\psi_{H}^{\chi_K(\tau)t})^{-1}
\ee
and let $X_\tau$ be the vector field generating the contact Hamiltonian path. Then
by definition we have
$$
\Dev\left(\tau \mapsto \nu^K_{(\tau,t)}\right) = -\lambda(X_\tau).
$$
We denote
\be\label{eq:GK}
G_K(\tau,t,x) = \Dev\left(\tau \mapsto \nu_{(\tau,t)}^K\right)(x).
\ee
Then we compute
$$
\frac{\del \overline u_K}{\del \tau}
$$
at $t = 0, \, 1$ respectively.

At $t = 1$, $\nu^K_{(\tau,1)} =\id$ and hence $X(\tau,1,x) \equiv 0$.
Then we have
$$
\overline u_K^*\lambda\left(\frac{\del}{\del \tau}\right) (\tau,1)=
\lambda\left(\frac{\del \overline u_K}{\del \tau}(\tau,1)\right)
=\lambda\left(\frac{\del u}{\del \tau}(\tau,1)\right) = 0
$$
by the Legendrian boundary condition put on $u$. This proves that the first integral of
\eqref{eq:energy-K} vanishes.

On the other hand, at $t = 0$, we have
$$
\nu^K_{(\tau,0)} = \psi_{H}^{\chi_K(\tau)},
$$
and so
\be\label{eq:ubarK-t=0}
\overline u_K(\tau,0) = \psi_{H}^{\chi_K(\tau)}(u(\tau,0))
\ee
from \eqref{eq:ubarK}.
We compute
$$
\frac{\del \overline u_K}{\del \tau}(\tau,0)
= d\psi_{H}^{\chi_K(\tau)}\left(\frac{\del u}{\del \tau} (\tau,0)\right)
+  X_{G_K}(\tau,0,\overline u_K(\tau,0))
$$
keeping the moving boundary condition we required at $t = 0$ in mind.
Therefore
\bea\label{eq:ubarKlambda}
&{}& \overline u_K^*\lambda|_{t=0}\left(\frac{\del}{\del \tau}\right)  
= \lambda\left(\frac{\del \overline u_K}{\del \tau}(\tau,0)\right) \nonumber\\
& = & \lambda\left(d\psi_{H}^{\chi_K(\tau)}\left(\frac{\del u}{\del \tau} (\tau,0)\right)\right)
+ \lambda\left(X_{G_K}(\tau,0,\overline u_K(\tau,0)\right) \nonumber\\
& = & - G_K(\tau, 0, \overline u_K(\tau,0))
\eea
The first term
in the second line of the equation vanishes by the Legendrian boundary condition imposed  on $u$ at $t = 0$
since $\nu^K(\tau,0) : = \nu^K_{(\tau,0)} = \psi_{H}^{\chi_K(\tau)}$ are contact diffeomorphisms which preserve contact
distribution.

We note
\be\label{eq:nuK=mu}
\nu^K(\tau,0) = \mu(\chi_K(\tau),0) = \psi_H^{\chi_K(\tau)}
\ee
where we recall $\mu(s,t) = \phi_{H}^{st}$. 
On the other hand, at $t = 0$, we have
$$
\nu^K(\tau,0)(u(\tau,0)) = \psi_H^{\chi_K(\tau)}(u(\tau,0))
$$
and so
\be\label{eq:ubarK-t=0}
\overline u_K(\tau,0) = \psi_H^{\chi_K(\tau)}(u(\tau,0))
\ee
from \eqref{eq:ubarK}. Combining the two, we have derived
\be\label{eq:nuK=mu}
\nu^K(\tau,0) (u(\tau,0)) = \psi_H^{\chi_K(\tau)}(u(\tau,0))
\ee
where $\mu(s,t) = \psi_{H}^{st}$ that already appeared in Lemma \ref{lem:muK-family}.
We have already computed
\be\label{eq:GKu}
G_K (\tau, 0, x )= \Dev_\lambda \left(\tau \mapsto \mu_K(\tau,1)\right)(x) 
= \chi_K'(\tau) H(\chi_K(\tau),x)
\ee
from \eqref{eq:Devtau} in Lemma \ref{lem:muK-family}.

Since $\chi_K'$ is supported in $[-2K-1,-2K] \sqcup [2K,2K+1]$, the second
integral of \eqref{eq:energy-K} is reduced to
\beastar
&{}& \left(\int_{-2K-1}^{-2K} + \int_{2K}^{2K+1} \right)(\overline u_K)^* \lambda|_{t=0} \, d\tau \\
& = & \int_{-2K-1}^{-2K}- \chi_K'(\tau)\,
H\left(\chi_K(\tau) ,\psi_{H}^{\chi_K(\tau)r}(u(\tau,0))\right)\, d\tau \\
&{}& \quad  \int_{2K}^{2K+1} - \chi_K'(\tau) \,
H\left(\chi_K(\tau),\psi_{H}^{\chi_K(\tau)r}(u(\tau,0))\right) \, d\tau
\eeastar
Combining the above calculations, we have finished the proof of \eqref{eq:EJKu=}.
\end{proof}

\subsection{The case of $\R \times [0,1]$}

In this section, we consider $\dot \Sigma = \R \times [0,1]$. 
Similarly as for the case $\Theta_{2K_0+1}$, Proposition \ref{prop:pienergy-bound-H-infty} is an
immediate consequence of the following energy identity.

\begin{prop}\label{prop:pienergy-identity-infty}
Then we have
\beastar
E_{(J_K,H_K)}^\pi(u) & = & \int_{-2K-1}^{-2K}- \chi_K'(\tau)\, 
H\left(\chi_K(\tau) ,\psi_{H}^{\chi_K(\tau)r}(u(\tau,0))\right)\, d\tau \\
&{}& \quad  \int_{2K}^{2K+1} - \chi_K'(\tau) \,
H\left(\chi_K(\tau),\psi_{H}^{\chi_K(\tau)r}(u(\tau,0))\right) \, d\tau \\
&{}& \quad + \int \gamma_+^*\lambda - \int \gamma_-^*\lambda.
\eeastar
\end{prop}
\begin{proof}  We start with the identity
$$
E_{(J_K,H_K)}^\pi(u) = E_{J'}^\pi(\overline u) = \int_{\R \times [0,1]} \overline u^*d\lambda.
$$
We compute
$$
 \int_{\R \times [0,1]} \overline u^*d\lambda
 =\int_{\R \times \{0,1\}} \overline u^*\lambda
 + \int (\overline \gamma_+)^*\lambda - \int (\overline \gamma_-)^*\lambda.
 $$
 Since $\psi_H^{\chi_K(\pm\tau)} = \id$ for $|\tau| \geq 2K+1$, the last two integral becomes
 $$
 \int \gamma_+^*\lambda - \int \gamma_-^*\lambda.
 $$
 For the first integral,  we evaluate 
\be\label{eq:energy-K-bar}
\int_{\R \times \{0,1\}} \overline u^*\lambda = \int_\R  \overline u^*\lambda|_{t=0}
- \int_\R \overline u^*\lambda|_{t=0} 
\ee
We need to  examine the two summands of the right hand side separately.
This part of calculation is exactly the same as the case of $\Theta_{2K_0+1}$
by the support condition on $\chi_K$.
This proves
\beastar
\int_{\R \times [0,1]}\overline u^*d\lambda 
& = & \int_{-2K-1}^{-2K}- \chi_K'(\tau)\,
H\left(\chi_K(\tau) ,\psi_{H}^{\chi_K(\tau)r}(u(\tau,0))\right)\, d\tau \\
&{}& \quad  \int_{2K}^{2K+1} - \chi_K'(\tau) \,
H\left(\chi_K(\tau),\psi_{H}^{\chi_K(\tau)r}(u(\tau,0))\right) \, d\tau
\eeastar
as before.
Combining the above calculations, we have finished the proof of \eqref{eq:EJKu=}.
\end{proof}

\section{Definition of off-shell vertical energy of contact instantons}
\label{sec:lambda-energy}

The vertical energy for a general
smooth map $w: \dot \Sigma\to M$, which is denote by $E^\perp(j,w)$, is given
in \cite{oh:contacton} for a general punctured Riemann surface $(\dot \Sigma, j)$
for the closed string case which can be easily adapted to the open string case as in
our paper.   Defining the vertical part of the energy is not as straightforward as in the
symplectization case introduced by Hofer \cite{hofer:invent}, because the form
$w^*\lambda \circ j$ is only required to be closed but not exact in general and
we do not have the presence of the background radial function $s$ for the case of
contact instanton equation.

Luckily, the Riemann surfaces that
are relevant to the purposes of the present paper are of the following three types:
\begin{situ}[Charge vanishing]\label{situ:charge-vanish}
\begin{enumerate}
\item
First, we mention that the \emph{starting} Riemann surface will be an open Riemann surface of genus zero
with a finite number of boundary punctures, and mostly
$$
\dot \Sigma \cong \R \times [0,1]
$$
together with an contact instanton with Legendrian pair boundary condition $(R_0,R_1)$, or
$$
\dot \Sigma \cong D^2
$$
with \emph{moving Legendrian boundary condition}.
\item $\C$ which will appear in the bubbling analysis at an interior point of $\dot \Sigma$,
\item $\H = \{ z \in \C \mid \Im z \geq 0\}$ which will appear in the bubbling analysis at a boundary
point of $\dot \Sigma$.
\end{enumerate}
\end{situ}
An upshot is that \emph{the asymptotic charges vanish in all these three cases.}
Therefore in the present paper, 
\emph{ it will be enough to consider the punctures (both interior or boundary)
where the asymptotic charges vanish, which we will assume from now on.}
Since our main interest in the present paper concerns the domain $\R \times [0,1]$,
we can simplify the definition of vertical energy as follows.

As discovered by \cite{hofer:invent} in \cite{{hofer:invent}} in the context of symplectization, one
needs to examine the $R_\lambda$-part of energy that controls the asymptotic behavior of
contact instantons near the puncture. For this purpose, the Hofer-type energy
introduced in \cite{{hofer:invent}} is crucial. In this section, we generalize this energy
to the general context  without involving the symplectization, by 
introducing the notion of \emph{contact instanton potential} for any smooth 
map $w: \dot \Sigma \to M$ satisfying the closedness condition
$$
d(w^*\lambda \circ j) = 0
$$
not necessarily satisfying $\delbar^\pi w = 0$.

Following \cite{{hofer:invent}}, \cite{behwz} and \cite{oh:contacton},
 we introduce the following class of test functions
\begin{defn}\label{defn:CC} We define
\be
\CC =\{\varphi: \R \to \R_{\geq 0} \mid \supp \varphi \, \text{is compact}, \, \int_\R \varphi = 1\}
\ee
and define $\psi$ by $\psi(s): = \int_{-\infty}^s\, \varphi(r)\, dr$.
\end{defn}

For the purpose of emphasizing the charge vanishing for the open-string context
established in \cite{oh-yso:index} and the current circumstance of construction of an $A_\infty$ category
\cite{kim-oh:category} for which we consider the domain $\dot \Sigma$ of \emph{disc-type surface} which is
simply connected,  we first verbatim recall the definition of $\lambda$-energy introduced 
from \cite{oh:contacton} for the closed string context, which we apply to the open string context 
without much change in general. (See Section 4 \& 5 \cite{oh:contacton} for the details.)

By the simply connectedness of $\Theta_{K_0+1}$ or $\R \times [0,1]$, we can write
$$
w^*\lambda \circ j = dg,
$$
for some globally defined function $g: \cdot \Sigma \to \R$, i.e.,
i.e., $w^*\lambda\circ j$ becomes an exact one-form on $\dot \Sigma$. 
In terms of the strip-like coordinates
$(\tau,t)$ near a puncture, we can express
$g = g(\tau,t)$ and then we have
$$
g = f + c
$$
for some constant $c$. 

\begin{rem} We note that one can lift any smooth exact contact instanton $w: \dot \Sigma \to M$
can be lifted to the symplectization $M \times \R$ by setting $u: = (w,f)$ for each individual instanton $w$
modulo a shift to the radial direction of $M \times \R$.
\end{rem}

\begin{defn}[$E_\CC$-energy]\label{defn:CC-energy} Let $w$ satisfy $d(w^*\lambda \circ j) = 0$. Then we define
$$
E_{\CC}(j,w) = \sup_{\varphi \in \CC} \int_\Sigma d(\psi(g)) \wedge dg\circ j
= \sup_{\varphi \in \CC} \int_\Sigma d(\psi(g)) \wedge (- w^*\lambda).
$$
\end{defn}
 We note that
$$
d(\psi(g)) \wedge dg \circ j = \psi'(g) dg \wedge dg\circ j= \varphi(g) dg \wedge dg\circ j \geq 0
$$
since
$$
dg \circ j  \wedge dg = |dg|^2\, d\tau \wedge dt.
$$
Therefore we can rewrite $E_{\CC}(j,w)$ into
\be\label{eq:finallambdaenergy}
E^\lambda(j,w) = \sup_{\varphi \in \CC} \int_\Sigma d(\varphi(g)) \wedge dg\circ j.
\ee
This function $g$ seems to deserve a name.

\begin{defn}[Contact instanton potential] 
Let $w$ be a smooth map satisfying $d(w^*\lambda \circ j) = 0$.
We call the above normalized function $g$
the \emph{contact instanton potential} of $w$ and the one-from $w^*\lambda \circ j$ 
the contact instanton charge form of $w$.
\end{defn}
Now we define the final form of the off-shell energy.

\begin{defn} Let $w:\dot \Sigma \to Q$ be any smooth map.
We define the total energy of $w$ by
\be\label{eq:total-energy}
E(j,w) = E^\pi_\lambda (j,w) + E^\lambda(j,w).
\ee
\end{defn}
In the rest of the paper, we suppress $j$ from the arguments of
the energy $E(j,w)$ and just write $E(w)$ for the simplicity of notations unless necessary 
to emphasize the dependence on $j$. 

We now collect a series results proved in \cite{oh:contacton} 
in the closed-string context whose proofs equally apply without change
except slight adaptations to the current open-string context.

The following lemma was proved in \cite{oh:contacton} for the closed-string context
whose proof also applies to the open-string context without change.
\begin{lem}[Lemma 8.1 \cite{oh:contacton}]\label{lem:|dw|<infty} 
Suppose $E(w) = E^\pi(w) + E^\lambda(w) < \infty$. Then
$$
\|dw\|_{C^0} < \infty.
$$
\end{lem}

\begin{prop}[Proposition 9.2 \cite{oh:contacton}]\label{prop:proper-energy} 
Suppose that $E^\pi(w) < \infty$ and the function $g:\dot \Sigma \to \R$ is proper.
Then $E(w) < \infty$.
\end{prop}

\begin{rem}
We remark that when $w$ is given, 
the function $g = g(\tau,t)$ on $D_\delta(p) \setminus \{p\}$ in strip-like coordinates
is uniquely determined modulo
the shift by a constant. However  the definition 
of $E_{\CC}(j,w)$ does not depend on the constant shift in the choice of $g$.
This  enables us to introduce the notion of (local) vertical energy at a puncture $p$:
We denote the common value of $E_{\CC,f}(j,w;p)$ by $E^\lambda_p(w)$, 
and call the \emph{$\lambda$-energy at $p$}. This local version does not
enter in the present study.
\end{rem}
%
%The following then would be the preliminary definition of the total energy.
%Let $w:\dot \Sigma \to Q$ be any smooth map.
%We take the sum and write
%\be\label{eq:total-energy}
%E(j,w) = E^\pi(j,w) + \sum_{l=1}^k E^\lambda_{p_l}(w).
%\ee
%\begin{rem}
%If $\dot \Sigma \cong \C$, e.g., $\dot \Sigma$ is $\C P^1$ with  one puncture,
%there exists a globally defined function $f: \dot \Sigma \to \R$ such that $w^*\lambda \circ j = df$
%in which case we may regard the pair $(w,f)$ as a pseudoholomorphic map to the
%symplectization in the sense of \cite{hofer:invent}.
%\end{rem}

We also need to consider the these energies for the domain dependent family of 
triads $(M,\lambda_z,J_z)$ with $z\in \dot \Sigma$. 

Now we define the final form of the off-shell energy. For the simplicity of notations, we
will drop $\lambda$ from the notations.

\begin{defn}[Total energy]\label{defn:total-enerty}
Let $(M,\lambda_z,J_z)$ with $z\in \dot \Sigma$ be a domain dependent triads constant
near punctures. For any smooth map $w:\dot \Sigma \to Q$,
we define the \emph{total energy} to be
the sum
\be\label{eq:final-total-energy}
E(j,w) = E^\pi(j,w) + E^\perp(j,w).
\ee
\end{defn}

In the rest of the paper, we suppress $j$ from the arguments of
the energy $E(j,w)$ and just write $E(w) = E^\pi(w) + E^\perp(w)$.

\section{A priori uniform vertical energy bound}
\label{sec:C1-estimate}

In this section, we establish the $C^1$-estimates and weak convergence result for
the parameterized moduli space
$$
\CM^{\text{\rm para}}(M,R;J,H)
$$
defined in \eqref{eq:M-para-K0} for $K_0 = \infty$.

For this purpose, we need to establish the fundamental a priori energy bound for the energy
$$
E_{J'}(\overline u_K) = E_{J'}^\pi(\overline u_K) + E_{J'}^\perp(\overline u_K)
$$
where $\overline u_K(\tau,t) = \psi_H^1 (\psi_H^t)^{-1}(u(\tau,t))$.

We have already shown $E_{J,H}^\pi(u) = E_{J'}^\pi(\overline u)$ in Proposition \ref{prop:EpiJH=Epi},
and already proved the $\pi$-energy bound
$$
E_{J'}^\pi(\overline u_K) \leq \|H\|
$$
in Proposition \ref{prop:pienergy-bound-H}.

In this section, we  complete the study of energy bounds by proving the bound for the $\lambda$-energy as well.

\begin{prop}\label{prop:lambdaenergy-bound2}
Let $u$ be any finite energy solution of \eqref{eq:K}. Then we have
\be\label{eq:lambdaenergy-bound2}
E_{J'}^\perp(\overline u_K) \leq \|H\|.
\ee
\end{prop}
\begin{proof} For the simplicity and to highlight that $\overline u_K$ is
an unperturbed contact instanton, we denote
$$
\overline u_K =: w
$$
unless the original notation $\overline u_K$ needs to be used for clarity.

By the defining equation $d(w^*\lambda \circ j) = 0$ of contact instantons,
there exists a global function $f: \Theta_{K_0+1} \to \R$ such that
$$
w^*\lambda \circ j = df
$$
since $\Theta_{K_0 + 1}$ is simply connected. By definition of $E^\perp$, we need to get a
uniform bound for the integral
$$
\int_{\Theta_{K_0+1}} d(\psi(f)) \wedge (-w^*\lambda) \geq 0.
$$
(See Definition \ref{defn:CC-energy}.) By integration by parts, we rewrite
$$
\int_{\Theta_{K_0+1}} (-w^*\lambda) \wedge d(\psi(f))
= \int_{\Theta_{K_0+1}} d(\psi(f) w^*\lambda) - \psi(f)  dw^*\lambda.
$$
Recall that  $dw^*\lambda = \frac12 |d^\pi w|^2$ for any
$w$ satisfying $\delbar^\pi w = 0$. Then similarly as we proved
Proposition \ref{prop:pienergy-bound-H}, we derive
\beastar
0 & \leq & \int_{\Theta_{K_0+1}} (-w^*\lambda) \wedge d(\psi(f)) \leq \int_{\Theta_{K_0 +1}} d(\psi(f) w^*\lambda)\\
& = & \left(\int_{\del D_{K_0}^+ } \psi(f(\infty,t)) ({\overline \gamma_+})^*\lambda
- \int_{\del D_{K_0}^-} \psi(f) ({\overline \gamma_-})^*\lambda \right)\\
&{}& + \int_{-2K_0-1}^{2K_0+1} \psi(f(\tau,0)) \lambda\left(\frac{\del \overline u_K}{\del \tau}(\tau,0) \right)\, d\tau \\
&{}& \quad  -
\int_{-2K_0-1}^{2K_0+1} \psi(f(\tau,1)) \lambda\left(\frac{\del \overline u_K}{\del \tau}(\tau,1) \right)\, d\tau\\
& = & \int_{-2K_0 -1}^{2K_0+1} \psi(f(\tau,0)) \lambda\left(\frac{\del \overline u_K}{\del \tau}(\tau,0) \right)\, d\tau \\
&{}& \quad -
\int_{-2K_0-1}^{2K_0+1} \psi(f(\tau,1)) \lambda\left(\frac{\del \overline u_K}{\del \tau}(\tau,1) \right)\, d\tau\\
& \leq & \int_0^1 - \min \psi(f(\cdot,t)) H_t)\, dt + \int_0^1 \max (\psi(f(\cdot,t)) H_t) \, dt \\
& \leq & \int_0^1 (\max  H_t - \min H_t) \, dt = \|H\|.
\eeastar
Here for the penultimate inequality, we employ the following:
\begin{itemize}
\item
We use the same calculations as the ones performed in the proof of
Proposition \ref{prop:pienergy-bound2}, especially in the
course of evaluation of the integral
$$
\int_{[-2K-1,2K+1]} \overline u^*\lambda|_{t = 0}
$$
in \eqref{eq:energy-K-bar}, we apply \eqref{eq:ubarKlambda} and \eqref{eq:GKu}.
\item Moreover, we also have used the inequality
$$
\chi_K'
\begin{cases}
 \geq 0 \quad & \text{\rm for }\, \tau \in [-2K_0-1,-2K_0]\\
\leq 0 \quad & \text{\rm for }\, \tau \in [2K_0,2K_0+1].
\end{cases}
$$
\end{itemize}
Then for the last equality, we use the fact $0 \leq \psi(f) \leq 1$.
Combining all the above discussion, we have finished the proof of
$$
\int_{\Theta_{K_0+1}} d(\psi(f)) \wedge  (-w^*\lambda) \leq \|H\|
$$
for any $\psi \in \CC$ and hence the proof of the proposition by definition of
$E^\perp$.
\end{proof}

\section{Bubbling analysis for the contact instantons}
\label{sec:bubbling-analysis}

As in Hofer's bubbling-off analysis in pseudo-holomorphic curves on symplectization \cite{hofer:invent}, it
turns out that the study of contact instantons on the plane for the closed string case, and
on the half plane in addition for the open string case, plays a crucial role
in the bubbling-off analysis of contact instantons. Overall bubbling arguments is by now standard
which we apply to the new case of conformally invariant elliptic boundary value problem,
\emph{the contact instanton equation with Legendrian boundary condition}.
We refer readers to \cite[Section 8.4]{oh:book1} for the full details of this bubbling argument,
especially for the process of disc bubbling for pseudoholomorphic curves with
Lagrangian boundary condition in symplectic geometry.

However there are two marked differences between the current blowing-up argument and that of
pseudoholomorphic curves:

\begin{rem}
\begin{enumerate}
\item We need to replace the standard harmonic energy by
the $\pi$-harmonic energy for the blowing-up argument
in the proofs of the $\epsilon$-regularity or of the
period-gap theorem. Of course, we also need the uniform bound for the $\lambda$-energy
which itself, however,  does not enter in this proof other than the presence of uniform bound.
\item  Because of the presence of the equation
$$
d(w^*\lambda\circ j) = 0
$$
in the defining equation of the contact instanton map which is
\emph{of the second order, not of the first order}, the local $C^{k,\alpha}$ a priori estimate
for $k =2$ given in \cite{oh:contacton-Legendrian-bdy} is the minimum regularity needed to
apply the bubbling argument to establish that the limit
map of a subsequence obtained via application of Ascoli-Arzela theorem still
satisfies the equation
$$
\delbar^\pi w = 0, \quad d(w^*\lambda \circ j) = 0.
$$
The upshot is that the  $C^{1,\alpha}$ bound is not enough
to carry out the bubbling process and
produce a limit for the contact instanton map, unlike the case of pseudoholomorphic curves.
\end{enumerate}
\end{rem}

We recall the following useful lemma from \cite{hofer-viterbo} whose proof we refer thereto.
\begin{lem}\label{lem:Hofer-lemma} Let $(X,d)$ be a complete metric space, $f: X \to \R$ be a
nonnegative continuous function, $x \in X$ and $\delta > 0$. Then there exists $y \in X$ and
a positive number $\epsilon \leq \delta $ such that
$$
d(x,y) < 2 \delta, \, \max_{B_y(\epsilon)} f \leq 2 f(y), \, \epsilon\, f(y) \geq \delta f(x).
$$
\end{lem}

We start with the case of $\C$.

\subsection{Contact instantons on the plane}

We start with a proposition which is an analog to
\cite[Theorem 31]{hofer:invent}. We refer to \cite{oh:contacton} or \cite{oh-savelyev}
for its proof. We also refer the proof of the corresponding statement of Proposition \ref{prop:C1-bdy}
for the case of contact instantons on the half place, which is harder to prove than
the case of the plane.

\begin{prop}\label{prop:C^1}
Let $w: \C \to Q$ be a contact instanton. Regard $\infty$ as
a puncture of $\C  = \C P^1\setminus \{\infty\}$. Suppose $\|dw\|_{C^0} < \infty$ and
\be\label{eq:asymp-densitybound}
E^\pi(w)=0, \quad  E^\perp_{\infty}(w) < \infty.
\ee
Then $w$ is a constant map.
\end{prop}

Using the above proposition, we prove the following fundamental result.

\begin{thm}\label{thm:C1bound} Let $w: \C \to Q$ be a contact instanton.
Suppose
\be\label{eq:Epi-bound}
E(w) = E^\pi(w) + E^\perp_\infty(w) < \infty.
\ee
Then $\|dw\|_{C^0} < \infty$.

\end{thm}
\begin{proof}
Suppose to the contrary that $\|dw\|_{C^0} = \infty$ and let $z_\alpha$ be a blowing-up
sequence. We denote $R_\alpha = |dw(z_\alpha)| \to \infty$. Then by applying Lemma \ref{lem:Hofer-lemma},
we can choose another such sequence $z_\alpha'$ and $\epsilon_\alpha \to 0$ such that
\be\label{eq:blowingup-sequence}
|dw(z_\alpha')| \to \infty, \quad \max_{z \in D_{\epsilon_\alpha}(z_\alpha')}|dw(z)| \leq 2 R_\alpha,
\quad \epsilon_\alpha R_\alpha \to 0.
\ee
We consider the re-scaling maps $v_\alpha: D^2_{\epsilon_\alpha R_\alpha}(0) \to Q$
defined by
$$
v_\alpha(z) = w \left(z_\alpha' + \frac{z}{R_\alpha}\right).
$$
Then we have
$$
\|d v_\alpha\|_{C^0; \epsilon_\alpha R_\alpha} \leq 2, \quad |d v_\alpha(0)|=1.
$$
Applying Ascoli-Arzela theorem, there exists a continuous map $v_\infty: \C \to Q$ such that
$ v_\alpha \to v_\infty$ uniformly on compact subsets. Then by the a priori $C^{k,\alpha}$-estimates,
\cite[Section 5]{oh:contacton-Legendrian-bdy},
the convergence is in compact $C^\infty$ topology and $v_\infty$ is smooth. Furthermore
$v_\infty$ satisfies $\delbar^\pi v_\infty = 0 = d(v_\infty^*\lambda \circ j) = 0$, $E^\perp(v_\infty) \leq E(w) < \infty$
and
\be\label{eq:dwinfty0=1}
\|d v_\infty \|_{C^0; \C} \leq 2, \quad |dv_\infty(0)|=1.
\ee
On the other hand, by the finite $\pi$-energy hypothesis and density identity
$$
\frac12|d^\pi w|^2 \, dA = d(w^*\lambda),
$$
we derive
\beastar
0 & = & \lim_{\alpha \to \infty} \int_{D_{\epsilon_\alpha R_\alpha}(z_\alpha')} d(v^*\lambda)
= \lim_{\alpha \to \infty} \int_{D_{\epsilon_\alpha R_\alpha}(z_\alpha')} d( v_\alpha^*\lambda)\\
&= & \lim_{\alpha \to \infty} \int_{D_{\epsilon_\alpha R_\alpha}(z_\alpha')} |d^\pi v_\alpha|^2
= \int_\C |d^\pi v_\infty|^2 < \infty.
\eeastar
 Therefore we derive
$$
E^\pi(v_\infty) = 0.
$$
Then Proposition \ref{prop:C^1} implies $v_\infty$ is a constant map which contradicts to
$$
|dv_\infty(0)| = 1
$$
in \eqref{eq:dwinfty0=1}. This finishes the proof.
\end{proof}

An immediate corollary of this theorem and Proposition \ref{prop:C^1} is the following

\begin{cor}\label{cor:pi-positive} For any non-constant contact instanton $w: \C \to Q$ with the
energy bound $E(w) < \infty$, we obtain
$$
E^\pi(w) = \int \gamma_w^*\lambda > 0
$$
for $\gamma_w = \lim_{R \to \infty} w(R e^{2\pi it})$. In particular $E^\pi(w) \geq T(M,\lambda)> 0$.
\end{cor}

Now we have the following refinement of the
asymptotic convergence result from \cite{hofer:invent} and \cite{oh-wang:connection}. It is a refinement
 of \cite[Theorem 6.3]{oh-wang:connection}
in that the derivative
bound $\|dw\|_{C^0} < \infty$ imposed therein is replaced by the more natural energy bound $E(w) < \infty$.

Combining Theorem \ref{thm:subsequence-convergence} and Theorem \ref{thm:C1bound}, we immediately derive

\begin{cor} Let $w$ be a non-constant contact instanton on $\C$ with
\be\label{eq:C1-densitybound}
E(w) < \infty.
\ee
Then there exists a sequence $R_j \to \infty$ and a Reeb orbit $\gamma$ such that
$z_{R_j} \to \gamma(T(\cdot))$ with $T \neq 0$ and
$$
T = E^\pi(w), \quad Q = \int_z w^*\lambda \circ j = 0.
$$
\end{cor}
\begin{proof} If $T = 0$, the above theorem shows that there exists a sequence
$\tau_i \to \infty$ such that $w(\tau_i,\cdot)$ converges to a constant in $C^\infty$
topology and so
$$
\int_{\{\tau = \tau_i\}} w^*\lambda \to 0
$$
as $i \to \infty$. By Stokes' formula, we derive
$$
\int_{D_{e^{\tau_i}}(0)} w^*d\lambda = \int_{\tau = \tau_i} w^*\lambda \to 0.
$$
On the other hand, we have
$$
E^\pi(w) = \lim_{i \to \infty} \frac12 \int_{D_{e^{\tau_i}}(0)} |d^\pi w|^2
 = \lim_{i \to \infty}  \int_{D_{e^{\tau_i}}} w^*d\lambda = 0.
$$
This contradicts to Corollary \ref{cor:pi-positive}, which finishes the proof.
\end{proof}

The following is the analog to \cite[Proposition 30]{hofer:invent}.

\begin{cor}\label{cor:C^1oncylinder} Let $w$ be a contact instanton on $\R \times S^1$ with $E(w) < \infty$.
Then $\|dw\|_{C^0} < \infty$.
\end{cor}
\begin{proof} As in Hofer's proof of \cite[Proposition 30]{hofer:invent},
we apply the same kind of bubbling-off argument as that of Theorem \ref{thm:C1bound}
and derive the same conclusion. Since the arguments are essentially the same, we
omit the details by referring readers to the proof of \cite[Proposition 30]{hofer:invent}.
\end{proof}

Now we consider the case $w: \H \to M$ with $w(\del \H) \subset R$ (resp. $w(\del \H) \subset Z$)
with Legendrian boundary condition (resp. with co-Legendrian boundary condition).

\subsection{Contact instantons on $\mathbb H$ with Legendrian boundary condition}

Let $R$ be a compact Legendrian submanifold on general contact manifold $(M,\lambda)$.

\begin{prop}\label{prop:C1-bdy}
Let $w: (\H, \del \H) \to (M,R)$ be a contact instanton
\be\label{eq:contacton-on-H}
\begin{cases}
\delbar^\pi w = 0, \, d(w^*\lambda \circ j) = 0\\
w(\del \H) \subset R
\end{cases}
\ee
Regard $\infty$ as a puncture of $D^2 \setminus  \{1\} \cong \H$. Suppose $\|dw\|_{C^0} < \infty$ and
\be\label{eq:asymp-densitybound-bdy}
E^\pi(w)=0, \quad  E^\perp_{\infty}(w) < \infty.
\ee
Then $w$ is a constant map valued in $R$.
\end{prop}
\begin{proof} By the same argument used in the beginning of the proof of Proposition \ref{prop:abbas-open}
 using the vanishing $E^\pi(w) = 0$ which implies $d^\p w = 0$, we derive
$$
dw = w^*\lambda\otimes R_\lambda(w)
$$
with $w^*\lambda$ a bounded harmonic one-form.
Since $\H$ is connected, the image of $w$ must be contained
in a single leaf of Reeb foliation.

Then there is a smooth function $b = b(z)$ such that
$$
w(z) = \gamma(b(z))
$$
for a Reeb trajectory $\gamma = \gamma(t)$ as before but this time, $w$ satisfies
the boundary condition
$$
w(\del \Theta) = \gamma(b(\del \Theta)) \subset R.
$$
In particular, we have $(w|_{\del \H})^*\lambda = 0$. We compute
$$
w^*\lambda = b^*\gamma^*\lambda = b^*(dt) = db
$$
which is bounded on $\H$. We also have
$$
db|_{\del \H} \equiv 0
$$
and so $b|_{\del \H}$ is a constant function. Let $b|_{\del \H} \equiv b_0$.
Since we also have $d(w^*\lambda\circ j) = 0$,
$$
d(db \circ j) = 0
$$
i.e., $b: \H \to \R$ is a harmonic function  hence $b$ is the imaginary part of a holomorphic
function, say $f$, with $f(z) = a(z) + ib(z)$ whose gradient is bounded as before. Furthermore
$b$ satisfies
$$
b|_{\del \H} \equiv b_0.
$$
this time. By applying the reflection principle, we obtain a holomorphic function
$$
\widetilde f: \C \to \C
$$
which  has bounded gradient on $\C$. Therefore $\widetilde f(z) = \alpha z + \beta$
for some constants $\alpha, \, \beta \in \C$ as before. Then $b(z) \equiv b_0$ on $\del \H$
implies $\text{\rm Im}\widetilde f|_{\del \H} \equiv b_0$. By the unique continuation applied to
holomorphic functions, we conclude $\widetilde f$ must be constant on $\C$.
This in turn implies $b(z) \equiv b_0$ on $\H$ in particular.
Thanks to the Legendrian boundary condition $w(\del \H) \subset R$, $w$ must be a
constant map valued in $R$. This finishes the proof.
\end{proof}

Again using the above proposition, we prove the following fundamental $C^1$-bound.

\begin{thm}\label{thm:C1bound-R} Let $w: (\H, \del \H) \to (M,R)$ be a solution of
\eqref{eq:contacton-on-H} with
\be\label{eq:Epi-bound-bdy}
E(w) = E^\pi(w) + E^\perp_\infty(w) < \infty.
\ee
Then $\|dw\|_{C^0} < \infty$.
\end{thm}
\begin{proof}
Suppose to the contrary that $\|dw\|_{C^0} = \infty$ and let $z_\alpha$ be a blowing-up
sequence. We denote $R_\alpha = |dw(z_\alpha)| \to \infty$. Again by applying Lemma \ref{lem:Hofer-lemma},
we can choose another such sequence $z_\alpha'$ and $\epsilon_\alpha \to 0$ such that
\be\label{eq:blowingup-sequence-bdy}
|dw(y_\alpha)| \to \infty, \quad \max_{z \in D_{\epsilon_\alpha(y_\alpha)}}|dw(z)| \leq 2 R_\alpha,
\quad \epsilon_\alpha R_\alpha \to 0.
\ee
We consider the re-scaling maps $v_\alpha: D^2_{\epsilon_\alpha R_\alpha}(0) \to M$
defined by
$$
v_\alpha(z) = w \left(y_\alpha + \frac{z}{R_\alpha}\right).
$$
Then we have
$$
|d v_\alpha|_{C^0; \epsilon_\alpha R_\alpha} \leq 2, \quad |d v_\alpha(0)|=1
$$
as before. \emph{Up until now, the proof is the same as that of Theorem \ref{thm:C1bound}.}

Due to the presence of the boundary $\del \H$, we consider two cases separately:
\begin{enumerate}
\item The case $y_\alpha \to y_\infty \in \Int \H$,
\item The case $y_\alpha \to y_\infty \in \del \H$.
\end{enumerate}
We denote
$$
d_\alpha : = \dist(y_\alpha, \del \H).
$$
The case (1) can be treated in the same way as in Theorem \ref{thm:C1bound-R} to produce
a non-constant contact instanton defined on $\C$.

Therefore we will focus on the case (2) from now on.
In this case, the map $v_\alpha$ is defined at least for
those $z$'s satisfying
$$
\text{\rm Im} \left(y_\alpha + \frac{z}{R_\alpha}\right) \geq 0.
$$
In particular, $v_\alpha(z)$ is defined at least on the domain
$$
\Theta_\alpha: = \{z \in \C \mid |z| \leq \epsilon_\alpha R_\alpha, \,
\Im y_\alpha \geq \min\{\epsilon_\alpha R_\alpha, R_\alpha(d_\alpha - \Im y_\alpha) \}.
$$
\emph{At this stage, after applying Ascoli-Arzela theorem,  there are two cases to consider:}
\begin{itemize}
\item there exists a continuous map $v_\infty: \C \to M$ or,
\item there exists a continuous map $v_\infty:(\H,\del \H) \to (M,R)$
\end{itemize}
such that $v_\alpha \to v_\infty$ uniformly on compact subsets.

Since the first case can be studied as before to produce a sphere-bubble, we now focus on the latter case.
By the a priori $C^{k,\alpha}$-estimates, Theorem \ref{thm:higher-regularity},
the convergence is in compact $C^\infty$ topology and $v_\infty$ is smooth. Furthermore
$v_\infty$ satisfies
$$
\delbar^\pi v_\infty = 0 = d(v_\infty^*\lambda \circ j) = 0, \, E^\perp(v_\infty) \leq E(w) < \infty
$$
and
$$
\|d v_\infty\|_{C^0; \C} \leq 2, \quad |dv_\infty(0)|=1
$$
and $v_\infty$ also satisfies the boundary  condition $v_\infty(\del \H) \subset R$.

On the other hand, Combining the finite $\pi$-energy hypothesis, Legendrian boundary condition,
and the density identity
$$
\frac12|d^\pi w|^2 \, dA = d(w^*\lambda),
$$
we again derive
\beastar
0 & = & \lim_{\alpha \to \infty} \int_{D_{\epsilon_\alpha}(y_\alpha)} d(w^*\lambda) =
\lim_{\alpha \to \infty} \int_{D_{\epsilon_\alpha R_\alpha}(y_\alpha)} d( v_\alpha^*\lambda)\\
&= & \lim_{\alpha \to \infty} \int_{D_{\epsilon_\alpha R_\alpha}(y_\alpha)} |d^\pi \widetilde v_\alpha|^2
= \int_\C |d^\pi v_\infty|^2
\eeastar
and hence
$
E^\pi(v_\infty) = 0.
$
Then Proposition \ref{prop:C1-bdy} implies $v_\infty$ is a constant map which contradicts to
$|dv_\infty(0)| = 1$. This finishes the proof.
\end{proof}

An immediate corollary of this theorem and Proposition \ref{prop:C1-bdy} is the following

\begin{cor}\label{cor:pi-positive-R} For any non-constant contact instanton $w: (\H,\del \H) \to (M,R)$ with the
energy bound $E(w) < \infty$, we obtain
$$
E^\pi(w) = \int z^*\lambda > 0
$$
for $z = \lim_{R \to \infty} w(R e^{\pi it})$. In particular $E^\pi(w) \geq T_\lambda(M,R) > 0$.
\end{cor}

Combining Theorem \ref{thm:subsequence-convergence} and Theorem \ref{thm:C1bound-R}, we immediately derive

\begin{cor} Let $w: (\H,\del \H) \to (M,R)$ be a non-constant contact instanton with
\be\label{eq:C1-densitybound-bdy}
E(w) < \infty.
\ee
Then there exists a sequence $R_j \to \infty$ and a Reeb orbit $\gamma$ such that
$z_{R_j} \to \gamma(T(\cdot))$ with $T \neq 0$ and
$$
T = E^\pi(w), \quad Q = \int_z w^*\lambda \circ j = 0.
$$
\end{cor}
\begin{proof} By considering the coordinates $(\tau,t) \in [0,\infty) \times [0,1]$ with
$e^{\tau + it} \in \H$ which is a strip-like coordinate of $\H \cong D^2 \setminus \{1\}$
near $\infty$ of $\H$, the charge vanishing theorem \ref{thm:subsequence-convergence} implies $Q = 0$.

We then prove $T \neq 0$ by contradiction.
If $T = 0$, the above theorem shows that there exists a sequence
$\tau_i \to \infty$ such that $w(\tau_i,\cdot)$ converges to a constant in $C^\infty$
topology and so
$$
\int_{\{\tau = \tau_i\}} w^*\lambda \to 0
$$
as $i \to \infty$. By Stokes' formula and the Legendrian boundary condition, we derive
$$
\int_{D_{e^{\tau_i}}(0)} w^*d\lambda = \int_{\del D_{e^{\tau_i}}(0)} w^*\lambda \to 0.
$$
On the other hand, this implies
$$
E^\pi(w) = \lim_{i \to \infty} \int_{D_{e^{\tau_i}}(0)} |d^\pi w|^2
 = \lim_{i \to \infty}  \int_{D_{e^{\tau_i}}(0)} w^*d\lambda = 0.
$$
This contradicts to Corollary \ref{cor:pi-positive}, which finishes the proof.
\end{proof}

We also have the following.

\begin{cor}\label{cor:C^1onstrip} Let $w$ be a contact instanton on $\R \times [0,1]$ with $E(w) < \infty$.
Then $\|dw\|_{C^0} < \infty$.
\end{cor}
\begin{proof}
We apply the same kind of bubbling-off argument as that of Theorem \ref{thm:C1bound-R}
and derive the same conclusion.
\end{proof}

\begin{rem}
A priori we cannot rule out the possibility $\operatorname{Spec}(M,\lambda)  = \emptyset$ or
$\operatorname{Spec}(M,R;\lambda) = \emptyset$.
Nonemptyness of this set is precisely the content of Weinstein's conjecture or of the Arnold chord
conjecture: The conjecture has been proved by Taubes \cite{taubes}
(resp. by Hutchings-Taubes \cite{hutchings-taubes} respectively in the three dimensional cases
after other scattered results obtained earlier.
\end{rem}

\section{$C^1$-estimates and weak convergence of contact instantons}
\label{sec:C1-estimates-J1B}

In this section, we establish the key convergence result which 
is also the step towards a full compactification of
the moduli space of contact instantons similarly as in the case of pseudoholomorphic curves,
which will be needed, for example, for the construction of Legendrian contact instanton homology
\cite{oh-yso:spectral} and the Legendrian Fukaya-type category on contact manifolds 
\cite{kim-oh:category}.
Because of this, we will state a complete construction of the 
full compactification $\CM^{\text{\rm para}}(M,R;J,H)$ for the readers' convenience and
for the future purpose of the sequels such as \cite{kim-oh:category}, \cite{oh:entanglement2}, 
although this full version is  not needed for the main application 
of the present paper, especially
under the hypothesis $\|H \|< T_\lambda(M,R)$ as in the present paper and 
in \cite{oh:shelukhin-conjecture}. (See Remark \ref{rem:short-cut} for an illustration
of a short cut under the hypothesis.)

For this purpose, we also need to consider a version of $\CM^{\text{\rm para}}(M,R;J,H)$
that has one end open, i.e., a puncture at either $\tau =-\infty$ or at $\tau = \infty$.
We recall the sets $\Theta_\pm$ and $\Theta_{\pm,K+1}$ from \eqref{eq:Theta+-}.
Then we define the following analog to  $\CM^{\text{\rm para}}(M,R;J,H)$
\be\label{eq:MMK+-}
\CM_\pm^{\text{\rm para}}(M,R;J,H): = \bigcup_{K \in \R} \{K\} \times 
 \CM_{\pm,K}(M,R;J,H)
 \ee
 where the moduli spaces $ \CM_{\pm,K}(M,R;J,H)$ are defined as follows.
 As usual, we denote by $\overline{\CM}$ to be the Gromov-Floer type compactification
 of the relevant moduli spaces below.
 
 \begin{defn}[$\CM_{\pm,K}(M,R;J,H)$] Let $\rho_{\pm,K}$ be the elongation function
 $\rho_{+,K}: \rho_K$ defined in \eqref{eq:rhoK} and $\rho_{-,K}$ is defined by 
 $\rho_{-,K}(\tau) : = \rho_K(-\tau+1)$. Then we define the moduli spaces
 $$
 \CM_{\pm,K}(M,R;J,H)
  $$
  in the same way as we define $ \CM_K(M,R;J,H)$ as that of solutions of 
  \eqref{eq:HG1}  associated Hamiltonians $H^{\rho_K}$.
  \end{defn}
 
 With this preparation, 
combining all the results established in the previous sections of Parts \ref{part:shelukhin} and \ref{part:bubbling}, we prove the following alternatives. 

\begin{thm}\label{thm:bubbling}
Consider the moduli space $\CM^{\text{\rm para}}(M,R;J,H)$ under the assumption of 
uniform energy bound, i.e., there exists a constant $C' > 0$ such that
$$
E_{J,H}(u) < C' < \infty
$$
for all $u \in \CM^{\text{\rm para}}(M,R;J,H)$.

Then one of the following alternatives holds:
\begin{enumerate}
\item
There exists some $ C  > 0$ such that
\be\label{eq:dwC0-intro}
\|d u\|_{C^0;\R \times [0,1]} \leq C
\ee
where $C$ depends only on $(M,R;J,H)$ and $\lambda$.
\item There exists a sequence $u_\alpha \in \CM_{K_\alpha}(M,R;J,H)$ with $K_\alpha \to K_\infty \leq K_0$
and a finite set $\{\gamma_j^+\}$ of closed Reeb orbits of $(M,\lambda)$ such that $u_\alpha$
weakly converges to the union
$$
u_\infty = u_{-,0 } + u_0 + u_{+,0}+ \sum_{j=1} v_j + \sum_k w_k
$$
in the Gromov-Floer-Hofer sense, where
$$
\begin{cases}
u_{-,0} \in \overline{\CM}_-^{\text{\rm para}} (M,R;J,H), \\
u_{+,0} \in \overline{\CM}_+^{\text{\rm para}}K_+(M,R;J,H),\\
u_0 \in   \overline{\CM}(M,R;J,H),
\end{cases}
$$
$$
v_j \in \overline{\CM}(M,J_{z_j}';\alpha_j); \quad \alpha_j \in \frak{Reeb}(M,\lambda_{z_j}'),
$$
and
$$
w_k \in \overline{\CM}(M,\psi_{z_j}(R)), J_{z_j}';\beta_k); \quad \beta_k \in \frak{Reeb}(M,R;\lambda).
$$
\end{enumerate}
Here the domain point $z_j \in \del \dot \Theta_{K_\infty +1}$ is the point at which the corresponding bubble is attached.
\end{thm}
\begin{proof} We denote $w = \overline u$ for the simplicity of notation.

This being said, suppose that the case (1) fails to hold so that  there exist a sequence of pairs
$(K_\alpha,x_\alpha)$ with $K_\alpha$ and $x_\alpha \in \Theta_{K_\alpha}$ such that
$$
|dw_\alpha(x_\alpha)| \to \infty, \quad  x_\alpha = (\tau_\alpha,t_\alpha).
$$
We
divide our proof into the two cases, after choosing a subsequence if necessary, one with $K_\alpha \to K_0$
for some $K_0$ and the other with $K_\alpha \to \infty$.

We start with the first case $K_\alpha \to K_0$ for some $K_0 > 0$.
In this case, we have $|\tau_\alpha| \leq K_0 +1$ by definition of the domain $\Theta_{K_0 +1}$.
We denote
\be\label{eq:dalpha}
d_\alpha: = \dist(x_\alpha, \del \Theta_{K_0+1}).
\ee
Similarly as in the proof of Theorem \ref{thm:C1bound-R}, the first case will
gives rise to a non-constant contact instanton on $\C$.

Again we will focus on the second case where $d_\alpha \to 0$. In this case,
regarding $\Theta_{K_0+1}$ as a  subset of $\C$,
we can find a pair of nested discs in $\C$
$$
D' \subset D \subset \C
$$
centered at $x_\infty$ with $\overline{D'} \subset \overset{\circ}D$
and a sequence $\{ w_\alpha \}$ such that the map $w_\alpha$ is defined on
$$
\Theta_{K_0+1} \cap D
$$
and
$$
\delbar^\pi w_\alpha = 0, \quad d(w_\alpha \circ j) = 0
$$
thereon. It also satisfies
\be\label{eq:2to0ptoinfty}
E^\pi_{\lambda,J;D}(w_\alpha) \leq C',\,  E^\perp(w_\alpha) \leq C',
\quad \norm{dw_\alpha}{C^0,D'\cap \Theta_{K_0+1}} \to \infty
\ee
as $\alpha \to \infty$. Obviously we have $d(x_\alpha - x_\infty) = |x_\alpha - x_\infty|  \to 0$.

We choose sufficiently small constants $\delta_\alpha \to 0$ so that
$$
\delta_\alpha |dw_\alpha(x_\alpha)| \to \infty.
$$
We adjust the sequence $x_\alpha$ to $y_\alpha$ by applying
Lemma \ref{lem:Hofer-lemma}, so that $d(y_\alpha, x_\infty) \to 0$ and
\be\label{eq:adjustedy}
\max_{x \in B_{y_\alpha}(\epsilon_\alpha)}|dw_\alpha| \leq 2|dw_\alpha(y_\alpha)|, \quad
\delta_\alpha |dw_\alpha(y_\alpha)| \to \infty.
\ee
We denote $R_\alpha =  |dw_\alpha(y_\alpha)|$ and consider the re-scaled map
$$
\widetilde w_\alpha(z) = w_\alpha\left(y_\alpha + \frac{z}{R_\alpha}\right).
$$
Then the domain of $w_\alpha$ at least includes $z \in \Theta_{K_0+1} \subset \C$ such that
$$
y_\alpha + \frac{z}{R_\alpha} \in D^2(\delta) \cap  \Theta_{K_0+1}.
$$
In particular, $\widetilde w_\alpha(z)$ is defined at least on the domain
$$
\Theta_\alpha: = \{z \in \C \mid |z| \leq \epsilon_\alpha R_\alpha, \,
\Im y_\alpha \geq \min\{\epsilon_\alpha R_\alpha, R_\alpha(d_\alpha - d(y_\alpha,x_\infty)) \}
$$
%Since $d(y_\alpha,x_\infty) \to 0$ and $\delta_\alpha \to 0$ as $\alpha \to \infty$,
%$R_\alpha (\delta - d(y_\alpha,x_\infty))> R_\alpha \epsilon_\alpha$
%eventually, $\widetilde w_\alpha$ is defined on $D^2(\epsilon_\alpha R_\alpha) \subset \C$ for all sufficiently
%large $\alpha$'s.
%Since $\delta_\alpha R_\alpha \to \infty$ by \eqref{eq:adjustedy}, for any given $R>0$,
%$D^2(\delta_\alpha R_\alpha)$ of $\widetilde w_\alpha (z)$ eventually contains $B_{R+1}(0)$.

The case where
$$
\epsilon_\alpha R_\alpha \leq R_\alpha(d_\alpha - d(y_\alpha,x_\infty))
$$
(with $\Im y_\alpha \to 0$)
corresponds to a `sphere bubble' which is
easier to handle than the rest. Therefore we will focus on the remaining cases henceforth.
We divide our discussion on the remaining cases into three after choosing a subsequence if necessary:
\begin{enumerate}
\item The case $\dist(y_\alpha, \{t = 1\}) \to 0$ and so
$
\dist(y_\alpha, \{t = 1\}) <  \dist(y_\alpha, \{t = 0\})
$
\item The case $\dist(y_\alpha, \{t = 0\}) \to 0$ and so $\dist(y_\alpha, \{t = 1\}) > \dist(y_\alpha, \{t = 0\})$.
\item The case $\dist(y_\alpha, \del D_{K_0+1}^\pm) \to 0$.
\end{enumerate}

The arguments needed to study these cases have little difference and so we
focus on the Case (1).

Note that $\delta_\alpha R_\alpha \to \infty$ by \eqref{eq:adjustedy} for any given $R>0$,
$\widetilde w_\alpha (z)$ is defined eventually on
$$
\Theta_{R+1}(0): = \{ z \in B_{R+1}(0) \mid \Im z \geq I_0\}
$$
for some $\{- \infty\} \cup I_0 \in \R $.
Furthermore, we may assume,
$$
\Theta_{R+1}(0) \subset \left\{ z \in \mathbb{C} \mid \eta_\alpha z + y_\alpha \in \overline{D}'\right\}
$$
Therefore, the maps
$$
\widetilde w_\alpha : \Theta_{R+1} (0) \subset \mathbb{C} \to M
$$
satisfy the following properties:
\begin{enumerate}
 \item[(i)]  $E^\pi(\widetilde w_\alpha) \leq C'$, \, $\delbar^\pi \widetilde w_\alpha = 0$, \,
 $E^\perp(\widetilde w_\alpha) \leq C'$,
 (from the scale invariance)
 \item[(ii)] $|d\widetilde w_\alpha(0)|=1 $ by definition of $\widetilde w_\alpha$ and $R_\alpha$,
 \item[(iii)]$\norm{d\widetilde w_\alpha}{C^0,B_1(x)} \leq 2$ for all $x \in B_R(0) \subset D^2(\epsilon_\alpha R_\alpha)$,
 \item[(iv)] $\delbar^\pi \widetilde w_\alpha =0$ and $d(\widetilde w_\alpha^*\lambda \circ j) = 0$,
 \item[(v)] $\widetilde w_\alpha(\del \Theta_{R+1}) \subset R$.
\end{enumerate}

For each fixed $R$, we take the limit of $\widetilde w_\alpha|_{B_R}$, which we denote by $w_R$.
Applying (iii) and then the local $C^{2,\alpha}$ estimates, Theorem \ref{thm:higher-regularity},  we obtain
$$
\norm{d\widetilde w_\alpha}{2,\alpha;B_{\frac9{10}}(x)} \leq C
$$
for some $C=C(R)$. Therefore we have a
subsequence that converges in $C^2$ in each $B_{\frac{8}{10}}(x), x \in \overline D'$. Then
we derive that the convergence is in $C^2$-topology on $B_{\frac{8}{10}}(x)$ for all $x \in \overline D'$ and
in turn on $\Theta_R(0)$.

Therefore the limit $w_R: B_R(0) \to M$ of $\widetilde w_\alpha|_{B_R(0)}$ satisfies
\begin{itemize}
\item[(1)] $\delbar^\pi w_R = 0$, $d(w_R^*\lambda \circ j) = 0$ and
$$
E^\pi(w_R), \,  E^\perp(\widetilde w_\alpha) \leq C',
$$
\item[(2)] $E^\pi(w_R) \leq \limsup_\alpha  E^\pi_{(\lambda,J;B_R(0))}(\widetilde w_\alpha) \leq C'$,
\item[(3)]  Since $\widetilde w_\alpha \to w_R$ converges in $C^2$, we have
$$
\norm{dw_R}{p,B_1(0)}^2 = \lim_{\alpha \to \infty} \norm{d\widetilde w_\alpha}{p,B_1(0)}^2 \geq \frac{1}{2}.
$$
\end{itemize}
By letting $R \to \infty$ and taking a diagonal subsequence argument, we have
derived nonconstant contact instanton map $w_\infty: (\H, \del \H) \to (M, Z)$.

On the other hand, the bound
$E^\pi(w_R) \leq C'$ for all $R$ and again by Fatou's lemma implies
$$
E^\pi(w_\infty) \leq C'.
$$
(By definition of $T_\lambda$, we must have $E^\pi(w_\infty) \geq T_\lambda$.)
Now we examine the effect of the equation
$$
d(w_\infty^*\lambda \circ j) = 0
$$
on $E^\pi(w_\infty)$.

Using the identity $d(w_\infty^*\lambda) = \frac12|d^\pi w_\infty|^2$ and Fatou's lemma,
we have
\beastar
E^\pi(w_\infty) & = & \int_{\H}d(w_\infty^*\lambda) = \lim_{R \to \infty} \int_{\H_R} d(w_\infty^*\lambda)\\
& = & \lim_{r \to \infty}\int_{{\{|x| \leq r, \, y = 0\} \cup \H \cap \{|z| = r\}}} w_\infty^*\lambda.
\eeastar
On the other hand, since $\Image w_\infty|_{\del \H} \subset R$, a Legendrian submanifold,
$$
(w_\infty|_{\del \H})^*\lambda  = 0.
$$
Therefore we have shown that
$$
E^\pi(w_\infty) =  \lim_{r \to \infty} \int_{\{|z| = r, y \geq 0\}} w^*\lambda = \int \alpha^*\lambda > 0
$$
where $\alpha$ is the asymptotic self Reeb chord of $R$ given by
$$
\alpha(t): = \lim_{\tau \to \infty} w(e^{\pi(\tau + it)}), \quad x+i y = e^{\pi(\tau+it)} \in \H
$$
Hence we have produced a nontrivial bubble map $w_\infty: (\H, \del \H) \to (M, R)$ with
its asymptotic chord given by $\alpha$. 
An examination of the above proof in fact shows that the whole
argument can be repeated equally applies to the case of $K_\alpha \to \infty$
and $x_\alpha \to \infty$, except that $J_{K_0}$ is replaced by $J_\infty = J_0$
\emph{when there is no bubbling at $\tau = \pm \infty$.}

On the other hand, when bubbling occurs at
$\tau = \pm \infty$, the components $u_{-,0}$ and $u_{+,0}$ of the type described in
the statement of the theorem are added.

Hence we have established the two alternatives of the statement of the theorem.
\end{proof}

%Now we specialize to the case of our main interest with the standing hypothesis $\|H\| \leq T_\lambda(M,R)$
%and the uniform energy estimate
%$$
%E^\pi_{J,H}(u) \leq \|H\|.
%$$
%This then remove the possibility of any finite-time bubbling and hence $v_j$ or $w_k$ are
%not generated. The resulting outcome is the uniform $C^1$-bound 
%$$
%\|u_\alpha\|_{C^1} < C , \infty
%$$
%for any sequence $u_\alpha \in \CM^{\text{\rm para}}_{K_\alpha}(M,R;J,H)$. Therefore any such sequence 
%contains a subsequence converging to a stable map of the type
%$$
%u_{-,0} + u_0 + u_{+, 0}.
%$$
%This again contradicts to the standing hypothesis
%$\|H\| \leq T_\lambda(M,R)$ by the sequence of inequalities.
%$$
%0 < \int \alpha^*\lambda \leq E^\pi(\overine u_{-,0})) \leq \|H\| \leq T_\lambda(M,R).
%$$
%By applying the standard procedure based on Theorem \ref{thm:bubbling}, we have
%finally finished the proof of Theorem \ref{thm:bubbling-intro}.

\appendix

\section{Contact action functional and its first variation}

In this section, for readers' convenience,
 we provide the first variation formula for the contact action functional indicated in 
 Section \ref{sec:with-nondegeneracy} 
borrowing from \cite{oh:perturbed-contacton-bdy,oh:contacton-gluing,oh-yso:spectral}.

Next we introduce the crucial notions of perturbed action integrals of a path
and the energy relevant to the global study of perturbed contact instantons
\eqref{eq:contacton-Legendrian-bdy} introduced in \cite[Introduction]{oh:contacton-Legendrian-bdy}.

We start with the definition of perturbed action integrals associated to
contact Hamiltonian $H = H(t,x)$. We recall the standard contact action functional
$$
\CA(\gamma) = \int \gamma^*\lambda
$$
associated to the contact form $\lambda$ (associated to $H = 0$) in contact geometry.

\begin{defn}[Perturbed action functional] Let $H = H(t,x)$ be a contact Hamiltonian and recall
$\phi_H^t = \psi_H^t (\psi_H^1)^{-1}$.
We define a functional $\CA_H: \CL(R_0,R_1)  \to \R$ by
\be\label{eq:action}
\CA_H(\gamma) := \int_\gamma e^{g_{(\phi_H^t)^{-1}}(\gamma(t))} \gamma^*(\lambda + H\, dt)
 \left(=  \int_\gamma e^{g_{(\phi_H^t)^{-1}}(\gamma(t))} \gamma^*\lambda_H \right)
\ee
for any smooth path $\gamma:[0,1] \to M$. When $H = 0$, we write $\CA_0 = \CA$.
\end{defn}
 Then we have the following first variation formula.

\begin{prop}[Proposition 2.3 \cite{oh:perturbed-contacton-bdy}]
For any vector field $\eta$ along $\gamma$, we have
\bea\label{eq:1st-variation}
\delta \CA_H(\gamma)(\eta) & = & \int_0^1 e^{g_{(\phi_H^t)^{-1}}(\gamma(t))}
\left(d\lambda(\dot \gamma - X_H(t, \gamma(t)), \eta) \right)\, dt \nonumber\\
 & {}& + \lambda( \eta(1))
-  e^{-g_{\psi_H^1}((\psi_H^1)^{-1}(\gamma(0))} \lambda(\eta(0))).
\eea
\end{prop}

When we are given a pair $(R_0,R_1)$ of Legendrian submanifolds,
we can consider the path space
$$
\CL(R_0,R_1) = \CL(M;R_0,R_1): = \{ \gamma:[0,1] \to M
\mid \gamma(0) \in R_0, \, \gamma(1) \in R_1\}
$$
and the restriction of $\CA$ thereto.
An immediate corollary of this proposition is that the critical point equation of the action functional
\emph{under the Legendrian boundary condition} is precisely
$$
(\dot \gamma(t) - X_H(t,\gamma(t)))^\pi = 0,
$$
i.e., $\dot \gamma(t) - X_H(t,\gamma(t)) = a(t) R_\lambda(\gamma(t))$
for some function $a = a(t)$. (See Proposition 2.5 \cite{oh:perturbed-contacton-bdy} 
for some relevant discussion on the function $a$.)

Then  the following identity
\be\label{eq:energy-action}
E_{H,J} ^\pi(u) = \CA_H(u(+\infty) - \CA_H(u(-\infty))
\ee
is proved in \cite[Theorem 3.10]{oh:perturbed-contacton-bdy}
 for any finite energy perturbed contact instanton $u$.

\bibliographystyle{amsalpha}

\bibliography{biblio}

\newcommand{\etalchar}[1]{$^{#1}$}
\def\cprime{$'$}
\providecommand{\bysame}{\leavevmode\hbox to3em{\hrulefill}\thinspace}
\providecommand{\MR}{\relax\ifhmode\unskip\space\fi MR }
\providecommand{\MRhref}[2]{%
  \href{http://www.ams.org/mathscinet-getitem?mr=#1}{#2}
}
\providecommand{\href}[2]{#2}
\begin{thebibliography}{BLMN15}

\bibitem[Abb11]{abbas}
C.~Abbas, \emph{Holomorphic open book decompositions}, Duke Math. J.
  \textbf{158} (2011), 29--82.

\bibitem[ACH05]{abbas-cieliebak-hofer}
C.~Abbas, K.~Cieliebak, and H.~Hofer, \emph{The {W}einstein conjecture for
  planar contact structures in dimension three}, Comment. Math. Helv.
  \textbf{80} (2005), 771--793.

\bibitem[AG01]{arnold-givental}
V.~I. Arnold and A.~Givental, \emph{Symplectic {G}eometry}, Encyclopedia of
  Mathematical Sciences, vol.~IV, Springer, New York, 2001.

\bibitem[Ban78]{banyaga}
Augustin Banyaga, \emph{Sur la structure du groupe des diff\'eomorphismes qui
  pr\'eservent une forme symplectique}, Comment. Math. Helv. \textbf{53}
  (1978), no.~2, 174--227. \MR{490874 (80c:58005)}

\bibitem[BEH{\etalchar{+}}03]{behwz}
F.~Bourgeois, Y.~Eliashberg, H.~Hofer, K.~Wysocki, and E.~Zehnder,
  \emph{Compactness results in symplectic field theory}, Geom. Topol.
  \textbf{7} (2003), 799--888.

\bibitem[BKO23]{bko:wrapped}
Y.~Bae, S.~Kim, and Y.-G. Oh, \emph{A wrapped {F}ukaya category of knot
  complement}, Math. Z. \textbf{304} (2023), no.~2, Paper No. 29, 66 pp.

\bibitem[BLM15]{bravetti-lopez}
A.~Bravetti and C.S. Lopez-Monsalvo, \emph{Para-{S}asakian geometry in
  thermodynamic fluctuation theory}, Journal of Physics A: Math. Theor.
  \textbf{48} (2015), 125206 (21pp).

\bibitem[BLMN15]{bravetti-lopez-nettel}
A.~Bravetti, C.S. Lopez-Monsalvo, and F.~Nettel, \emph{Contact symmetries and
  {H}amiltonian thermodynamics}, Annals of Physics \textbf{361} (2015),
  377--400.

\bibitem[Che98]{chekanov:dmj}
Yu.~V. Chekanov, \emph{Lagrangian intersections, symplectic energy, and areas
  of holomorphic curves}, Duke Math. J. \textbf{95} (1998), 213--226.

\bibitem[dLLV19]{dMV}
Manuel de~Le\'on and Manuel Lainz~Valc\'azar, \emph{Contact {H}amiltonian
  systems}, J. Math. Phys. \textbf{60} (2019), no.~10, 102902, 18 pp.

\bibitem[DRS20]{rizell-sullivan}
Georios Dimitroglou-Rizell and Michael Sullivan, \emph{An energy-capacity
  inequality for {L}egendrian submanifolds}, J. Top. and Analysis \textbf{12}
  (2020), 547--623.

\bibitem[EHS95]{eliash-hofer-salamon}
Y.~Eliashberg, H.~Hofer, and D.~Salamon, \emph{Lagrangian intersections in
  contact geometry}, Geom. Funct. Anal. \textbf{5} (1995), no.~2, 244--269.

\bibitem[Ent01]{entov:K-area}
Michael Entov, \emph{K-area, {H}ofer metric and geometry of conjugacy classes
  in lie groups}, Invent. Math. \textbf{146} (2001), no.~1, 93–141.

\bibitem[EP21]{entov-polterov:thermodynamics}
M.~Entov and L.~Polterovich, \emph{Contact topology and non-equilibrim
  theromodynamics}, preprint, arXiv:2101.03770v2, 2021.

\bibitem[Flo89]{floer-fixedpoints}
Andreas Floer, \emph{Symplectic fixed points and holomorphic spheres}, Comm.
  Math. Phys. \textbf{120} (1989), no.~4, 575--611.

\bibitem[FO97]{fukaya-oh}
K.~Fukaya and Y.-G. Oh, \emph{Zero-loop open strings in the cotangent bundle
  and {M}orse homotopy}, Asian J. Math. \textbf{1} (1997), 96--180.

\bibitem[FOOO09]{fooo:book1}
K.~Fukaya, Y.-G. Oh, H.~Ohta, and K.~Ono, \emph{Lagrangian intersection {F}loer
  theory: anomaly and obstruction. part i.}, AMS/IP Studies in Advanced
  Mathematics, 46.1 \& 46.2., American Mathematical Society, Providence, RI;
  International Press, Somerville, MA, 2009.

\bibitem[FOOO13]{fooo:polydisks}
\bysame, \emph{Displacement of polydisks and {L}agrangian {F}loer theory}, J.
  Symplectic Geom. \textbf{11} (2013), no.~2, 231--268.

\bibitem[GPS20]{gps}
Sheel Ganatra, John Pardon, and Vivek Shende, \emph{Covariantly functorial
  wrapped {F}loer theory on {L}iouville sectors}, Publ. Math. Inst. Hautes
  \'Etud. Sci. \textbf{2020} (2020), 73--200.

\bibitem[GPS24]{gps-2}
\bysame, \emph{Sectorial descent for wrapped {F}ukaya categories}, J. Amer.
  Math. Soc. \textbf{37} (2024), no.~2, 499–635.

\bibitem[GT70]{GT}
D.~Gilbarg and N.~S. Trudinger, \emph{Elliptic partial differential equations
  of second order}, Comprehensive Studies in Math., vol. 224, Springer-Verlag,
  1970.

\bibitem[Hof93]{hofer:invent}
H.~Hofer, \emph{Pseudoholomorphic curves in symplectizations with applications
  to the {W}einstein conjecture in dimension three}, Invent. Math. \textbf{114}
  (1993), 515--563.

\bibitem[Hof00]{hofer:gafa}
\bysame, \emph{Holomorphic curves and real three-dimensional dynamics}, Geom.
  Func. Anal. (2000), 674--704, Special Volume, Part II.

\bibitem[HT13]{hutchings-taubes}
M.~Hutchings and C.~Taubes, \emph{Proof of the {A}rnold chord conjecture in
  three dimensions {II}}, Geom. Topol. \textbf{17} (2013), no.~5, 2601--2688.

\bibitem[HV92]{hofer-viterbo}
H.~Hofer and C.~Viterbo, \emph{The {W}einstein conjecture in the presence of
  holomorphic spheres}, Comm. Pure Appl. Math. \textbf{45} (1992), no.~5,
  583--622.

\bibitem[KO]{kim-oh:category}
Jongmyeong Kim and Y.-G. Oh, \emph{Rational contact instantons and {L}egendrian
  {F}ukaya category}, in preparation.

\bibitem[LO23]{lim-oh:thermodynamics}
Jin-wook Lim and Y.G. Oh, \emph{Nonequilibrium thermodynamics as a
  symplecto-contact reduction and relative information entropy}, Rep. Math.
  Phys. \textbf{92} (2023), no.~3, 347--400.

\bibitem[LOTV18]{LOTV}
Hong~Van Le, Yong-Geun Oh, Alfonso Tortorella, and Luca Vitagliano,
  \emph{Deformations of coisotropic submanifolds in {J}acobi manifolds}, J.
  Symplectic Geom. \textbf{16} (2018), no.~4, 1051--1116.

\bibitem[Mil00]{milinkovic-pacific}
D.~Milinkovi\^c, \emph{On equivalence of two constructions of invariants of
  lagrangian submanifolds}, Pacific J. Math. \textbf{195} (2000), no.~2,
  371--415.

\bibitem[MNSS90]{MNSaSc}
R.~Mrugala, J.~D. Nulton, J.~C Sch{\"o}en, and P.~Salamon, \emph{Statistical
  approach to the geometric structure of theomodynamics}, Physical Review A
  \textbf{41} (1990), no.~6, 3156 (6pp).

\bibitem[MS15]{mueller-spaeth-I}
S.~M\"uller and P.~Spaeth, \emph{Topological contact dynamics {I}:
  symplectization and applications of the energy-capacity inequality}, Adv.
  Geom. \textbf{15} (2015), no.~3, 349--380.

\bibitem[Oha]{oh:perturbed-contacton-bdy}
Y.-G. Oh, \emph{Geometric analysis of perturbed contact instantons with
  {L}egendrian boundary condition}, preprint, arXiv:2205.12351.

\bibitem[Ohb]{oh:entanglement2}
\bysame, \emph{Geometry and analysis of contact instantons and entangement of
  {L}egendrian links {II}}, in preparation.

\bibitem[Ohc]{oh:contacton-gluing}
\bysame, \emph{Gluing theories of contact instantons and of pseudoholomorphic
  curves in {SFT}}, preprint, arXiv:2111.02597.

\bibitem[Ohd]{oh:intrinsic}
\bysame, \emph{Monoid of {L}iouville sectors with corners and its intrinsic
  characterization}, Asian J. Math., to appear, arXiv:2110.11726.

\bibitem[Oh95]{oh:Riemann-Hilbert}
\bysame, \emph{Riemann-{H}ilbert problem and application to the perturbation
  theory of analytic discs}, Kyungpook Math. J. \textbf{35} (1995), no.~1,
  39--75.

\bibitem[Oh96]{oh:imrn}
\bysame, \emph{Floer cohomology, spectral sequence and the {M}aslov class of
  {L}agrangian embeddings}, Internat. Math. Res. Notices \textbf{7} (1996),
  305--346.

\bibitem[Oh97a]{oh:mrl}
\bysame, \emph{Gromov-{F}loer theory and disjunction energy of compact
  {L}agrangian embeddings}, Math. Res. Lett. \textbf{4} (1997), no.~6,
  895--905.

\bibitem[Oh97b]{oh:jdg}
\bysame, \emph{Symplectic topology as the geometry of action functional. {I}.
  relative floer theory on the cotangent bundle}, J. Differential Geom.
  \textbf{46} (1997), no.~3, 499--577.

\bibitem[Oh99]{oh:cag}
\bysame, \emph{Symplectic topology as the geometry of action functional {II}:
  Pants product and cohomological invariants}, Comm. Anal. Geom. \textbf{7}
  (1999), no.~1, 1--54.

\bibitem[Oh05a]{oh:alan}
\bysame, \emph{Construction of spectral invariants of {H}amiltonian paths on
  closed symplectic manifolds}, The breadth of symplectic and Poisson geometry,
  Progr. Math., vol. 232, Birkhäuser, Boston, MA, 2005, pp.~525--570.

\bibitem[Oh05b]{oh:dmj}
\bysame, \emph{Spectral invariants, analysis of the {F}loer moduli space, and
  geometry of the {H}amiltonian diffeomorphism group}, Duke Math. J.
  \textbf{130} (2005), no.~2, 199--295.

\bibitem[Oh15a]{oh:book1}
\bysame, \emph{Symplectic {T}opology and {F}loer {H}omology. vol. 1.}, New
  Mathematical Monographs, 28., Cambridge University Press, Cambridge., 2015.

\bibitem[Oh15b]{oh:book2}
\bysame, \emph{Symplectic {T}opology and {F}loer {H}omology. vol. 2.}, New
  Mathematical Monographs, 29., Cambridge University Press, Cambridge., 2015.

\bibitem[Oh21]{oh:contacton-Legendrian-bdy}
\bysame, \emph{Contact {H}amiltonian dynamics and perturbed contact instantons
  with {L}egendrian boundary condition}, preprint, arXiv:2103.15390(v2), 2021.

\bibitem[Oh22]{oh:shelukhin-conjecture}
\bysame, \emph{Contact instantons, anti-contact involution and proof of
  {S}helukhin's conjecture}, preprint, arXiv:2212.03557, 2022.

\bibitem[Oh23]{oh:contacton}
\bysame, \emph{Analysis of contact {C}auchy-{R}iemann maps {III}: energy,
  bubbling and {F}redholm theory}, Bulletin of Math. Sci. \textbf{13} (2023),
  no.~1, Paper No. 2250011, 61 pp.

\bibitem[OK]{oh-kim:survey}
Y.-G. Oh and Taesu Kim, \emph{Analysis of pseudoholomorphic curves on
  symplectization: Revisit via contact instantons}, 2021-22 MATRIX Annals,
  MATRIX Book Series, Springer, (a contribution for the IBS-CGP and MATRIX
  workshop on Symplectic Topology, December, 2022).

\bibitem[OP05]{oh-park}
Y.-G. Oh and J.-S. Park, \emph{Deformations of coisotropic submanifolds and
  strong homotopy {L}ie algebroids}, Invent. Math. \textbf{161} (2005), no.~2,
  287--360. \MR{2180451 (2006g:53152)}

\bibitem[OS]{oh-savelyev}
Y.-G. Oh and Y.~Savelyev, \emph{Pseudoholomoprhic curves on the
  $\mathfrak{LCS}$-fication of contact manifolds}, Advances in Geometry
  \textbf{23}, no.~2, 153--190.

\bibitem[OW14]{oh-wang:connection}
Y.-G. Oh and R.~Wang, \emph{Canonical connection on contact manifolds}, Real
  and Complex Submanifolds, Springer Proceedings in Mathematics \& Statistics,
  vol. 106, 2014, (arXiv:1212.4817 in its full version), pp.~43--63.

\bibitem[OW18a]{oh-wang:CR-map1}
\bysame, \emph{Analysis of contact {C}auchy-{R}iemann maps {I}: A priori
  {$C^k$} estimates and asymptotic convergence}, Osaka J. Math. \textbf{55}
  (2018), no.~4, 647--679.

\bibitem[OW18b]{oh-wang:CR-map2}
\bysame, \emph{Analysis of contact {C}auchy-{Ri}emann maps {II}: Canonical
  neighborhoods and exponential convergence for the {M}orse-{B}ott case},
  Nagoya Math. J. \textbf{231} (2018), 128--223.

\bibitem[OY]{oh-yso:spectral}
Y.-G. Oh and Seungook Yu, \emph{Legendrian contact instanton cohomology and its
  spectral invariants on the one-jet bundle}, preprint, 2023, arXiv:2301.06704.

\bibitem[OY24]{oh-yso:index}
\bysame, \emph{Contact instantons wit {L}egendrian boundary condition: a priori
  estimates, asymptotic convergence and index formula}, Internat. J. Math.
  \textbf{35} (2024), no.~7, Paper No. 2450019.

\bibitem[San12]{sandon-translated}
S.~Sandon, \emph{On iterated translated points for contactomorphisms of
  {$\mathbb{R}^{2n+1}$} and {$\mathbb{R}^{2n}\times {S^1}$}}, Internat. J.
  Math. \textbf{23} (2012), no.~2, 1250042, 14 pp.

\bibitem[Sei97]{seidel:pi1}
P.~Seidel, \emph{{$\pi_1$} of symplectic automorphism groups and invertibles in
  quantum homology rings}, Geom. Funct. Anal. \textbf{7} (1997), no.~6,
  1046--1095. \MR{1487754 (99b:57068)}

\bibitem[She17]{shelukhin:contactomorphism}
E.~Shelukhin, \emph{The {H}ofer norm of a contactomorphism}, J. Symplectic
  Geom. \textbf{15} (2017), no.~4, 1173--1208.

\bibitem[Tau07]{taubes}
C.~Taubes, \emph{The {S}eiberg-{W}itten equations and the {W}einstein
  conjecture}, Geom. Topo. \textbf{11} (2007), 2117--2202.

\bibitem[Zam08]{zambon}
M.~Zambon, \emph{An example of coisotropic submanifolds {$C^1$}-close to a
  given coisotropic submanifold}, Differential Geom. Appl. \textbf{26} (2008),
  no.~6, 635--637.

\end{thebibliography}

\end{document}